\documentclass[11pt]{article}
\usepackage[latin1]{inputenc}
\usepackage[normalem]{ulem}
\usepackage{amsthm, multirow, youngtab}
\usepackage{amsmath, framed, mathrsfs}
\usepackage{amssymb, pst-gantt}
\usepackage{graphicx,stmaryrd, soul}
 \DeclareFontFamily{U}{wncy}{}
    \DeclareFontShape{U}{wncy}{m}{n}{<->wncyr10}{}
    \DeclareSymbolFont{mcy}{U}{wncy}{m}{n}      
    \DeclareMathSymbol{\Sh}{\mathord}{mcy}{"58} 
\usepackage[english]{babel}
\usepackage[curve,matrix,arrow,cmtip]{xy}
\NoComputerModernTips

\newdir^{ (}{{}*!/-3pt/\dir^{(}}    
\newdir_{ (}{{}*!/-3pt/\dir_{(}}    

\usepackage[scale=0.7,hmarginratio=1:1,vmarginratio=2:3]{geometry}

\newcounter{themargin}
\long\def\forget#1{}

\theoremstyle{plain} 
\newtheorem{Thm}{Theorem}[section]

\newtheorem{Prop}[Thm]{Proposition}

\newtheorem{Lem}[Thm]{Lemma}

\newtheorem{Cor}[Thm]{Corollary}

\newtheorem{thmx}{Lemma}

\theoremstyle{definition} 
\newtheorem{Def}[Thm]{Definition}

\newtheorem{Ex}[Thm]{Example}

\theoremstyle{remark} 
\newtheorem{Rem}[Thm]{Remark}

\newtheorem{Ass}[Thm]{Assumption}

\forget{\newif\ifnormalesBeweisEnde

  {\vskip 0.3ex plus 0.5ex minus 0ex \pagebreak[1]
   \global\normalesBeweisEndetrue
   \trivlist
   \item[\hskip\labelsep {\textsc{Proof} \rm #1}:]}%
  {\ifnormalesBeweisEnde \EndOfProof \fi
   \endtrivlist
\begin{flushright}
   •
   \end{flushright}   \vskip 1ex plus 1ex minus 0ex \pagebreak[2]}
  {\vskip 0.3ex plus 0.5ex minus 0ex \pagebreak[1]
   \global\normalesBeweisEndetrue
   \trivlist
   \item[\hskip\labelsep \textsc{Proof}:]}%
  {\ifnormalesBeweisEnde \EndOfProof \fi
   \endtrivlist
   \vskip 1ex plus 1ex minus 0ex \pagebreak[2]}
\def\EndOfProof{\hskip .5em \vrule width 1.0ex height 1.0ex depth 0.3ex}
}

\usepackage{latexsym}
\usepackage{amsmath}
\usepackage{amstext}
\usepackage{amssymb}
\usepackage{amsbsy}

\newbox\dottobox
\setbox\dottobox=\hbox{$\longrightarrow$}
\setbox\dottobox=\hbox to\wd\dottobox{\hss$
\UseComputerModernTips\xymatrix@C=.5cm{\ar@{.>}[r]&\\}
                                      $\hss}

\newbox\leftdottobox
\setbox\leftdottobox=\hbox{$\longleftarrow$}
\setbox\leftdottobox=\hbox to\wd\leftdottobox{\hss$
\UseComputerModernTips\xymatrix@C=.5cm{\ar@{<.}[r]&\\}
                                      $\hss}

\newbox\dotintobox
\setbox\dotintobox=\hbox{$\longrightarrow$}
\setbox\dotintobox=\hbox to\wd\dotintobox{\hss$
\UseComputerModernTips\xymatrix@C=.5cm{\ar@{^{ (}.>}[r]&\\}
                                      $\hss}

\newcommand{\Tempnewpage}
{\newpage}

\title{Unobstructedness of Galois deformation rings associated to RACSDC automorphic representations}
\author{David-A. Guiraud  \\
	IWR Heidelberg  
	}

\date{February 28, 2017}
\begin{document}

\maketitle

\begin{abstract}
Let $F$ be a CM field and let $(\overline{r}_{\pi,\lambda})_{\lambda}$ be the 
compatible system of residual $\mathcal{G}_n$-valued representations of $\operatorname{Gal}_{F}$ attached to a RACSDC automorphic representation $\pi$ of $\operatorname{GL}_n(\mathbb{A})$, as studied by Clozel, Harris and Taylor \cite{CHT} and others. Under mild assumptions, we prove that 
the fixed-determinant universal deformation rings attached to $\overline{r}_{\pi,\lambda}$ are
unobstructed for all places $\lambda$ in a subset of Dirichlet density $1$, continuing the
investigations of Mazur, Weston and Gamzon. During the proof, we develop a general framework for 
proving unobstructedness (which could be useful for other applications in future) and 
an $R=T$-theorem, relating the universal crystalline deformation ring of $\overline{r}_{\pi,\lambda}$ and a certain unitary fixed-type 
Hecke algebra.
\end{abstract}

\section{Introduction}
This article is concerned with unobstructedness of 
certain Galois deformation rings. For the purpose of this 
introduction, let $F$ be a number field, let $k$ be a finite field of characteristic $\ell$ and fix an 
absolutely irreducible representation
\begin{equation*}
\overline{\rho}: \operatorname{Gal}_{F, S} \rightarrow \operatorname{GL}_n(k),
\end{equation*}
where $S\subset \operatorname{Pl}_F$ is a finite set of places.
Then assigning to a complete Noetherian local algebra $A$ over the 
ring $W$ of Witt vectors of $k$ the set of all $\operatorname{GL}_n(A)$-valued deformations of $\overline{\rho}$ 
defines a functor which is representable by a universal deformation 
ring $R_S(\overline{\rho})$, studied first by 
Mazur \cite{Mazur_deforming}. It is easily seen that a vanishing of the cohomology group 
$H^2(\operatorname{Gal}_{F,S}, \operatorname{ad}\overline{\rho})$ implies that 
$R_S(\overline{\rho})$ is formally smooth, i.e. isomorphic to a power
series ring over $W$. 
In this sense, the group $H^2(\operatorname{Gal}_{F,S}, \operatorname{ad}\overline{\rho})$ can
be interpreted as the obstruction to the smoothness of 
$R_S(\overline{\rho})$, and we say that $R_S(\overline{\rho})$ is 
\textit{unobstructed} if 
$H^2(\operatorname{Gal}_{F,S}, \operatorname{ad}\overline{\rho})=0$.
We remark the following connection with a conjecture of Jannsen:
Assume that $\overline{\rho}$ is the reduction of the $\ell$-adic
representation $\rho_{f,\ell}$ attached to a cuspidal modular eigenform 
$f$ (cf. \cite{Deligne,Shimura,Deligne_Serre}).
Then the Frobenius eigenvalues of $\rho_{f,\ell}$ are Weil-numbers of 
some fixed weight $w$, i.e. $\rho_{f,\ell}$ is pure of 
weight $w$. A conjecture of Jannsen \cite[Conjecture 1]{Jannsen}
(see also \cite[Conjecture 5.1]{Bellaiche}) predicts the vanishing 
of $H^2(\operatorname{Gal}_{F,S}, \operatorname{ad}{\rho})$.
This implies that $H^2(\operatorname{Gal}_{F,S}, \Lambda)$ is 
finite and torsion, where $\Lambda\subset \operatorname{ad}\rho$
denotes an integral $\operatorname{Gal}_{F,S}$-stable
lattice. On the other hand, our residual $H^2$-vanishing 
implies the vanishing of $H^2(\operatorname{Gal}_{F,S}, \Lambda)$ 
by Nakayama's Lemma. This, in turn, implies the vanishing 
of $H^2(\operatorname{Gal}_{F,S}, \operatorname{ad}{\rho})$, as 
predicted by Jannsen. Besides this application, the numerous usages of Galois deformation
theoretic methods in number theory indicate that the structure of  
universal deformation rings is of independent interest.

Unobstructedness for Galois representations attached to automorphic
objects can usually not be expected for all choices of 
$\ell$. The best we can hope for is that unobstructedness holds for
almost all primes (or: for all primes in a subset of Dirichlet density
$1$), and this question has been studied (under different technical 
assumptions) in the following cases:
\begin{itemize}
\item For $\overline{\rho}$ the reduction of 
the representation $\rho_{E, \ell}$ attached to an elliptic 
curve $E$ over $F= \mathbb{Q}$, cf. \cite{Mazur1};
\item for $\overline{\rho}$ the reduction of 
the representation $\rho_{f, \ell}$ attached to
a newform $f$ of weight $k\geq 3$ over $F= \mathbb{Q}$, cf. \cite{Weston}
(but see also \cite{Yamagami,Hatley});
\item for $\overline{\rho}$ the reduction of 
the representation $\rho_{f, \ell}$ attached to  
a Hilbert eigenform $f$ over a totally real field $F$, cf. \cite{Gamzon}.
\end{itemize}
Note that $n=2$ in all these cases. 

In this article, we develop a general framework for proving 
unobstructedness, which differs significantly from the previous 
approaches and which uses an $R=T$-theorem as the main ingredient.
We apply this framework to the reduction of the Galois representation
attached to a regular algebraic conjugate self-dual\footnote{We remark that, in light of 
the results of \cite{BLGGT}, it should be possible to weaken the conjugate self-duality assumption to an essentially self-duality assumption, thus treating RACESDC automorhphic representations.}
cuspidal (RACSDC) automorphic representation $\pi$ of $\operatorname{GL}_n(\mathbb{A}_F)$ with
ramification set $S$, where $F$ is a CM field. 
In order to give a more precise statement, we have to recall 
that $\pi$ gives in the first instance rise not to 
$\operatorname{GL}_n$-valued representations, but to morphisms 
${r}_{\pi, \lambda}: \operatorname{Gal}_{F^+}
\rightarrow \mathcal{G}_n(\overline{\mathbb{Q}}_{\ell(\lambda)})$, where
$\lambda$ runs through the places of the coefficient field of $\pi$, where $\mathcal{G}_n$ denotes the group scheme from 
\cite[Section 2.1]{CHT} and where $\ell(\lambda)$ denotes the rational prime below $\lambda$.
We make the following assumption:
\begin{Ass}
The set of the $\lambda$ for which the $\operatorname{GL}_n$-valued representation $\overline{r}_{\pi, \lambda}|\operatorname{Gal}_F$ 
is absolutely irreducible has Dirichlet density $1$.
\end{Ass}
We remark that this assumption is fulfilled e.g. if $n\leq 5$ or if $\pi$ is extremely 
regular, or would follow from absolute irreducibility of the $\ell$-adic system 
$(r_{\pi, \lambda}|\operatorname{Gal}_F)$, cf. Remark \ref{20150407_irreducibility_remark}. For the following, 
we fix for each $\lambda$ a lift $\chi$ of the character $\texttt{m}\circ \overline{r}_{\pi, \lambda}$ of $\operatorname{Gal}_F$, where $\texttt{m}$ is the multiplier character of 
the group $\mathcal{G}_n$, cf. Section \ref{20170407_section_grp_gn}. 
By $R^{\chi}_{S_{\ell}}(\overline{r}_{\pi, \lambda})$ we denote the 
universal ring parametrizing deformations $r$ of $\overline{r}_{\pi, \lambda}$
which are unramified outside the places 
which are in $S$ or divide $\infty. \ell(\lambda)$
and which fulfill $\texttt{m}\circ r = \chi$. 
The correct unobstructedness requirement is then the vanishing of
$H^2(\operatorname{Gal}_{F, S_{\ell}}, \mathfrak{g}^{\tt{der}}_n)$, where 
$\mathfrak{g}^{\tt{der}}_n$ denotes the Lie algebra of the derived subgroup of 
$\mathcal{G}_n$.
Our main result is:
\begin{Thm}\label{20170407_main_thm_intro}
Assume that all Hodge-Tate weights of $r_{\pi, \lambda}$ (which are independent
of $\lambda$, as the $r_{\pi, \lambda}$ form a compatible system) 
are non-consecutive: if $a,b\in \mathbb{Z}$
show up as Hodge-Tate weights, then $|a-b|\neq 1$. Then, for all 
$\lambda$ in a set of places of Dirichlet density $1$ the universal deformation ring $R^{\chi}_{S_{\ell}}(\overline{r}_{\pi, \lambda})$
is unobstructed.
\end{Thm}
Remark that we do not require a particular splitting behavior at the
places in $S$.
We also want to stress that the developed framework is flexible and in principle 
applicable to Galois representations with values in other groups and can be
used to establish unobstructedness of universal deformation rings with imposed 
deformation conditions which are more sophisticated than the fixed-determinant 
condition $\texttt{m}\circ r = \chi$. Therefore, we hope that the framework 
will be useful for other applications, as better modularity lifting
results become available in future. We also remark that presently the condition on the Hodge-Tate
weights is necessary in order to use a local unobstructedness property at the 
places above $\ell(\lambda)$, a technical inconvenience we expect to  
weaken in future work.

We give a short outline of the article: After some remarks about notation, we start 
in Section \ref{20170405_section_liftings_and_defos} with a collection of the 
general deformation theoretic methods we will use. Moreover, we will
define a suitably flexible notion of unobstructedness for conditioned deformation
functors (Definition \ref{20170407_def_glob_unobstr}).
In Section \ref{20170405_section_on_framework}, we state and prove the core 
framework (Theorem \ref{20160908_MainThm}), which uses a list of six assumptions
as input and provides unobstructedness as output. This framework is presented 
with respect to local deformation conditions $\texttt{crys, min, sm}$, which have
a purely formal meaning throughout Section \ref{20170405_section_on_framework}. 
The main input is the formal smoothness of the deformation ring with respect 
to the conditions $\texttt{min}$ and $\texttt{crys}$, which is the natural
output of a suitable $R=T$-theorem, and the desired unobstructedness is then deduced 
by commutative algebra arguments and comparing dimensions. 
It is the purpose of Section \ref{20170316_sect_defo_conds} to introduce and study useful 
local conditions which will be plugged in into the framework theorem later.
After a reminder on the association of Galois deformations to automorphic forms, 
the additional results are provided in Section 
\ref{20160407_sect_cons}: We consider the deformation ring 
$R^{\tt{min, crys}}:= R^{\chi,\tt{min,crys}}_{S_{\ell}}(\overline{r}_{\pi, \lambda})$ parametrizing
those lifts which are minimally ramified (in the sense of 
Section \ref{20170321_sect_min_ram}) at all places in $S$ and 
crystalline (in the Fontaine-Laffaille range) at all places dividing $\ell$. Morover,
we consider a corresponding Hecke algebra 
$\mathbb{T}^{\tt{min}}$ which is defined as the localization of a certain endomorphism algebra
of automorphic forms of the same weight and level as $\pi$, and with a certain
fixed type-requirement at the places in $S$. Then, using the modularity lifting
results of \cite{BLGGT}, we show
\begin{Thm}
$R^{\tt{min, crys}}\cong \mathbb{T}^{\tt{min}}$ and, for almost all $\lambda$, 
$\mathbb{T}^{\tt{min}} \cong W$.
\end{Thm}
This result is crucial to prove in Section \ref{20170407_main_sect} that, 
for almost all $\lambda$, there exists a suitable finite solvable extension $F'$ 
of $F$ such that the deformation ring $R^{\chi, \tt{min}}_{S_{\ell}}(\overline{r}_{\pi, \lambda}|\operatorname{Gal}_{F', S})$, parametrizing deformations 
of the base change of 
$\overline{r}_{\pi, \lambda}$ to $F'$ which are minimally ramified at all places above $S$, 
is unobstructed.
Moreover, we show that the minimally-ramified condition can be waived
for almost all $\lambda$ (Theorem \ref{20170213_global_rrmin_thm}).
It is important to keep track of the different field extensions necessary
when running through all $\lambda$, so that we are left with a set of Dirichlet density $1$
to which we can apply a result on potential unobstructedness (Lemma 
\ref{20170404_pot_unobstr}) and finally deduce the claim of Theorem \ref{20170407_main_thm_intro}.

\paragraph{Acknowledgements} The main part of this article grew out of my 
PhD thesis \cite{MyPhd} at the University of Heidelberg. I would like to thank my
doctoral advisor Gebhard B\"ockle for his support and help, throughout my PhD studies
and thereafter. Moreover, I would like to thank the Graduiertenkolleg Heidelberg (Stipendium
nach dem Landesgraduiertenf\"ordergesetz), MATCH (The Mathematics Center Heidelberg),
the DAAD (Promos Stipendium) and the DFG (FG1920) for the funding received during my
PhD studies and thereafter.

\section{Notation}
Before we start with the main body of this article, let us make some remarks on the used notation: If $F$ denotes a number field, we denote by $\operatorname{Pl}_F$ the set 
of places of $F$ and by $\operatorname{Pl}_F^{\text{fin}}$ the set 
of finite places of $F$. Moreover, we set $\Omega^F_{\infty} = \operatorname{Pl}_F\backslash
\operatorname{Pl}_F^{\text{fin}}$ and, for a rational prime $\ell$, we denote by $\Omega^F_{\ell}$ the set of places of $F$ dividing $\ell$. If $F$ is understood, we will
simply write $\Omega_{\infty}$ and $\Omega_{\ell}$. For a place $\lambda\in 
\operatorname{Pl}_F^{\text{fin}}$ we define $\ell(\lambda)$ (or $\ell$, if $\lambda$ is understood) as the rational prime below $\lambda$. If $S\subset \operatorname{Pl}_F^{\text{fin}}$ and $\ell$ is some rational prime, we set $S_{\ell} := S \cup \Omega_{\infty} \cup \Omega_{\ell}$.\\
We denote by $\widetilde{F}$ the Galois closure of $F$.
When dealing with a quadratic extension $F|F^+$, we will denote by $c$ the non-trivial
element of the Galois group $\operatorname{Gal}(F|F^+)$.
Moreover, for a rational prime $\ell$, we denote by $\epsilon_{\ell}:
\operatorname{Gal}_F\rightarrow \overline{\mathbb{Z}}_{\ell}^{\times}$ the $\ell$-adic
cyclotomic character and by $\overline{\epsilon}_{\ell}$ its mod-$\ell$ reduction.
\\
If $L|F$ is a finite extension and $S$ is a fixed set of places of $F$, we will write $S\langle L\rangle$ for the set $\{\nu' \in \operatorname{Pl}_L: \nu' \text{ divides some }\nu\in S\}$. 
For example, if 
$\operatorname{Gal}_{F, S}$ denotes the Galois group of the maximal, unramified outside $S$ extension of $F$, we will write $\operatorname{Gal}_{L, S\langle L\rangle}$ for the Galois group of the maximal, unramified outside $S\langle L\rangle$ extension of $L$.
However, if 
there is no risk of confusion, we will often simplify this and write $S$ in place of $S\langle L\rangle$
and hence $\operatorname{Gal}_{L, S}$ instead of $\operatorname{Gal}_{L, S\langle L\rangle}$.
In a completely analogous way, if $S$ is a subset of $\operatorname{Pl}_L$, we will
write $S\langle F\rangle$ (or, if there is no risk of confusion, simply $S$) for the set
$\{\nu' \in \operatorname{Pl}_F: \nu' \text{ is divided by some }\nu\in S\}$. 
%
%
%
If $\rho$ is a representation of $\operatorname{Gal}_F$ and 
$\nu$ a place of $F$, we will use the symbol $\rho_{\nu}$ for the restriction of 
$\rho$ to a decomposition subgroup at $\nu$.\\
For a topological group $\Gamma$ and a topological ring $R$, we denote by $\operatorname{Rep}_R(\Gamma)$
the category of finitely generated $R$-modules with a continuous $\Gamma$-action.
If $A$ is a $\Gamma$-module, we denote by $A^\ast$ the Pontryagin dual and by $A^\vee$ the Tate dual of $A$.\\
We will often make statements concerning variations of deformation rings and we will
shorten this using brackets. E.g., we will use the notation $R^{(\chi), [\tt{min}]}(\overline{\rho})=0$ as a shortcut for the four statements 
$R(\overline{\rho})=0, R^{\chi}(\overline{\rho})=0, R^{\tt{min}}(\overline{\rho})=0$
and $R^{\chi, \tt{min}}(\overline{\rho})=0$.
For cohomology groups, we abbreviate $h^i(\ast, \ast)$ for 
$\operatorname{dim} H^i(\ast, \ast)$.\\
Let $k$ be a finite field of characteristic $\ell$.
For the valuation ring $\Lambda$ of a finite extension of $\mathbb{Q}_{\ell}$ 
with residue field $k_{\Lambda} = k$, we will consider the category
$\mathcal{C}_{\Lambda}$ of complete Noetherian local $\Lambda$-algebras $A$ fulfilling
$k_{A}= k$.
\section{Liftings and deformations}\label{20170405_section_liftings_and_defos}
In this section, which contains nothing original, we recall the main results on deformation theory.
For general background literature, we refer the reader to \cite{Tilouine,MaugerThesis,Levin,Balaji,Bleher_Chinburg}. 
Let us first fix a finite field $k$ and denote $\ell= \operatorname{char}(k)$. We will denote
the ring of Witt vectors over $k$ by $W(k)$, or, if $k$ is understood, by $W$. Moreover,
let us fix a profinite group $\Gamma$ which fulfills the $\ell$-finiteness condition $(\Phi_{\ell})$
of \cite{Mazur_deforming}: For any open subgroup $H\subset \Gamma$, the maximal
pro-$\ell$ quotient of $H$ is topologically finitely generated.

Let $G$ be a smooth linear algebraic group over $W$ and fix a continuous group homomorphism
$\overline{\rho}:\Gamma \rightarrow G(k)$, where $G(k)$ carries the discrete topology.
\paragraph{Basic facts on coefficient rings} Let us first state some basic facts on the category $\mathcal{C}_{\Lambda}$, whose proofs we leave to the reader: The pullback in $\mathcal{C}_{\Lambda}$ is realized by the completed tensor product
$\widehat{\otimes}$, cf. \cite[\S 12]{Mazur1}. Consequently, if $C\leftarrow A \rightarrow B$
is a diagram in $\mathcal{C}_{\Lambda}$, then $\operatorname{Hom}_{\mathcal{C}_{\Lambda}}(
B\widehat{\otimes}_A C, \underline{\;\;})$ is the pullback of the diagram 
of functors $\operatorname{Hom}_{\mathcal{C}_{\Lambda}}(C, \underline{\;\;})\rightarrow 
\operatorname{Hom}_{\mathcal{C}_{\Lambda}}(A, \underline{\;\;})\leftarrow
\operatorname{Hom}_{\mathcal{C}_{\Lambda}}(B, \underline{\;\;})$.\\
Consider a pushout diagram in $\mathcal{C}_{\Lambda}$ where one arrow (say, $f$)
is surjective. This implies that the parallel arrow (say, $g$) is surjective as well, so
taking $I=\operatorname{ker}(f)$ and $J=\operatorname{ker}(g)$ we can extend the orthogonal
arrow (say, $\pi$) to a map of short exact sequences of $\Lambda$-modules:
\begin{equation*}
\xymatrix{
A\ar@{->>}^f[r]\ar_{\pi}[d] & B\ar@{.>}[d] &\ar@{}[d]^{\leadsto}&&0\ar[r]&I\ar[r]\ar_{\pi|I}[d]&A\ar^f[r]\ar_{\pi}[d]&B\ar[r]\ar[d]&0\\
C\ar@{.>}_g[r] & P &&&0\ar[r]&J\ar[r]&C\ar_g[r]&P\ar[r]&0
}
\end{equation*}
If $\mathfrak{I}$ is a finitely generated ideal of some $D\in \mathcal{C}_{\Lambda}$ we
denote cardinality of a minimal set of generators of $\mathfrak{I}$ by $\operatorname{gen}_D(\mathfrak{I}) := \operatorname{dim}_k\mathfrak{I}/\mathfrak{m}_D\mathfrak{I}$. Then, we easily 
see that the following holds for the above diagram:
\begin{Prop}
$\operatorname{gen}_C(J)\leq \operatorname{gen}_A(I)$.
\end{Prop}
Moreover, we have the following results, which follow easily from standard facts about regular systems of parameters (cf. \cite[Proposition 22]{SerreLocalAlg} and its use in Section 2 of 
\cite{MyPhd}) :
\begin{Lem}
Suppose $A=\Lambda\llbracket x_1, \ldots, x_a\rrbracket, B=\Lambda\llbracket x_1, \ldots, x_b\rrbracket\in 
\mathcal{C}_{\Lambda}$ and let $J\subset A$ be an ideal of the form $J=(f_1, \ldots, 
f_u)$ with $f_i \in A$ and $u\leq a$. Suppose moreover that there exists a surjective morphism $f: A/J \twoheadrightarrow B$ and denote its kernel by $I$. Then
the following are equivalent:
\begin{itemize}
\item $A/J \cong \Lambda \llbracket x_1, \ldots, x_{a-u}\rrbracket$;
\item $\operatorname{gen}_{A/J}(I) 
= a-u-b$;
\item $\operatorname{gen}_{A/J}(I) 
\leq a-u-b $.
\end{itemize}
\end{Lem}
\begin{Prop}\label{20170324_some_prop}
Let $m\in \mathbb{N}$. Then $A\in \mathcal{C}_{\Lambda}$ is regular if and only if 
$A\llbracket x_1, \ldots, x_m\rrbracket$ is regular.
\end{Prop}
\begin{Prop}
Let $f:A\rightarrow B$ be a morphism in $\mathcal{C}_{\Lambda}$. 
\begin{enumerate}
\item $f$ is formally smooth
(cf. \cite[\S 19]{Grothendieck_ega}) if and only if $B$ is isomorphic to a formal power series
ring over $A$;
\item Assume that $A$ is formally smooth over $\Lambda$ of relative dimension $d$ and that $f$ is
surjective. Then $f$ is an isomorphism if $\operatorname{rdim}_{\Lambda}(B) = d+1$ (where
$\operatorname{rdim}_{\Lambda}(B) = \operatorname{dim}_k \mathfrak{m}_B/(\mathfrak{m}_B^2, 
\mathfrak{m}_{\Lambda})$ denotes the relative dimension over $\Lambda$).
\end{enumerate}
\end{Prop}
\begin{Lem}
Let $A\in \mathcal{C}_{\Lambda}, m\in\mathbb{N}$ such that 
$\Lambda\llbracket x_1, \ldots, x_m\rrbracket \cong A\,\widehat{\otimes}_{\Lambda}\, 
\Lambda\llbracket x\rrbracket$. Then $A \cong \Lambda\llbracket x_1, \ldots, x_{m-1}\rrbracket$.
\end{Lem}
\begin{Lem}\label{20170405_lemma_res_of_coeff_ring}
Let $\Lambda', R\in \mathcal{C}_{\Lambda}$, such that the structure
morphism $\Lambda\rightarrow \Lambda'$ is flat. Then $R$ is formally
smooth of relative dimension $d\in \mathbb{N}$ over $\Lambda$ if and only if 
$\Lambda' \widehat{\otimes}_{\Lambda} R$ is formally smooth of relative dimension
$d$ over $\Lambda'$.
\end{Lem}
\paragraph{Liftings and deformations of $G$-valued representations}
\begin{Def}\label{20170406_def_lifting}
\begin{enumerate}
\item A lifting of $\overline{\rho}$ to an $A\in \mathcal{C}_{\Lambda}$ is a continuous group
homomorphism $\rho: \Gamma \rightarrow G(A)$ fulfilling $\operatorname{mod}_{\mathfrak{m}_A}\circ\rho = \overline{\rho}$, where $\operatorname{mod}_{\mathfrak{m}_A}: G(A)
\rightarrow G(A/\mathfrak{m}_A) = G(k)$ is the canonical reduction;
\item Denote by $D_{\Lambda}^{\boxempty}(\overline{\rho}): \mathcal{C}_{\Lambda} \rightarrow
\operatorname{Sets}$ the functor which assigns to an object $A\in \mathcal{C}_{\Lambda}$ the
set of all liftings of $\overline{\rho}$ to $A$.
\end{enumerate}
\end{Def}
By \cite[Theorem 1.2.2]{Balaji}, $D_{\Lambda}^{\boxempty}(\overline{\rho})$ is representable
by an object $R_{\Lambda}^{\boxempty}(\overline{\rho})\in \mathcal{C}_{\Lambda}$. As an 
examination of its proof easily yields, we get (with respect to the ring of integers
$\Lambda'$ of some finite extension of $\operatorname{Quot}(\Lambda)$ with residue
field $k_{\Lambda'}=k$) an isomorphism
\begin{equation}\label{20170316_lifting_ring_base_change}
R_{\Lambda'}^{\boxempty}(\overline{\rho}) \cong
\Lambda'\, \widehat{\otimes}_{\Lambda}\, R_{\Lambda}^{\boxempty}(\overline{\rho}).
\end{equation}
\begin{Def}\label{20170316_def_lifting_cond}
A lifting condition is a family $\mathcal{D}= (S(A))_{A\in \mathcal{C}_{\Lambda}}$ of subsets
$S(A)\subset D_{\Lambda}^{\boxempty}(\overline{\rho})(A)$ s.t.
\begin{enumerate}
\item $\overline{\rho}\in S(k)$;
\item If $f:A\rightarrow B$ is a morphism in $\mathcal{C}_{\Lambda}$ and $\rho\in S(A)$, then
$G(f)\circ \rho \in S(B)$;
\item Let $f_1: A_1 \rightarrow A, f_2: A_2 \rightarrow A$ be morphisms in $\mathcal{C}_{\Lambda}$ and let $\rho_3$ be a lifting of $\overline{\rho}$ to $A_3 := A_1\times_A A_2$.
For $i=1,2$ denote by $\pi_i: A_3\rightarrow A_i$ the canonical map and by $\rho_i$ the lifting
$G(\pi_i)\circ \rho_3$ of $\overline{\rho}$ to $A_i$. Then, $\rho_3\in S(A_3)$ if and only
if $\rho_1 \in S(A_1)$ and $\rho_2 \in S(A_2)$.
\end{enumerate}
\end{Def}
Condition 2. guarantees that $\mathcal{D}$ defines a subfunctor 
$D_{\Lambda}^{\boxempty, \mathcal{D}}(\overline{\rho})\subset
D_{\Lambda}^{\boxempty}(\overline{\rho})$. Condition 3. is a variation of the 
Mayer-Vietoris property, so a standard argument yields
\begin{Prop}\label{20170324_prop_rev_rep}
$D_{\Lambda}^{\boxempty, \mathcal{D}}(\overline{\rho})$ is a relatively representable
subfunctor (in the sense of \cite[\S 19]{Mazur1}) of $D_{\Lambda}^{\boxempty}(\overline{\rho})$, i.e. representable by some $R_{\Lambda}^{\boxempty, \mathcal{D}}(\overline{\rho})\in \mathcal{C}_{\Lambda}$. On the other hand, any representable subfunctor
$F\subset D_{\Lambda}^{\boxempty}(\overline{\rho})$ yields a lifting condition 
$\mathcal{D}= (S(A))_{A\in \mathcal{C}_{\Lambda}}$ via $S(A) := F(A)$.
\end{Prop}
We have the following conditioned version of (\ref{20170316_lifting_ring_base_change}): 
\begin{equation}\label{20170406_conditioned_tensor_prod}
R_{\Lambda'}^{\boxempty, \mathcal{D}'}(\overline{\rho}) \cong
\Lambda'\, \widehat{\otimes}_{\Lambda}\, R_{\Lambda}^{\boxempty,\mathcal{D}}(\overline{\rho}),
\end{equation}
where the condition $\mathcal{D}'$ on the left is a truncated version of $\mathcal{D}$, i.e.
denotes the family of those $S(A)$ as in the definition of $\mathcal{D}$ for which $A\in 
\mathcal{C}_{\Lambda'}$. We will often omit this distinction and write $\mathcal{D}$ in 
place of $\mathcal{D}'$.
\begin{Rem}\label{20170406_remark_extended_conds}
Let $\Lambda$ be as above and let $^{\ast}\mathcal{C}_{\Lambda}$ denote the category of complete Noetherian local $\Lambda$-algebras $A$ such that $[k_A:k]$ is finite. Then one can extend $D^{\boxempty}_{\Lambda}(\overline{\rho})$ to a functor on $^{\ast}\mathcal{C}_{\Lambda}$ by considering $A$-valued liftings of $\overline{\rho}$ as continuous group homomorphisms
$\rho:\Gamma \rightarrow G(A)$ which fulfill $\operatorname{mod}_{\mathfrak{m}_A}\circ\rho = \iota_{k\subset k_A}\circ\overline{\rho}$, where $\iota_{k\subset k_A}: G(k)\rightarrow 
G(k_A)$ is the map induced by the structure map $\Lambda\rightarrow A$. It is easy to check
that this extended functor is representable by the same universal object $R^{\boxempty}_{\Lambda}(\overline{\rho})$ as the functor from Definition \ref{20170406_def_lifting}. Moreover, if $\Lambda'$ is the ring of integers of some finite extension of $\operatorname{Quot}(\Lambda)$ such that $[k_{\Lambda'}:k]<\infty$, we have the following 
version of (\ref{20170316_lifting_ring_base_change}): 
\begin{equation*}
R_{\Lambda'}^{\boxempty}(\iota_{k\subset k_A}\circ\overline{\rho}) \cong
\Lambda'\, \widehat{\otimes}_{\Lambda}\, R_{\Lambda}^{\boxempty}(\overline{\rho}).
\end{equation*}
Moreover, if $\mathcal{D}$ is an extended lifting condition, i.e. a family 
$(S(A))_{A\in \,^{\ast}\mathcal{C}_{\Lambda}}$
fulfilling the analogue conditions of 
Definition \ref{20170316_def_lifting_cond} (with $A,A_i, B\in \,^{\ast}\mathcal{C}_{\Lambda}$), we have the following conditioned version
of (\ref{20170406_conditioned_tensor_prod}):
\begin{equation*}
R_{\Lambda'}^{\boxempty, \mathcal{D}}(\iota_{k\subset k_A}\circ\overline{\rho}) \cong
\Lambda'\, \widehat{\otimes}_{\Lambda}\, R_{\Lambda}^{\boxempty,\mathcal{D}}(\overline{\rho}),
\end{equation*}
where $\mathcal{D}$ on the left hand side is to be understood as the $\Lambda'$-truncated 
version of the condition $\mathcal{D}$, i.e. a family indexed by $^{\ast}\mathcal{C}_{\Lambda'}$ instead of $^{\ast}\mathcal{C}_{\Lambda}$. Moreover, the statement of 
Lemma \ref{20170405_lemma_res_of_coeff_ring} holds if $\Lambda'$ is in $^{\ast}\mathcal{C}_{\Lambda}$ instead of $\mathcal{C}_{\Lambda}$. (The content of this remark is strongly
inspired by the treatment in \cite[Appendix A]{CDT} and \cite{Mazur1}.)
\end{Rem}

\begin{Def}
\begin{enumerate}
\item A deformation of $\overline{\rho}$ to $A\in \mathcal{C}_{\Lambda}$ is an equivalence
class of liftings to $A$, where two lifts are taken to be equivalent if they are
conjugate by some element of $\widehat{G}(A) := \operatorname{ker}(\operatorname{mod}_{\mathfrak{m}_A})$.
\item Denote by $D_{\Lambda}(\overline{\rho}): \mathcal{C}_{\Lambda} \rightarrow
\operatorname{Sets}$ the functor which assigns to an object $A\in \mathcal{C}_{\Lambda}$ the
set of all deformations of $\overline{\rho}$ to $A$.
\end{enumerate}
\end{Def}
For the following, denote by $Z_G$ the center of $G$ and by $\mathfrak{g}$ (resp. by $\mathfrak{z}$) the Lie algebra of
the special fiber of $G$ (resp. of $Z_G$). We assume from now on that $Z_G$ is formally smooth over $\Lambda$.
\begin{Thm}[{\cite[Theorem 3.3]{Tilouine}}]\label{20170316_Tilouines_representability_thm}
If $H^0(\Gamma, \mathfrak{g}) = 
\mathfrak{z}$ then $D_{\Lambda}(\overline{\rho})$ is representable by an object 
$R_{\Lambda}(\overline{\rho})\in \mathcal{C}_{\Lambda}$.
\end{Thm}
Observe that in the case $G=\operatorname{GL}_n$, the condition of 
Theorem \ref{20170316_Tilouines_representability_thm} becomes the usual centralizer condition 
$\operatorname{End}_{k[\Gamma]}(\overline{\rho}) = k$. In practice, this is often deduced
from absolute irreducibility of $\overline{\rho}$ by Schur's Lemma. This reasoning can 
be adopted to more general groups $G$ as follows:
\begin{Def}[Absolute Irreducibility, cf. \cite{SerreMorsound}]\label{20170322_defn_abs_irred}
We say that $\overline{\rho}$ is absolutely irreducible if there does not exist a proper 
parabolic subgroup $P\subsetneq G$ over $\overline{k}$ such that $\overline{\rho}(\Gamma) \subset P$.
\end{Def}
Then the following can be deduced from \cite[Proposition 2.13]{LiebeckTesterman}:
\begin{Lem}[Schur's Lemma]
Assume that $\ell$ is very good for $G$ (cf. \cite[Section 2]{BMRT}) or that there exists an 
embedding $G\hookrightarrow \operatorname{GL}(V)$ such that $(\operatorname{GL}(V), G)$ is
a reductive pair (in the sense of \cite[Definition 3.32]{LiebeckTesterman}). Then $H^0(\Gamma, \mathfrak{g}) = 
\mathfrak{z}$ if $\overline{\rho}$ is absolutely irreducible.
\end{Lem}
We now give an appropriate version of Definition \ref{20170316_def_lifting_cond}:
\begin{Def}\label{20170323_defn_defo_cond}
A deformation condition is a 
lifting condition in the sense of 
Definition \ref{20170316_def_lifting_cond} 
which fulfills additionally
\begin{enumerate}
\item[4.] If ${\rho}\in S(A)$ and $g\in \widehat{G}(A)$, then $g\rho g^{-1} \in S(A)$.
\end{enumerate}
\end{Def}
This defines a relatively representable subfunctor $D^{\mathcal{D}}_{\Lambda}(\overline{\rho})$
of $D_{\Lambda}(\overline{\rho})$: If $D_{\Lambda}(\overline{\rho})$ is representable,
the so is $D^{\mathcal{D}}_{\Lambda}(\overline{\rho})$ and the representing object
$R^{\mathcal{D}}_{\Lambda}(\overline{\rho})$ is a quotient of 
$R_{\Lambda}(\overline{\rho})$.
In addition to the conditions appearing in Section \ref{20170316_sect_defo_conds} below, we will be interested in the following conditions:
\begin{enumerate}
\item If $\Delta\subset \Gamma$ is a profinite subgroup and $\overline{\rho}(\Delta) = \{1\}$, then the assignment $S(A) := \{\rho | \rho(\Delta) = \{1\}\}$ defines a deformation condition. In the case $\Gamma = \operatorname{Gal}_K$ for a local field $K$ and $\Delta = I_K$, we call 
this the unramified lifting condition and write $D_{\Lambda}^{(\boxempty), \tt{nr}}(\overline{\rho})$ for the corresponding subfunctor.
\item Fix a representation $\chi: \Gamma \rightarrow G^{\tt{ab}}(\Lambda)$ such that $d(k)\circ\overline{\rho}= \overline{\chi}$, where $d: G\rightarrow G^{\tt{ab}}$ is the canonical
projection modulo the derived subgroup $G^{\tt{der}}$ and where $\overline{\chi}$ denotes 
the reduction of $\chi$. In accordance with the case $G=\operatorname{GL}_n$, we call this 
the fixed deformation condition and write $D_{\Lambda}^{(\boxempty), \chi}(\overline{\rho})$ for the corresponding subfunctor.
\item Let $\Gamma = \operatorname{Gal}_{F}$ for a global field $F$ and $\Sigma\subset \operatorname{Pl}_F$ a set of places and fix for each $\nu\in \Sigma$ a local condition $D_{\nu}$ of the functor $D^{(\boxempty)}_{\Lambda}(\overline{\rho}_{\nu})$, where $\overline{\rho}_{\nu}$ denotes the restriction of $\overline{\rho}$ to a decomposition group at $\nu$. Then the assignment $S(A) = \{\rho|\rho_{\nu} \in D^{(\boxempty), D_{\nu}}_{\Lambda}(\overline{\rho}_{\nu}) \, \forall \nu\in \Sigma\}$ defines a global condition, denoted by 
$\mathcal{D}= (D_{\nu})_{\nu\in \Sigma}$. The afforded subfunctor of $D^{(\boxempty)}_{\Lambda}(\overline{\rho})$ is denoted by $D^{(\boxempty), \mathcal{D}}_{\Lambda}(\overline{\rho})$.
\item If $\Gamma, F, \Sigma$ are as above and if $\overline{\rho}$ is unramified outside $\Sigma$, then requiring that a lift $\rho$ is unramified outside $\Sigma$ defines a global deformation condition, and we denote the corresponding subfunctor by  
$D^{(\boxempty)}_{\Sigma,\Lambda}(\overline{\rho})$. It is easily seen that studying these
lifts is equivalent to studying unconditioned lifts of $\overline{\rho}$, understood as a 
representation of the Galois group $\operatorname{Gal}_{F, \Sigma}$ of the maximal, unramified outside $\Sigma$ extension $F_{\Sigma}$ of $F$.
\end{enumerate}
It is easily seen that decreeing multiple conditions defines another condition, i.e. it makes sense to write for example $D^{\boxempty, \chi, \tt{nr}}_{\Lambda}(\overline{\rho})$.
\paragraph{Multiply framed deformations}
Continue to denote $\Gamma = \operatorname{Gal}_F$ and fix finite subsets $\Sigma \subset S\subset
\operatorname{Pl}_F$ such that $\overline{\rho}$ is unramified outside $S$.
\begin{Def}\label{20170328_multiply_framed_defo_functor}
Following \cite[Section 4.1.1]{KW2}, we define the functor $D^{\boxempty_{\Sigma}}_{\Lambda}(\overline{\rho}):\mathcal{C}_{\Lambda}
\rightarrow \operatorname{Sets}$  by
\begin{equation*}
A \mapsto \Bigl\{ (\rho, (\rho_{\nu},\beta_{\nu})_{\nu\in\Sigma}) \;\Bigl|\; 
\substack{
\rho\in D^{\boxempty}_{\Lambda}(\overline{\rho})(A),\;
\rho_{\nu}\in D^{\boxempty}_{\Lambda}(\overline{\rho}_{\nu})(A),\; \beta_{\nu}\in \widehat{G}(A)\\ \text{s. t. }
\rho|\operatorname{Gal}(F_{\nu}) = \beta_{\nu}\rho_{\nu}\beta_{\nu}^{-1}}
\Bigr\}_{/\sim}
\end{equation*}
where $(\rho, (\rho_{\nu},\beta_{\nu})_{\nu\in\Sigma})$ and $(\rho', (\rho_{\nu}',\beta_{\nu}')_{\nu\in\Sigma})$ are taken 
to be equivalent if $\rho_{\nu} = \rho_{\nu}'$ for all $\nu$ and if there is a $\gamma \in \widehat{G}(A)$
such that $\rho' = \gamma \rho \gamma^{-1}$ and $\beta_{\nu}' = \gamma^{-1}\beta_{\nu}$ for all $\nu$.
\end{Def}
Note that specifying the $\rho_{\nu}$ is not strictly necessary, as they can be obtained from $\rho$ and $\beta_{\nu}$. We can impose a deformation condition $\mathcal{D}=(S(A))_{A \in \mathcal{C}_{\Lambda}}$ on multiply framed deformations in the same way we did for liftings and deformations, i.e. we allow only those triples $(\rho, (\rho_{\nu},\beta_{\nu})_{\nu\in\Sigma})$ for which $\rho\in S(A)$.
The following assertions are immediate, cf. \cite[Proposition 4.1]{KW2} or \cite[Proposition 2.62]{MyPhd}:
\begin{Prop}\label{20170525_rep_res_proposition}
\begin{enumerate}
\item $D^{\boxempty_{\Sigma}, (\chi), \mathcal{D}}_{[S], \Lambda}$ is representable and 
we denote the afforded deformation ring by $R^{\boxempty_{\Sigma}, (\chi), \mathcal{D}}_{[S], \Lambda}$ (if $\Sigma=\emptyset$, we have to assume $H^0(\Gamma, \mathfrak{g}) = 
\mathfrak{z}$);
\item If $\#\Sigma = 1$, then the functors $D^{\boxempty_{\Sigma}, (\chi), \mathcal{D}}_{[S], \Lambda}$ and $D^{\boxempty, (\chi), \mathcal{D}}_{[S], \Lambda}$ are naturally isomorphic;
\item If $\Sigma\neq \emptyset$, then 
\begin{equation*}
R^{\boxempty_{\Sigma}, (\chi), \mathcal{D}}_{[S], \Lambda} \cong
R^{\boxempty, (\chi), \mathcal{D}}_{[S], \Lambda}\llbracket x_1, \ldots, x_{t}\rrbracket
\text{ and, if }H^0(\Gamma, \mathfrak{g}) = 
\mathfrak{z}, \text{ then also }
R^{\boxempty, (\chi), \mathcal{D}}_{[S], \Lambda} \cong
R^{(\chi), \mathcal{D}}_{[S], \Lambda}\llbracket x_1, \ldots, x_{u}\rrbracket
\end{equation*}
with $t= \operatorname{dim}(\mathfrak{g}).(\#\Sigma -1), u = 
\operatorname{dim}(\mathfrak{g}) - \operatorname{dim}(\mathfrak{z}) = 
\operatorname{dim}(\mathfrak{g}^{\tt{der}})$.
%
\end{enumerate}
\end{Prop}
From now on, let us suppose 
\begin{Ass}\label{20150904_Ass_H0_vanishing_for_presentations}
$H^0(\operatorname{Gal}_{F,S}, \mathfrak{g}^{\tt{der}})= 0$.
\end{Ass}
With respect to a deformation condition $\mathcal{D}= (D_{\nu})_{\nu\in \Sigma}$ as in 
example 4. above, we set 
\begin{equation*}
R_{\Lambda}^{\tt{loc}_{\Sigma}, (\chi), \mathcal{D}}(\overline{\rho})\; := \;\widehat{\bigotimes}_{\nu\in \Sigma}
\;R_{\Lambda}^{\boxempty, (\chi_{\nu}), D_{\nu}}(\overline{\rho}_{\nu}).
\end{equation*}
The following is essentially a special case of 
\cite[Proposition 4.2.5]{Balaji} (which goes back
to \cite[Proposition 4.1.5]{Kisin2}):
\begin{Prop} \label{20170329_first_rep_result}
Assume that $D^{(\chi)}_{\Lambda}(\overline{\rho})$ is representable. Then
\begin{equation*}
R^{\boxempty_{\Sigma}, (\chi), \mathcal{D}}_{S, \Lambda}(\overline{\rho}) \cong
R_{\Lambda}^{\tt{loc}_{\Sigma}, (\chi), \mathcal{D}}(\overline{\rho})\llbracket x_1, \ldots, x_{a+b}\rrbracket/(f_1, \ldots, f_a)
\end{equation*}
for suitable $a\in \mathbb{N}, f_i \in R_{\Lambda}^{\tt{loc}_{\Sigma}, (\chi), \mathcal{D}}(\overline{\rho})\llbracket x_1, \ldots, x_{a+b}\rrbracket$ and with $b=0$ if the determinant is not fixed (resp. $b=(\#\Sigma -1).\operatorname{dim}(\mathfrak{g}^{\tt{ab}})$ if the 
determinant is fixed).
\end{Prop}

\begin{Cor}\label{2017_cor_on_repn}
Assume that each $R^{\boxempty, (\chi_{\nu}), D_{\nu}}(\overline{\rho}_{\nu})$ is a complete intersection
ring of relative dimension $d_{\nu}$ over $\Lambda$. Assume moreover that $D^{\mathcal{D}}_{[S], \Lambda}(
\overline{\rho})$ is representable and that $d := \sum_{\nu\in\Sigma} d_{\nu} > \operatorname{dim}(\mathfrak{g}).\#\Sigma - \operatorname{dim}(\mathfrak{z})-b$ (with $b$ as in Proposition \ref{20170329_first_rep_result}). Then there exists a presentation 
\begin{equation*}
R^{\mathcal{D}}_{[S], \Lambda}(
\overline{\rho}) \cong \Lambda\llbracket x_1, \ldots, x_m\rrbracket /(f_1, \ldots, f_m)
\end{equation*}
for a suitable $m\in \mathbb{N}$.
\end{Cor}
\begin{proof}
Using Proposition \ref{20170329_first_rep_result} and the assumption on $\mathcal{D}$, we can
write
\begin{equation*}
R^{\boxempty_{\Sigma}, (\chi), \mathcal{D}}_{S, \Lambda}(\overline{\rho}) \cong
R_{\Lambda}^{\tt{loc}_{\Sigma}, (\chi), \mathcal{D}}(\overline{\rho})\llbracket x_1, \ldots, x_{a+b}\rrbracket/(f_1, \ldots, f_a)
\cong 
\Lambda\llbracket x_1, \ldots, x_{a+b+c+d}\rrbracket/(f_1, \ldots, f_{a+c})
\end{equation*}
for $a, b$ as above and for a suitably chosen $c\in\mathbb{N}_0$. On the other hand, 
by Cohen's structure theorem we can write 
$R^{(\chi), \mathcal{D}}_{S, \Lambda}(\overline{\rho}) \cong
\Lambda\llbracket x_1, \ldots, x_{u}\rrbracket/(f_1, \ldots, f_v)$ for suitable $u,v\in \mathbb{N}_0$ (and we assume that this is a minimal presentation, i.e. that the quantity
$u-v$ is maximal among all ways to write $R^{(\chi), \mathcal{D}}_{S, \Lambda}(\overline{\rho})$ as a quotient of a power series ring), so 
by the third part of Proposition \ref{20170525_rep_res_proposition}
we have
\begin{equation*}
R^{\boxempty_{\Sigma}, (\chi), \mathcal{D}}_{S, \Lambda}(\overline{\rho}) \cong
R^{(\chi), \mathcal{D}}_{S, \Lambda}(\overline{\rho})\llbracket x_1, \ldots, x_r\rrbracket
\cong \Lambda\llbracket x_1, \ldots, x_{r+u}\rrbracket/(f_1, \ldots, f_v)
\end{equation*}
with $r = \operatorname{dim}(\mathfrak{g}).\#\Sigma - \operatorname{dim}(\mathfrak{z})$.
Comparing these two presentations, we get
\begin{equation*}
u-v + \operatorname{dim}(\mathfrak{g}).\#\Sigma - \operatorname{dim}(\mathfrak{z}) \geq
b + d \Rightarrow u-v \geq b+d -\operatorname{dim}(\mathfrak{g}).\#\Sigma + \operatorname{dim}(\mathfrak{z}).
\end{equation*}
Thus, the claim follows immediately from our assumption on $d$.
\end{proof}

\paragraph{Tangent spaces and systems of local conditions}
With respect to a deformation condition $\mathcal{D}$ will consider the tangent space $t_{D^{(\boxempty), \mathcal{D}}_{\Lambda}} = D^{(\boxempty),\mathcal{D}}_{\Lambda}(k[\epsilon])$, which we consider as a (finite-dimensional) $k$-vector space (cf.
\cite[Lecture 2]{Gouvea}). There are canonical isomorphisms
\begin{equation*}
t_{D^{\boxempty, (\chi)}_{\Lambda}} \cong Z^1(\Gamma, \mathfrak{g}^{(\tt{der})}),
t_{D_{\Lambda}} \cong H^1(\Gamma, \mathfrak{g})
 \text{ and } 
t_{D_{\Lambda}^{\chi}} \cong H^1(\Gamma, \mathfrak{g}^{\tt{der}})':= 
\operatorname{im}\bigl(H^1(\Gamma, \mathfrak{g}^{\tt{der}}) \rightarrow
H^1(\Gamma, \mathfrak{g})\bigr),
\end{equation*}
so via the embedding $D^{(\chi),\mathcal{D}}_{\Lambda}(k[\epsilon]) \hookrightarrow
D^{(\chi)}_{\Lambda}(k[\epsilon])$ we are provided with an assignment 
$\mathcal{D}\mapsto L(\mathcal{D})^{(\chi)} := D^{(\chi),\mathcal{D}}_{\Lambda}(k[\epsilon])$ from deformation conditions to subspaces of $H^1(\Gamma, \mathfrak{g})$
(resp. $H^1(\Gamma, \mathfrak{g}^{\tt{der}})'$).
In the case $\Gamma = \operatorname{Gal}_F$ for a number field $F$ and if $\mathcal{D}= 
(D_{\nu})_{\nu\in \Sigma}$, we call the afforded family $\mathcal{L}^{(\chi)} = 
(L(D_{\nu}))_{\nu\in\operatorname{Pl}_F}$ of subspaces of $H^1(\operatorname{Gal}_{F_{\nu}}, \mathfrak{g})$
(resp. of $H^1(\operatorname{Gal}_{F_{\nu}}, \mathfrak{g}^{\tt{der}})'$)  
a \textit{system of local conditions}.
Also note that there is an exact sequence
\begin{equation*}
0 \rightarrow \mathfrak{g}/\mathfrak{g}^{\Gamma} \rightarrow t_{D^{\boxempty, (\chi)}_{\Lambda}} \rightarrow t_{D^{(\chi)}_{\Lambda}}
\end{equation*}
where, in case $\ell\gg 0$ (such that $\mathfrak{g}= \mathfrak{g}^{\tt{der}}\oplus
\mathfrak{g}^{\tt{ab}}$), the object $\mathfrak{g}/\mathfrak{g}^{\Gamma}$ can be replaced 
by $\mathfrak{g}^{\tt{der}}/(\mathfrak{g}^{\tt{der}})^{\Gamma}$.
\paragraph{Liftings at infinity}
\begin{Prop}\label{archProp}
Assume $\Gamma = \mathbb{Z}/2\mathbb{Z} = \{1,c\}$ and $\ell = \operatorname{char}(\mathbb{F}) \neq 2$.  Then 
\begin{equation*}
R^{\boxempty}_{\Lambda}(\overline{\rho}) \cong \Lambda\llbracket x_1, \ldots, x_m\rrbracket \text{ with } m = \operatorname{dim}(\mathfrak{g}^{c = -1}).
\end{equation*}
If $\psi$ is a lift of the determinant, then the same result holds for $R^{\boxempty,\psi}_{\Lambda}(\overline{\rho})$ after replacing
$\mathfrak{g}$ by $\mathfrak{g}^{\tt{der}}$.
\end{Prop}
\begin{proof}
We use the general formula $H^2(\mathbb{Z}/n\mathbb{Z}, M) = M^{\mathbb{Z}/n\mathbb{Z}} / \operatorname{im}(\varphi)$
with 
\begin{equation*}
\varphi: M\rightarrow M\qquad m\mapsto  \sum_{j=0}^{n-1} j.m.
\end{equation*}
Now, if 
$x\in \mathfrak{g}^{\{1,\tilde{c}\}}$, we see that $(\tilde{c} +1)(\frac{1}{2}x) = x \in \operatorname{im}(\tilde{c}+1)$, hence
$H^2(\{1,c\}, \mathfrak{g})=0$ and the lifting ring is unobstructed. To get the number of variables we have to evaluate 
\begin{equation*}
Z^1(\{1,c\}, \mathfrak{g}) = \{f:\{1,c\}\rightarrow \mathfrak{g}\;|\; f(xy)= f(x)+\,^xf(y)\}. 
\end{equation*}
Looking at $x=y=c$, we see that $f$ is uniquely determined by a vector $v = f(c)$. Looking at $x=1, y=c$, we see that $f(1) = v + \,^cv = 0$, 
i. e. that $v \in \mathfrak{g}^{c=-1}$. On the other hand, any such $v$ defines an $f\in Z^1$ via $1\mapsto 0, c\mapsto v$.\\
The modifications of this argument for the fixed-determinant case are straight-forward.
\end{proof}
\paragraph{A simple criterion for the vanishing of cohomology groups}
Now assume that $\Gamma= \operatorname{Gal}_K$ for a local field $K$.
Recall that, by local Tate duality, the Pontryagin dual of $H^2(\Gamma,\mathfrak{g})$ can 
be identified with $H^0(\Gamma, \mathfrak{g}^{\vee}) = (\mathfrak{g}^{\vee})^{\Gamma}$. Together with the identification of $(\operatorname{ad}\overline{\rho}^{(0)})^{\vee}$ and
$(\operatorname{ad}\overline{\rho}^{(0)})(1)$ via the trace pairing, this implies the following
criterion for the vanishing of $H^2(\Gamma, \mathfrak{g}^{\tt{der}})$ in the case $G=\operatorname{GL}_n$:
\begin{Lem} [Local case] \label{20150830_lemma_h2_vanishing_crit}
Let $\Gamma$ be the absolute Galois group of a non-archimedean local field, $k$ be a finite field of characteristic $\ell$ and
\begin{equation*}
\overline{\rho}: \Gamma \rightarrow \operatorname{GL}_n(k)
\end{equation*}
a representation. 
\begin{enumerate}
\item If $\operatorname{Hom}_{\Gamma}( \overline{\rho}, \overline{\rho}(1))$ vanishes, then 
$H^2(\Gamma, \operatorname{ad}\overline{\rho})$ vanishes.
\item Assume that $\ell\not|\, n$. 
Then, if $\operatorname{Hom}_{\Gamma}( \overline{\rho}, \overline{\rho}(1))$ vanishes, also 
$H^2(\Gamma, \operatorname{ad}\overline{\rho}^{0})$ vanishes.
\end{enumerate}
\end{Lem}

In the global case, there is no such duality and we record the following:
\begin{Lem} [Global case] \label{20150904_Lemma_glob_case}
Let $\Gamma=\operatorname{Gal}_{F,S}$ for a number field $F$ and a (possibly empty) finite set of places of $F$. Let $k,\overline{\rho}$ be as in Lemma \ref{20150830_lemma_h2_vanishing_crit} above.
\begin{enumerate}
\item If $\operatorname{Hom}_{\Gamma}( \overline{\rho}, \overline{\rho}(1))$ vanishes, then 
$H^0(\Gamma, (\operatorname{ad}\overline{\rho})^{\vee})$ vanishes.
\item Assume that $\ell\not|\, n$. 
Then, if $\operatorname{Hom}_{\Gamma}( \overline{\rho}, \overline{\rho}(1))$ vanishes, also 
$H^0(\Gamma, (\operatorname{ad}\overline{\rho}^{0})^{\vee})$ vanishes.
\end{enumerate}
\end{Lem}
We easily deduce the following result, which also implies the vanishing of the error term $\delta$ in \cite{Boeckle_CRM} (see Remark 5.2.3.(d) of loc.cit.) for large $\ell$:
\begin{Cor}\label{20150904_newH1dualvanishing}
There exists a constant $C$, depending only on $n$ and $F$, such that Assumption \ref{20150904_Ass_H0_vanishing_for_presentations}
holds if $\operatorname{char}(k)>C, G=\operatorname{GL}_n$ and $\overline{\rho}$ is irreducible.
\end{Cor}
\paragraph{Unobstructedness}
\begin{Def}\label{20170322_def_easiest_unobstructedness}
The functor $D^{(\boxempty), [\chi]}_{\Lambda}(\overline{\rho})$ is called unobstructed if $h^2(\Gamma, \mathfrak{g}^{[\tt{der}]})=0$.
\end{Def}
\begin{Def}
A relatively representable subfunctor of $D^{(\boxempty),[\chi]}_{\Lambda}(\overline{\rho})$ is called smooth (of dimension $m$) if its representing object is isomorphic to $\Lambda\llbracket
x_1, \ldots, x_m\rrbracket$.
\end{Def}
The most apparent application of the unobstructedness-property is that it implies the smoothness of the lifting/deformation ring, cf. \cite{Boeckle_presentations}: Assume that $D^{\boxempty, (\chi)}_{\Lambda}(\overline{\rho})$ is smooth and (in the fixed-determinant case) that $\ell\gg 0$ and (in the non-framed case) that 
$D^{(\chi)}_{\Lambda}(\overline{\rho})$ is representable. Then
\begin{equation*}
D^{\boxempty, (\chi)}_{\Lambda}(\overline{\rho}) \cong 
\llbracket x_1,\ldots,x_{a(+c)}\rrbracket \text{ and } 
 D^{(\chi)}_{\Lambda}(\overline{\rho}) \cong 
\llbracket x_1,\ldots,x_{b(+c)}\rrbracket
\end{equation*}
with $b = h^1(\Gamma, \mathfrak{g}), c= h^1(\Gamma, \mathfrak{g}^{\tt{der}})' - b, 
a = b + \operatorname{dim}(\mathfrak{g}^{(\tt{der})}) - h^0(\Gamma, \mathfrak{g}^{(\tt{der})})$. 
The converse direction (i.e. that smoothness implies unobstructedness) is known not to hold (for general profinite groups $\Gamma$), cf. 
\cite{Sprang}.
%
%
%
%

In order to relax this notion to functors corresponding to deformation conditions, we restrict to the case $\Gamma=\operatorname{Gal}_{F, S}$. 
%
Let $\mathcal{L}^{(\chi)}= (L^{(\chi)}_{\nu})_{\nu\in \operatorname{Pl}_F}$ be a system 
of local conditions and $\mathcal{D}^{(\chi)} = (D_{\nu}^{(\chi)})_{\nu\in \operatorname{Pl}_F}$ the corresponding global deformation condition.

Denote by $\mathfrak{g}^{(\tt{der}), \vee}$ the Tate dual of $\mathfrak{g}^{\tt{der}}$ and by
$L_{\nu}^{(\chi), \perp}$ the annihilator of $L_{\nu}^{(\chi)}$ under the Tate pairing
\begin{equation*}
H^i(F_{\nu}, \mathfrak{g}^{\tt{der}, \vee}) \times 
H^{2-i}(F_{\nu}, \mathfrak{g}^{\tt{der}}) \longrightarrow
H^2(F_{\nu}, k(1)) \cong \mathbb{Q}/\mathbb{Z}
\end{equation*}
for $i=1$, cf. \cite[(7.2.6) Theorem]{NSW}. Then we denote the corresponding dual Selmer group
by
\begin{equation*}
H^1_{\mathcal{L}^{(\chi), \perp}}(F, \mathfrak{g}^{(\tt{der}), \vee}) :=
\operatorname{ker}\Bigl(
\bigoplus_{\nu\in \operatorname{Pl}} \operatorname{res}_{\nu}: 
H^1(F, \mathfrak{g}^{(\tt{der}), \vee}) \longrightarrow 
\bigoplus_{\nu\in \operatorname{Pl}}
H^1(F, \mathfrak{g}^{(\tt{der}), \vee})/L^{(\chi), \perp}_{\nu}
\Bigr).
\end{equation*}
From now on, let us assume that $D^{(\chi)}_{\nu}$ for $\nu\notin S$ parametrizes unramified 
deformations.
\begin{Def}\label{20170407_def_glob_unobstr}
We say that $D^{\mathcal{D}^{(\chi)}}_{S, \Lambda}(\overline{\rho})$ (or  
$D^{\boxempty, \mathcal{D}^{(\chi)}}_{S, \Lambda}(\overline{\rho})$, 
or  
$D^{\boxempty_{\Sigma}, \mathcal{D}^{(\chi)}}_{S, \Lambda}(\overline{\rho})$ for some set 
of places $\Sigma$) has vanishing dual Selmer group if 
$H^1_{\mathcal{L}^{(\chi), \perp}}(F, \mathfrak{g}^{(\tt{der}), \vee})=0$.
\end{Def}
\begin{Def}
Let $\textbf{m} = (m_{\nu})_{\nu\in S}\in \mathbb{N}^{S}_0$. We say that  
$D^{\mathcal{D}^{(\chi)}}_{S, \Lambda}(\overline{\rho})$ (or  
$D^{\boxempty, \mathcal{D}^{(\chi)}}_{S, \Lambda}(\overline{\rho})$, 
or  
$D^{\boxempty_{\Sigma}, \mathcal{D}^{(\chi)}}_{S, \Lambda}(\overline{\rho})$) is 
globally unobstructed (of local dimensions $\textbf{m}$) if its dual Selmer group vanishes
and if each $D^{\boxempty, \mathcal{D}^{(\chi_{\nu})}}_{\Lambda}(\overline{\rho}_{\nu})$ for
$\nu \in S$ is smooth (of dimension $m_{\nu}$).
\end{Def}
We remark that if $D^{\mathcal{D}^{(\chi)}}_{S, \Lambda}(\overline{\rho})$ is globally unobstructed and representable, then by \cite[Theorem 5.2]{Boeckle_presentations} the 
representing object $R^{\mathcal{D}^{(\chi)}}_{S, \Lambda}(\overline{\rho})$
is isomorphic to a power series ring in 
$h^1_{\mathcal{L}^{(\chi)}}(F, \mathfrak{g}^{(\tt{der})})^{(\prime)}$ variables. 
The following results directly from the exact sequece of \cite[p. 7]{Boeckle_presentations},
\begin{equation*}
0 \rightarrow \Sh^2_{S}(\mathfrak{g}^{(\tt{der})}) \rightarrow 
H^2(\operatorname{Gal}_{F,S}, \mathfrak{g}^{(\tt{der})}) \rightarrow
\bigoplus_{\nu\in S} H^2(F_{\nu}, \mathfrak{g}^{(\tt{der})})
\rightarrow 
H^0(F, \mathfrak{g}^{(\tt{der}), \vee})^{\ast} \rightarrow 0,
\end{equation*}
where $H^0(F, \mathfrak{g}^{(\tt{der}), \vee})^{\ast}$ vanishes
for $\ell \gg 0$:
\begin{Prop}
Assume that $D^{(\chi_{\nu})}_{\Lambda}(\overline{\rho}_{\nu})$ is unobstructed (for all $\nu\in S$) and that 
$D^{(\chi)}_{S, \Lambda}(\overline{\rho})$ is globally unobstructed (without making an assumption on the 
dimension). Then $D^{(\chi)}_{S, \Lambda}(\overline{\rho})$ is unobstructed
in the sense of Definition \ref{20170322_def_easiest_unobstructedness}. For $\ell \gg 0$, also the converse is true.
\end{Prop}

\section{A general framework for unobstructedness}\label{20170405_section_on_framework}
For this section, we take the following static point of view: Let $k$ be a finite field with ring of Witt vectors $W=W(k)$, 
let $S$ be a finite set of finite places of $F$. We assume $\ell :=\operatorname{char}(k) \notin S\cup \{2\}$. 
Then we fix  a continuous representation
\begin{equation*}
\overline{\rho}: \operatorname{Gal}_{F,S}\rightarrow G(k)
\end{equation*}
together with a lift $\chi: \operatorname{Gal}_{F,S}\rightarrow G^{\tt{ab}}(k)$ of the determinant. Let us moreover fix a Borel subgroup $B\subset G$ and denote by
$\mathfrak{g}^{\tt{der}}$ (resp. $\mathfrak{b}^{\tt{der}}$) the Lie algebra of the derived subgroup 
$G^{\tt{der}}$ (resp. the Lie algebra of $B\cap G^{\tt{der}}$). 

With respect to some choice\footnote{During the following applications of the presented material, we will consider for $\tt{min}$ the condition of Section \ref{20170321_sect_min_ram},
for $\tt{crys}$ the condition of Section \ref{20170321_sect_crys_defos} and for 
$\tt{sm}$ the unconditioned deformation condition. We stress, however, that for the purpose of this section we treat $\tt{min, crys, sm}$ purely formally as deformation conditions satisfying the listed assumptions of Definition \ref{20170316_def_lifting_cond}.} of local deformation conditions 
\begin{itemize}
\item $\tt{min}$ of the restriction $\overline{\rho}_{\nu}$ of $\overline{\rho}$ to a decomposition group at $\nu\in S$,
\item $\tt{sm}$ and $\tt{crys}$ of the restriction $\overline{\rho}_{\nu}$ of $\overline{\rho}$ to a decomposition group at $\nu |\ell$,
\end{itemize}
consider the following list of assumptions, where we leave out the $W$ in the subscript
of the occurring deformation functors and rings:
\begin{enumerate}
\item[1.] {\textbf{(sm/$k$):}} For each $\nu |\ell$, the subfunctor $D^{\boxempty, \chi_{\nu}, \tt{sm}}(\overline{\rho}_{\nu})$ of 
$D^{\boxempty, \chi_{\nu}}(\overline{\rho}_{\nu})$ is representable by a formally smooth (over $W$) object $R_{\nu}^{\boxempty, \chi_{\nu}, \tt{sm}}$ (and 
we denote the dimension by $d_{\nu}^{\boxempty, \tt{sm}}$).
\item[2.] {\textbf{(crys):}} For each $\nu |\ell$, the subfunctor $D^{\boxempty, \chi_{\nu}, \tt{crys}}(\overline{\rho}_{\nu})$ of 
$D^{\boxempty, \chi_{\nu}}(\overline{\rho}_{\nu})$ is representable by a formally smooth (over $W$) object $R_{\nu}^{\boxempty, \chi_{\nu}, \tt{crys}}$ 
of relative dimension
\begin{equation*}
d^{\boxempty, \tt{crys}}_{\nu} = \operatorname{dim}(\mathfrak{g}^{\tt{der}}) + \bigl(\operatorname{dim}(\mathfrak{g}^{\tt{der}})
- \operatorname{dim}(\mathfrak{b}^{\tt{der}})\bigr)[F_{\nu}:\mathbb{Q}_{\ell}].
\end{equation*}
\item[3.] {\textbf{(min):}} For each $\nu\in S$, the subfunctor $D^{\boxempty, \chi_{\nu}, \tt{min}}(\overline{\rho}_{\nu})$ of 
$D^{\boxempty, \chi_{\nu}}(\overline{\rho}_{\nu})$ is representable by a formally smooth (over $W$) object $R_{\nu}^{\boxempty, \chi_{\nu}, \tt{min}}$ 
of relative dimension
\begin{equation*}
d^{\boxempty, \tt{min}}_{\nu} = \operatorname{dim}(\mathfrak{g}^{\tt{der}}).
\end{equation*}
\item[4.] {\textbf{($\infty$):}} For each $\nu|\infty$, the functor $D^{\boxempty, \chi_{\nu}}(\overline{\rho}_{\nu})$ is representable
by an object (over $W$) of relative dimension 
$d^{\boxempty}_{\nu} = \operatorname{dim}(\mathfrak{b}^{\tt{der}})$. (As $\ell > 2 = \#\operatorname{Gal}_{F_{\nu}}$, the
strict $\ell$-cohomological dimension $\operatorname{scd}_{\ell}(\operatorname{Gal}_{F_{\nu}})$ is zero, i.e. the representing 
object is automatically formally smooth over $W$.)
\item[5.] {\textbf{(Presentability):}} There exists a presentation
\begin{equation*}
R^{\boxempty_{S_{\ell}}, \chi, \tt{min}, \tt{sm}}_{S_{\ell}} \cong R^{\tt{loc, min, sm}}_{S_{\ell}}\llbracket x_1, 
\ldots, x_a \rrbracket/(f_1, \ldots, f_b)
\end{equation*}
for integers $a,b$ fulfilling $a-b= (\#S_{\ell} - 1).\operatorname{dim}(\mathfrak{g}^{\tt{ab}})$. In this equation, we take
\begin{equation}\label{20160209_Rloc_defn}
R^{\text{loc}, \tt{min}, \tt{sm}}_{S_{\ell}} = \widehat{\bigotimes_{\nu\in S_{\ell}}}\; \widetilde{R}_{\nu} \; \text{with} \;
\widetilde{R}_{\nu} = 
\begin{cases} 
R_{\nu}^{\boxempty, \chi_{\nu}, \tt{min}} & \text{ if $\nu \in S$};\\
R_{\nu}^{\boxempty, \chi_{\nu}, \tt{sm}} & \text{ if $\nu|\ell$};\\
R_{\nu}^{\boxempty, \chi_{\nu}} & \text{ if $\nu|\infty$}.
\end{cases}
\end{equation}
\item[6.]{\textbf{($R=T$):}} The ring $R^{\boxempty_{S_{\ell}}, \chi, \tt{min}, \tt{crys}}_{S_{\ell}}$ is formally smooth of relative dimension
\begin{equation*}
r_0 := \operatorname{dim}(\mathfrak{g}).\#S_{\ell} - \operatorname{dim}(\mathfrak{g}^{\tt{ab}}).
\end{equation*}
\end{enumerate}
\begin{Rem}[Taylor-Wiles condition]
Let $\nu|\infty$ so that $\operatorname{scd}_{\ell}(\operatorname{Gal}_{F_{\nu}}) = 0$, then it follows from condition $(\infty)$, $\operatorname{scd}_{\ell}(\operatorname{Gal}_{F_{\nu}}) = 0$ and the remark following Definition \ref{20170322_def_easiest_unobstructedness} that
\begin{equation*}
\operatorname{dim}(\mathfrak{b}^{\tt{der}}) = \operatorname{dim}_W(R^{\boxempty}) = 
h^1(\operatorname{Gal}_{F_{\nu}}, \mathfrak{g}^{\tt{der}})' + \operatorname{dim}(\mathfrak{g}^{\tt{der}})
- h^0(\operatorname{Gal}_{F_{\nu}}, \mathfrak{g}^{\tt{der}}) = \operatorname{dim}(\mathfrak{g}^{\tt{der}})
- h^0(\operatorname{Gal}_{F_{\nu}}, \mathfrak{g}^{\tt{der}}).
\end{equation*}
This implies
\begin{equation}\label{20160908_TWFormula}
\sum_{\nu | \infty} h^0(\operatorname{Gal}_{F_{\nu}}, \mathfrak{g}^{\tt{der}}) = [F:\mathbb{Q}].\bigl(\mathfrak{g}^{\tt{der}}
- \operatorname{dim}(\mathfrak{b}^{\tt{der}})\bigr).
\end{equation}
\end{Rem}
We can now state the main result of this section:
\begin{Thm}\label{20160908_MainThm}
Suppose conditions 1.-6. are met and, for $\nu|\ell$, write $d^{\boxempty, \tt{sm}}_{\nu}= \operatorname{dim}(\mathfrak{g}^{\tt{der}}).([F_{\nu}:\mathbb{Q}_{\ell}]
+1) - \delta_{\nu}$ for suitable numbers $\delta_{\nu}\in \mathbb{N}_0$. 
\begin{enumerate}
\item The ring $R^{\boxempty_{S_{\ell}}, \chi, \tt{min}, \tt{sm}}_{S_{\ell}}$ is formally smooth of relative dimension 
\begin{equation*}
\# S_{\ell}.\operatorname{dim}(\mathfrak{g}) - \operatorname{dim}(\mathfrak{g}^{\tt{ab}}) + 
[F:\mathbb{Q}].\operatorname{dim}(\mathfrak{b}^{\tt{der}}) - \sum_{\nu|\ell} \delta_{\nu}.
\end{equation*}
If the unframed deformation functor $D^{\chi, \tt{min}, \tt{sm}}_{S_{\ell}}$ is representable, then 
$R^{\chi, \tt{min}, \tt{sm}}_{S_{\ell}}$ is formally smooth of relative dimension $[F:\mathbb{Q}].\operatorname{dim}(\mathfrak{b}^{\tt{der}})
- \sum_{\nu|\ell} \delta_{\nu}$.
\item Let $\mathcal{L}:= (L_{\nu}^{\chi})_{\nu}$ be the system of local conditions corresponding to the deformation functor
$D^{\chi, \tt{min}, \tt{sm}}_{S_{\ell}}(\overline{\rho})$. Assume 
\begin{enumerate}
\item[a)] $\mathfrak{g}= \mathfrak{g}^{\tt{der}} \oplus \mathfrak{g}^{\tt{ab}}$ (e.g. because $\ell\gg 0$);
\item[b)] $H^{0}(\operatorname{Gal}_{F}, \mathfrak{g}^{\tt{der}, \vee}) =0$;
\item[c)] For $\nu\in S$, we have $\operatorname{dim}(L_{\nu}) = h^0(\operatorname{Gal}_{F_{\nu}}, \mathfrak{g}^{\tt{der}})$;
\item[d)] All $\delta_{\nu}$ vanish.
\end{enumerate}
Then $H^1_{\mathcal{L}^{\perp}}(\operatorname{Gal}_{F, S}, \mathfrak{g}^{\tt{der}, \vee}) = H^0(\operatorname{Gal}_{F,S}, \mathfrak{g}^{\tt{der}}) = 0$.
\end{enumerate}
\end{Thm}
\begin{Rem}
\begin{enumerate}
\item As the deformation conditions $\tt{sm}$ and $\tt{crys}$ are relatively representable (cf. conditions 1. and 2.), 
$D^{\chi, \tt{min}, \tt{sm}}_{S_{\ell}}$ is representable if $D_{S_{\ell}}^{\chi}$ is representable. For example, 
this is the case if $\overline{\rho}$ is absolutely irreducible (in the sense of Definition
\ref{20170322_defn_abs_irred}).
\item For $\nu\notin S_{\ell}$, the equality
$\operatorname{dim}(L_{\nu}) = h^0(\operatorname{Gal}_{F_{\nu}}, \mathfrak{g}^{\tt{der}})$ holds automatically if $\ell\gg 0$ (so that 
$\mathfrak{g}= \mathfrak{g}^{\tt{der}}\oplus \mathfrak{g}^{\tt{ab}}$). 
\end{enumerate}
\end{Rem}
\begin{proof}
First remark that the second claim of part 1. follows directly from Lemma 2.15, as $R_{S_{\ell}}^{\boxempty_{S_{\ell}}, 
\chi, \tt{min}, \tt{sm}}$ is a power series ring over $R_{S_{\ell}}^{\chi, \tt{min}, \tt{sm}}$, and from the formula $\operatorname{dim}\mathfrak{g} = \operatorname{dim}\mathfrak{g}^{\tt{der}} + 
\operatorname{dim}\mathfrak{g}^{\tt{ab}}$.

For the first sentence of 1., we use the shorthand notation $d^{\star}_{T} = \sum_{\nu\in T} d^{\star}_{\nu}$ for 
a subset $T$ of $\operatorname{Pl}_F$. Moreover, we write $d^{\boxempty}_{\infty}$ for $d^{\boxempty}_{\Omega_{\infty}}$ and
$d^{\star}_{\ell}$ for $d^{\star}_{\Omega_{\ell}}$. Let us consider the commutative diagram
\begin{equation*}
\xymatrix{
0 \ar[r]& I \ar[d]_{\pi}\ar[r] & R^{\tt{loc}, \tt{min}, \tt{sm}}_{S_{\ell}} \ar[r]^f \ar[d]_{\pi}& \ar[d]^{\pi'}
R^{\tt{loc}, \tt{min}, \tt{crys}}_{S_{\ell}} \ar[r]  & 0\\
0 \ar[r]& J \ar[r] & R^{\boxempty_{S_{\ell}}, \chi, \tt{min},\tt{sm}}_{S_{\ell}} \ar[r]_g& 
R^{\boxempty_{S_{\ell}}, \chi, \tt{min},\tt{crys}}_{S_{\ell}} \ar[r]& 
0.
}
\end{equation*}
In this diagram, the right square is a pushout square, $R^{\tt{loc, min, crys}}_{S_{\ell}}$ is 
defined as in (\ref{20160209_Rloc_defn}) (but with $\tilde{R}_{\nu} = R_{\nu}^{\boxempty, \chi_{\nu}, \tt{crys}}$ for $\nu|\ell$) and $f,g$ are the canonical projections.
Moreover, $\pi= \otimes_{\nu\in S_{\ell}}\pi_{\nu}$ is induced from the natural transformations 
\begin{equation*}
D^{\boxempty_{S_{\ell}}, \chi, \tt{min}, \tt{crys}}_{S_{\ell}} \rightarrow 
\widetilde{D}_{\nu},
\end{equation*}
where $\widetilde{D}_{\nu}$ is the deformation functor corresponding to (i.e. represented
by) the ring $\widetilde{R}_{\nu}$ in (\ref{20160209_Rloc_defn}) and, analogously, 
$\pi' = \otimes_{\nu\in S_{\ell}}\pi_{\nu}'$ is defined with $\tt{crys}$ in place of 
$\tt{sm}$.

Using the list of assumptions, we can rewrite the above diagram as
\begin{equation*}
\xymatrix{
0 \ar[r] & I\ar[d]_{\pi} \ar[r] & W\llbracket x_1, \ldots, x_{d^{\boxempty,\tt{sm}}_{\ell}+
d^{\boxempty}_{\infty} + d^{\boxempty, \tt{min}}_{S}}\rrbracket \ar[r]^f \ar[d]_{\pi} &
W\llbracket x_1, \ldots, x_{d^{\boxempty,\tt{crys}}_{\ell}+
d^{\boxempty}_{\infty} + d^{\boxempty, \tt{min}}_{S}}\rrbracket \ar[r] \ar[d]^{\pi'} & 0 \\
0 \ar[r] & J \ar[r] & W\llbracket x_1, \ldots, x_{m}\rrbracket/(f_1, \ldots,
f_{m-\gamma}) \ar[r]_<<<<<<<<<<g &
W\llbracket x_1, \ldots, x_{r_0}\rrbracket \ar[r] & 0
}
\end{equation*}
with $\gamma = (\#S_{\ell} -1).\operatorname{dim}(\mathfrak{g}^{\tt{ab}}) +
d^{\boxempty, \tt{sm}}_{\ell} + d^{\boxempty}_{\infty} + d^{\boxempty, \tt{min}}_S$.
It is easily seen that $R^{\boxempty_{S_{\ell}}, \chi, \tt{min},\tt{sm}}_{S_{\ell}}$
is formally smooth if we can show
$\operatorname{gen}(J) \leq m - (m-\gamma ) - r_0 = \gamma - r_0$.
From the pushout property of the diagram, we can easily deduce that $\operatorname{gen}(J)
\leq \operatorname{gen}(I)$. As $f$ is a surjection of regular rings, it follows from 
(Serre, Local Algebra,
Proposition 22) that $\operatorname{gen}(I) = d^{\boxempty, \tt{sm}}_{\ell} - 
d^{\boxempty, \tt{crys}}_{\ell}$. Thus, we are left to show the inequality
\begin{equation*}
d^{\boxempty, \tt{sm}}_{\ell} - 
d^{\boxempty, \tt{crys}}_{\ell} \leq \gamma - r_0 = 
(\#S_{\ell} -1).\operatorname{dim}(\mathfrak{g}^{\tt{ab}}) + d^{\boxempty, \tt{sm}}_{\ell} +
d^{\boxempty}_{\infty} + d^{\boxempty, \tt{min}}_{S} - 
\operatorname{dim}(\mathfrak{g}).\#S_{\ell} + \operatorname{dim}(\mathfrak{g}^{\tt{ab}})
\end{equation*}
\begin{equation*}
= \#S_{\ell}.(\operatorname{dim}(\mathfrak{g}^{\tt{ab}}) - \operatorname{dim}(\mathfrak{g}))
+ d^{\boxempty, \tt{sm}}_{\ell} +
d^{\boxempty}_{\infty} + d^{\boxempty, \tt{min}}_{S}.
\end{equation*}
By assumptions {\textbf{(min)}} and {\textbf{($\infty$)}} and by the identity 
$\operatorname{dim}(\mathfrak{g}^{\tt{der}}) + \operatorname{dim}(\mathfrak{g}^{\tt{ab}})
= \operatorname{dim}(\mathfrak{g})$, this amounts to
\begin{equation*}
d^{\boxempty, \tt{crys}}_{\ell} \geq 
\operatorname{dim}(\mathfrak{g}^{\tt{der}}).(\#\Omega_{\ell} + [F:\mathbb{Q}])
- \operatorname{dim}(\mathfrak{b}^{\tt{der}})[F:\mathbb{Q}].
\end{equation*}
Assumption {\textbf{(crys)}} amounts precisely to the fact that this inequality is fulfilled
(with equality), which implies the formal smoothness of 
$R^{\boxempty_{S_{\ell}}, \chi, \tt{min}, \tt{sm}}_{S_{\ell}}$.
Moreover, we easily check that the relative
dimension of $R^{\boxempty_{S_{\ell}}, \chi, \tt{min}, \tt{sm}}_{S_{\ell}}$ is 
\begin{gather*}
\gamma = (\#S_{\ell} -1).\operatorname{dim}(\mathfrak{g}^{\tt{ab}}) +
d^{\boxempty, \tt{sm}}_{\ell} + d^{\boxempty}_{\infty} + d^{\boxempty, \tt{min}}_S \\
= \# S_{\ell}.\operatorname{dim}\mathfrak{g}^{\tt{ab}} - 
\operatorname{dim}\mathfrak{g}^{\tt{ab}}
+ \operatorname{dim}\mathfrak{g}^{\tt{der}}.([F:\mathbb{Q}] + \#\Omega_{\ell}) 
- \Bigl(\;\sum_{\nu | \ell} \delta_{\nu}\Bigr) 
+ [F:\mathbb{Q}].\operatorname{dim}(\mathfrak{b}^{\tt{der}}) 
+ \# S.\operatorname{dim}(\mathfrak{g}^{\tt{der}})\\
= \# S_{\ell}.\operatorname{dim}(\mathfrak{g}) + 
[F:\mathbb{Q}].\operatorname{dim}(\mathfrak{b}^{\tt{der}})
- \operatorname{dim}(\mathfrak{g}^{\tt{ab}})
-  \sum_{\nu | \ell} \delta_{\nu}.
\end{gather*}

Concerning part 2., note that (using condition a)) we have an exact sequence
\begin{equation*}
0 \rightarrow 
\mathfrak{g}/\mathfrak{g}^{\operatorname{Gal}_{F_{\nu}}} 
= \mathfrak{g}^{\tt{der}}/(\mathfrak{g}^{\tt{der}})^{\operatorname{Gal}_{F_{\nu}}}
\rightarrow t_{D_{W}^{\boxempty, \chi_{\nu}, \tt{sm}}(\overline{\rho}_{\nu})}
\rightarrow t_{D_{W}^{\chi_{\nu}, \tt{sm}}(\overline{\rho}_{\nu})}
\rightarrow 0
\end{equation*}
for $\nu|\ell$. Therefore, using condition 2. and the vanishing of $\delta_{\nu}$,
we have for $\nu|\ell$ the following:
\begin{equation*}
\operatorname{dim}(L_{\nu}) = 
 \operatorname{dim} t_{D_{W}^{\chi_{\nu}, \tt{sm}}(\overline{\rho}_{\nu})}
= h^0(\operatorname{Gal}_{F_{\nu}}, \mathfrak{g}^{\tt{der}}) + 
[F_{\nu}:\mathbb{Q}_{\ell}].\operatorname{dim}(\mathfrak{g}^{\tt{der}}).
\end{equation*}
Recall the Greenberg-Wiles-Formula [NSW08, Theorem 8.7.9]:
\begin{gather*}
\operatorname{dim} H^1_{\mathcal{L}}(\operatorname{Gal}_{F,S}, \mathfrak{g}^{\tt{der}})
- \operatorname{dim} H^1_{\mathcal{L}^{\perp}}
(\operatorname{Gal}_{F,S}, \mathfrak{g}^{\tt{der},\vee})\\
= h^0(\operatorname{Gal}_{F,S}, \mathfrak{g}^{\tt{der}}) - 
h^0(\operatorname{Gal}_{F,S}, \mathfrak{g}^{\tt{der},\vee}) 
+ \sum_{\nu\in S_{\ell}} \bigl( \operatorname{dim}(L_{\nu}) 
- h^0(\operatorname{Gal}_{F_{\nu}}, \mathfrak{g}^{\tt{der}})\bigr)
\end{gather*}
By \cite[Section 5]{Boeckle_presentations}, we know that 
$H^1_{\mathcal{L}}(\operatorname{Gal}_{F,S}, \mathfrak{g}^{\tt{der}})$ can be 
identified with the tangent space of the functor 
$D^{\chi, \tt{min}, \tt{sm}}_{S_{\ell}}$ and hence (by part 2.) equals 
$[F:\mathbb{Q}].\operatorname{dim}(\mathfrak{b}^{\tt{der}})$. 
For $\nu|\infty$, we have $L_{\nu} \subset H^1(\operatorname{Gal}_{F,S}, 
\mathfrak{g}^{\tt{der}}) = 0$. Thus, using the Taylor-Wiles formula 
(\ref{20160908_TWFormula}) and assumption b), the sum evaluates to
\begin{equation*}
\sum_{\nu\in S_{\ell}} \bigl( \operatorname{dim}(L_{\nu}) 
- h^0(\operatorname{Gal}_{F_{\nu}}, \mathfrak{g}^{\tt{der}})\bigr) 
= [F:\mathbb{Q}].\operatorname{dim}(\mathfrak{g}^{\tt{der}}) - 
[F:\mathbb{Q}].\bigl(\operatorname{dim}(\mathfrak{g}^{\tt{der}})
- \operatorname{dim}(\mathfrak{b}^{\tt{der}})\bigr).
\end{equation*}
Therefore we get 
\begin{equation*}
- \operatorname{dim} H^1_{\mathcal{L}^{\perp}}
(\operatorname{Gal}_{F,S}, \mathfrak{g}^{\tt{der},\vee})
= h^0(\operatorname{Gal}_{F,S}, \mathfrak{g}^{\tt{der}}).
\end{equation*}
As neither quantity can be negative, they must both vanish and the result follows.
\end{proof}
From the exact sequence
\begin{equation*}
H^1_{\mathcal{L}^{\perp}}
(\operatorname{Gal}_{F,S}, \mathfrak{g}^{\tt{der},\vee})^{\ast}
\rightarrow
\Sh^2_{S_{\ell}}(\mathfrak{g}^{\tt{der}}) \rightarrow 0
\end{equation*}
(see e.g. equation (9) on p. 10 of [Boeckle07]) we can deduce:
\begin{Cor}\label{20170314_cor_317_phd}
Under the assumptions of part 2. of Theorem \ref{20160908_MainThm}, 
$\Sh^2_{S_{\ell}}(\mathfrak{g}^{\tt{der}})$ vanishes. In particular, under these hypotheses the unrestricted 
deformation functor $D^{(\boxempty_{S_{\ell}}), \chi}_{S_{\ell}}(\overline{\rho})$ is 
globally unobstructed precisely if the local deformation functors 
$D^{(\boxempty), \chi_{\nu}}(\overline{\rho}_{\nu})$ are relatively smooth
for $\nu\in S \cup \Omega_{\ell}$. 
\end{Cor}
We remark that 
$D^{(\boxempty), \chi_{\nu}}(\overline{\rho}_{\nu})$ is
relatively smooth for $\nu\in \Omega_{\infty}$
by Proposition \ref{archProp}, so Corollary \ref{20170314_cor_317_phd}
holds true with \glqq ... for $\nu \in S_{\ell}$\grqq $\,$in place of 
\glqq ... for $\nu\in S \cup \Omega_{\ell}$\grqq.
\paragraph{Potential unobstructedness}
We start with the following, easy observation:
\begin{Prop}\label{20170405_prop_easy_pot_unobstr}
Let $K$ be a local field and let $K'$ be a finite extension 
of $K$ such that $\ell$ does not divide the index 
$[K':K]$. Let $\overline{\rho}$ be a $G$-valued residual representation of $\operatorname{Gal}_K$ and fix a 
lift $\chi$ of the determinant. Then
unobstructedness of $D^{(\chi)}_{\Lambda}(\overline{\rho}|\operatorname{Gal}_{K'})$ implies unobstructedness of 
$D^{(\chi)}_{\Lambda}(\overline{\rho})$.
\end{Prop}
\begin{proof}
This follows immediately from the injectivity of  
\begin{equation*}
\operatorname{res}_{{K'}|{K}}: H^2({K}, \mathfrak{g}^{(\tt{der})}) \rightarrow
H^2({K'}, \mathfrak{g}^{(\tt{der})}),
\end{equation*}
cf. \cite[Corollary (1.5.7)]{NSW}.
\end{proof}
This proof is not directly applicable to the global 
situation, as we have to keep track of the set of places at which
we allow ramification.
Therefore, we first describe a more flexible method which can also handle conditioned 
deformation functors:
\begin{Def}[Dual-pre condition]\label{20170324_def_dual_pre_cond}
Let $F'|F$ be a finite extension of number fields.
\begin{enumerate}
\item Let $\nu'\in \operatorname{Pl}_{F'}, \nu\in \operatorname{Pl}_F$ such that $\nu'|\nu$. 
Moreover, let $L_{\nu}\subset H^1(F_{\nu}, 
\mathfrak{g}^{(\tt{der})}), L_{\nu'}'\subset H^1(F_{\nu'}', 
\mathfrak{g}^{(\tt{der})})$ be local conditions.
We say that $L_{\nu}$ is a dual-pre-$L_{\nu'}'$ condition if 
$\operatorname{res}^{\vee}_{\nu'}(L^{\perp}_{\nu}) \subset L_{\nu'}'^{\perp}$ , where
\begin{equation*}
\operatorname{res}^{\vee}_{\nu'}: H^1(F_{\nu}, \mathfrak{g}^{(\tt{der}), \vee}) 
\rightarrow
H^1(F_{\nu'}', \mathfrak{g}^{(\tt{der}), \vee}) 
\end{equation*}
denotes the usual restriction map.
\item 
Let $\mathcal{L}' = (L_{\nu'}')_{\nu'\in \operatorname{Pl}_{F'}}$ be a system of
local conditions for $F'$. We say that a system 
$\mathcal{L} = (L_{\nu})_{\nu\in \operatorname{Pl}_{F}}$
of local conditions for $F$ is dual-pre-$\mathcal{L}'$ if for each pair 
$\nu, \nu'$ as above, $L_{\nu}$ is a dual-pre-$L_{\nu'}'$ condition.
\end{enumerate}
\end{Def}
\begin{Ex}\label{20170323_ex_unram_pre_cond}
Let $F,F'$ be as in Definition \ref{20170324_def_dual_pre_cond} and fix a finite set $S\subset \operatorname{Pl}_F$ such that 
$\overline{\rho}$ is unramified outside $S$. 
Take for $\mathcal{L}$ the local system parametrizing all deformations which are
unramified outside $S$, i.e. $L_{\nu}= H^1(F_{\nu}, 
\mathfrak{g}^{(\tt{der})})$ if $\nu\in S$ and $L_{\nu}=H^1(\operatorname{Gal}_{F_{\nu}}/I_{F_{\nu}}, 
\mathfrak{g}^{(\tt{der})})$ otherwise. Analogously, let $\mathcal{L}'$ the local system parametrizing all deformations which are
unramified outside $S\langle F'\rangle$. Then any lift of $\overline{\rho}$ which is unramified outside $S$ is, after restriction to $\operatorname{Gal}_{F'}$, a lift of $\overline{\rho}|\operatorname{Gal}_{F'}$ which is unramified outside $S\langle F'\rangle$. But this implies easily that the restriction map 
$\operatorname{res}_{\nu'}: H^1(F_{\nu}, \mathfrak{g}^{(\tt{der})}) 
\rightarrow
H^1(F_{\nu'}', \mathfrak{g}^{(\tt{der})})$ maps $L_{\nu}$ into $L_{\nu'}'$ for any pair of places $\nu, \nu'$ with $\nu'|\nu$. Using the fact that Tate duality is given by the cup product, we see that $\mathcal{L}$
is dual-pre-$\mathcal{L}'$.
\end{Ex}
\begin{Lem}\label{20170404_pot_unobstr}
Let 
$\overline{\rho},F$ and $F'$ be as above and assume $(\ell, [F':F])=1$. 
Let $\mathcal{L} = (L_{\nu})_{\nu\in \operatorname{Pl}_{F}}, \mathcal{L}' = (L_{\nu'}')_{\nu'\in \operatorname{Pl}_{F'}}$
be systems of local conditions (with associated deformation conditions $\mathcal{D}$ and 
$\mathcal{D}'$) such that $\mathcal{L}$ is dual-pre-$\mathcal{L}'$.
Moreover, assume that $D^{[\boxempty], (\chi), \mathcal{D}'}(\overline{\rho}|\operatorname{Gal}_{F'})$
has vanishing dual Selmer group. Then also $D^{[\boxempty], (\chi),\mathcal{D}}(\overline{\rho})$
has vanishing dual Selmer group.
\end{Lem}
\begin{proof}
As above, the invertibility of $[F':F]$ implies that the restriction map
\begin{equation*}
H^1(\operatorname{Gal}_{F}, \mathfrak{g}^{(\tt{der}), \vee}) \rightarrow
H^1(\operatorname{Gal}_{F'}, \mathfrak{g}^{(\tt{der}), \vee})
\end{equation*}
is injective. Consider the diagram
\begin{equation*}
\xymatrix{
H^1_{\mathcal{L}^{\perp}}({F}, \mathfrak{g}^{(\tt{der}),\vee}) 
\ar@{^{(}->}[r]\ar^{\varphi}[d]&
H^1({F}, \mathfrak{g}^{(\tt{der}),\vee}) \ar[r]\ar@{^{(}->}[d]&
\bigoplus_{\nu\in \operatorname{Pl}_F} 
H^1(F_{\nu}, \mathfrak{g}^{(\tt{der}),\vee})/L_{\nu}^{\perp} \ar[d] \\
H^1_{\mathcal{L}^{\prime,\perp}}({F'}, \mathfrak{g}^{(\tt{der}),\vee}) 
\ar@{^{(}->}[r]&
H^1({F'}, \mathfrak{g}^{(\tt{der}),\vee}) \ar[r]&
\bigoplus_{\nu'\in \operatorname{Pl}_{F'}} 
H^1(F_{\nu'}', \mathfrak{g}^{(\tt{der}),\vee})/L_{\nu'}^{\prime,\perp}
}
\end{equation*}
The vertical map on the 
right is defined because $\mathcal{L}$ is dual-pre-$\mathcal{L}'$, and this 
implies the well-definedness of $\varphi$. A simple diagram chase implies injectivity of 
$\varphi$, from which the claim follows.
\end{proof}
The following follows now directly from 
Example \ref{20170323_ex_unram_pre_cond} and 
Lemma \ref{20170404_pot_unobstr}:
\begin{Cor}\label{20170525_final_pot_cor}
Let $F$ be a number field and let $F'$ be a finite extension
of $F$ such that $\ell$ does not divide the index $[F':F]$.
Let $\overline{\rho}$ be a $G$-valued residual representation
of $\operatorname{Gal}_{F}$ which is unramified outside a 
finite set of places $S$ and fix a lift $\chi$ of the 
determinant. Then unobstructedness of $D^{(\chi|\operatorname{Gal}_{F'})}_{\Lambda,
S\langle F'\rangle}(\overline{\rho}|\operatorname{Gal}_{F'})$ 
implies unobstructedness of 
$D^{(\chi)}_{\Lambda}(\overline{\rho})$.
\end{Cor}

\section{Local deformation conditions for $G=\operatorname{GL}_n$}\label{20170316_sect_defo_conds}
Let $K$ be a finite extension of $\mathbb{Q}_p$ and let
$k$ be a finite field of characteristic $\ell$.
In the following, we consider deformation conditions for a continuous representation
$\overline{\rho}: \operatorname{Gal}_K \rightarrow \operatorname{GL}_n(k)$.
\subsection{Unrestricted deformations ($p\neq\ell$)}
In the case $p\neq \ell$, we have the following result:
\begin{Thm} \label{20170213_Helms_theorem}
$\operatorname{Spec}R^{\boxempty}(\overline{\rho})$ is a reduced complete 
intersection, flat and equidimensional of relative dimension $n^2$ over $\operatorname{Spec} W$.
\end{Thm}
\begin{proof}
This is Theorem 2.5 in \cite{Shotton_prep}.
\end{proof}
\subsection{Unrestricted deformations ($p=\ell$)} \label{20170324_section_unres_defos}
For the remainder of this subsection, we assume that $\overline{\rho}$
is the semi-simplification of the reduction of a crystalline representation
\begin{equation*}
\rho: \operatorname{Gal}_K\rightarrow \operatorname{GL}_n(L)
\end{equation*}
for a suitable finite extension $L$ of $\mathbb{Q}_p$ with residue field $k$ and for
$p=\ell$.
Denote the set of embeddings $\tau: K\hookrightarrow \overline{\mathbb{Q}}_p$ by 
$\mathbb{E}_K$ and for $\tau\in \mathbb{E}_K$ denote by $\operatorname{HT}_{\tau}(\rho)$ 
the multiset of Hodge-Tate weights of $\rho$ with respect to $\tau$.
\begin{Thm}\label{20170313_Theorem_smoothness_sm_over_ell}\label{20170323_FL_main_thm}
Assume that $K|\mathbb{Q}_p$ is unramified and that for each $\tau\in \mathbb{E}_K$
\begin{enumerate}
\item there exists an $\alpha\in \mathbb{Z}$ such that all Hodge-Tate weights in $\operatorname{HT}_{\tau}(\rho)$
lie in the range $[\alpha, \alpha + \ell - 3]$;
\item the Hodge-Tate weights of $\rho$ are non-consecutive, i.e.
if two numbers $a,b\in \mathbb{Z}$ occur in $\operatorname{HT}_{\tau}(\rho)$, then 
$|a-b|\neq 1$.
\end{enumerate}
Then 
$R^{\boxempty}(\overline{\rho}) \cong W\llbracket x_1, \ldots, x_{m}\rrbracket$
with $m=n^2.([K:\mathbb{Q}_{\ell}] + 1)$.
\end{Thm}
Before we come to the proof, recall the theory of Fontaine-Laffaille \cite{FL}, as normalized in 
\cite{CHT} (see also \cite[Section 1.4]{BLGGT}): We consider the category 
$\underline{\operatorname{FL}}_{\mathcal{O}_K, \mathcal{O}_L}$, consisting of $\mathcal{O}_K
\otimes_{\mathbb{Z}_{\ell}} \mathcal{O}_L$-modules $M$, endowed with a decreasing
filtration $(\operatorname{Fil}^i M)_{i\in \mathbb{Z}}$ with $\operatorname{Fil}^0 M = M$ 
and $\operatorname{Fil}^{\ell -1}M = 0$ and a family of $\operatorname{Frob}\otimes 1 $-linear maps
$\operatorname{Fil}^i M \rightarrow M$ such that $\varphi^i|\operatorname{Fil}^{i+1} = 
\ell.\varphi^{i+1}$ and $\sum_i \varphi^i(\operatorname{Fil}^iM) = M$. Let
$\underline{\operatorname{FL}}_{\mathcal{O}_K, k}$ denote the full subcategory of 
finite length objects which are annihilated by the maximal ideal 
$\varpi_L.\mathcal{O}_L$. We need the following well-known facts:
\begin{itemize}
\item There exists an exact, fully-faithful, covariant and $\mathcal{O}_L$-linear functor
\begin{equation*}
{\tt{G}}_K: \underline{\operatorname{FL}}_{\mathcal{O}_K, \mathcal{O}_L}\rightarrow
\operatorname{Rep}_{\mathcal{O}_L}(\operatorname{Gal}_K).
\end{equation*}
The essential image is closed under taking subobjects and quotients. Moreover, ${\tt{G}}_K$
restricts to a functor 
\begin{equation*}
\underline{\operatorname{FL}}_{\mathcal{O}_K, k}\rightarrow
\operatorname{Rep}_{k}(\operatorname{Gal}_K).
\end{equation*}
\item For $\mathbf{M}\in \underline{\operatorname{FL}}_{\mathcal{O}_K, \mathcal{O}_L}$ projective over $\mathcal{O}_{L}$, we have
\begin{equation*}
\operatorname{HT}_{\tau}({\tt{G}}_K(\mathbf{M}\otimes_{\mathbb{Z}_{p}} \mathbb{Q}_p) )
= \operatorname{FL}_{\tau}(\mathbf{M}\otimes_{\mathcal{O}_L} k),
\end{equation*}
where for $\mathbf{N}\in \underline{\operatorname{FL}}_{\mathcal{O}_K, k}$ we denote by 
$\operatorname{FL}_{\tau}(\mathbf{M})$ the multiset of integers $i$, such that
\begin{equation*}
\operatorname{gr}^i(\mathbf{N}^{\tau}) = 
\operatorname{Fil}^iN \otimes_{\mathcal{O}_K \otimes_{\mathbb{Z}_{p}} \mathcal{O}_L, \tau
\otimes 1} \mathcal{O}_L /\operatorname{Fil}^{i+1}N \otimes_{\mathcal{O}_K \otimes_{\mathbb{Z}_{p}} \mathcal{O}_L, \tau
\otimes 1} \mathcal{O}_L
\end{equation*}
does not vanish, where $i$ is counted with multiplicity $\operatorname{dim}_k 
\operatorname{gr}^i(\mathbf{N}^{\tau})$.
\item Assuming condition 1. of Theorem \ref{20170313_Theorem_smoothness_sm_over_ell}, any $\operatorname{Gal}_K$-stable
$\mathcal{O}_L$-lattice of $\rho$ is in the image of ${\tt{G}}_K$, and so is its reduction
$\Lambda/\varpi_L.\Lambda$.
\item Morphisms in $\underline{\operatorname{FL}}_{\mathcal{O}_K, k}$ are strict with filtrations. If $f: \mathbf{M} \rightarrow \mathbf{N}$ is such a morphism, then 
$f(\operatorname{Fil}^i M) = f(M) \cap \operatorname{Fil}^i N$ for all $i\in \mathbb{Z}$. 
In particular, if $\mathbf{M}, \mathbf{N}\in \underline{\operatorname{FL}}_{\mathcal{O}_K, k}$ fulfill
\begin{equation}\label{20160913_FLweightcond}
\operatorname{FL}_{\tau}(\mathbf{M}) \cap \operatorname{FL}_{\tau}(\mathbf{N})
\end{equation}
for all
$\tau\in \mathbb{E}_K$,
then $\operatorname{Hom}_{\underline{\operatorname{FL}}_{\mathcal{O}_K, k}}
(\mathbf{M}, \mathbf{N})= 0$.
\end{itemize}

\begin{proof}
As $h^2(K, \operatorname{ad}\overline{\rho})$ is an upper 
bound on the number of generators of the kernel of a surjection
$W\llbracket x_1, \ldots, x_{s}\rrbracket \twoheadrightarrow R^{\boxempty}(\overline{\rho})$
with $s = \operatorname{dim}Z^1(K, \operatorname{ad}\overline{\rho})$ (cf. \cite[Proposition 2.1.2]{Allen_polarized}), we have to prove 
\begin{equation}\label{20160913_local_h2_vanishing}
H^2(K, \operatorname{ad}\overline{\rho}) = 0.
\end{equation}
%
%
Moreover, using the 
exact sequence
\begin{equation*}
0 \rightarrow 
\operatorname{ad}\overline{\rho}/(\operatorname{ad}\overline{\rho})^{\operatorname{Gal}_{K}} 
\rightarrow t_{D_{W}^{\boxempty}(\overline{\rho})}
\rightarrow t_{D_{W}(\overline{\rho})}
\rightarrow 0
\end{equation*}
and the local Euler-Poincare formula, we can compute
\begin{equation*}
s = h^1(K, \operatorname{ad}\overline{\rho})
+ n^2 - h^0(K, \operatorname{ad}\overline{\rho})
= n^2 - \chi(K, \operatorname{ad}\overline{\rho}) = 
n^2([K:\mathbb{Q}_p] + 1).
\end{equation*}
Thus, (\ref{20160913_local_h2_vanishing}) implies the claim. 


As the 
trace pairing identifies $\operatorname{ad}\overline{\rho}^{\vee}$ and 
$\operatorname{ad}\overline{\rho}(1)$, we are finished if we can show that
\begin{equation*}
H^2(K, \operatorname{ad}\overline{\rho})^{\ast} \cong 
H^0(K, \operatorname{ad}\overline{\rho}^{\vee}) \cong
H^0(K, \operatorname{ad}\overline{\rho}(1))\cong 
\operatorname{Hom}_{\operatorname{Gal}_K}(\overline{\rho},
\overline{\rho}(1)) = 0.
\end{equation*} 
Because
$\operatorname{Hom}_{\operatorname{Gal}_K}(\overline{\rho},
\overline{\rho}(1)) \cong
\operatorname{Hom}_{\operatorname{Gal}_K}(\overline{\rho}(1-\alpha),
\overline{\rho}(2-\alpha))$,
we can assume without loss of generality that $\alpha = 1$.

It is easy to see that we can choose a $\operatorname{Gal}_K$-stable 
$\mathcal{O}_L$-lattice $\Lambda$ of $\rho$ such that its reduction is semi-simple, 
i.e. $\Lambda/\varpi_L.\Lambda \cong \overline{\rho}$ (if necessary, after replacing $\rho$
by a base change $\rho\otimes_L L'$ to a sufficiently ramified finite extension $L'$ of $L$, 
which does not affect the validity of (\ref{20160913_local_h2_vanishing})).
By our first assumption that all weights of $\rho$ lie in the range $[1, \ell - 2]$, it thus
follows that $\overline{\rho}$ is of the 
form $\tt{G}_K(\mathbf{M})$ for a suitable $\mathbf{M}\in 
\underline{\operatorname{FL}}_{\mathcal{O}_K, k}$. By the same argument, 
$\overline{\rho}(1) = \tt{G}_K(\mathbf{N})$ for a suitable $\mathbf{N}\in 
\underline{\operatorname{FL}}_{\mathcal{O}_K, k}$.
As the 
cyclotomic character shifts the weights by $-1$, the second condition translates
precisely into the condition (\ref{20160913_FLweightcond}). Thus, using that $\tt{G}_K$ is fully faithful, we get
\begin{equation*}
0 = \operatorname{Hom}_{\underline{\operatorname{FL}}_{\mathcal{O}_K, k}}
(\mathbf{M},
\mathbf{N})
\cong 
\operatorname{Hom}_{\operatorname{Gal}_K}(\overline{\rho},
\overline{\rho}(1)).\qedhere
\end{equation*}
\end{proof}

\subsection{Crystalline deformations ($\ell = p$)}\label{20170321_sect_crys_defos}
Consider again a representation $\rho:\operatorname{Gal}_{K} \rightarrow \operatorname{GL}_n({L})$ which fulfills the conditions of Theorem \ref{20170323_FL_main_thm}. We will also make the additional 
assumption that all occurring Hodge-Tate weigths of $\overline{\rho}$ have multiplicity one.
We will consider the deformation 
problem ${\tt{crys}}$ of $\overline{\rho}$ consisting of those lifts 
$\tilde{\rho}:\operatorname{Gal}_{K} \rightarrow \operatorname{GL}_n(A)$ of 
$\overline{\rho}$ for which $\tilde{\rho}\otimes_{A} A'$ lies in the essential image 
of $\texttt{G}_K$ for all Artinian quotients $A'$ of $A$ (cf. \cite{CHT}, Section 2.4.1).
We refer to those lifts as \textit{FL-crystalline lifts} of $\overline{\rho}$.

That ${\tt{crys}}$ defines a deformation condition in the sense of Definition \ref{20170323_defn_defo_cond} was already remarked in \cite{CHT} and follows easily from
the Ramakrishna framework, cf. \cite{Ramakrishna}:
We remarked already in Section \ref{20170324_section_unres_defos} that the essential image of $\texttt{G}_K$ is closed under subobjects
and quotients. That the essential image is closed under direct sums follows immediately
from the exactness of $\texttt{G}_K$, since then $\texttt{G}_K$ preserves direct sums (see \cite[Theorem $3.12^{(\ast)}$]{Freyd}).
Thus we can record the following (where for part 2. we refer to the remark just below Proposition
\ref{20170324_prop_rev_rep}):
\begin{Lem}
Let $\Lambda$ be the ring of integers of a finite, totally ramified extension $E$ of $\operatorname{Quot}(W(k))$ and let $\Lambda'$ be the ring of integers of a finite, totally ramified extension of $E$ (so that we have $k=k_{\Lambda} = k_{\Lambda'}$.) 
Then:
\begin{enumerate}
\item The functor $D_{\Lambda}^{\boxempty, \tt{crys}}(\overline{\rho})$ is representable by a quotient $R_{\Lambda}^{\boxempty, \tt{crys}}(\overline{\rho})$
of $R_{\Lambda}^{\boxempty}(\overline{\rho})$.
\item The functor $D^{\boxempty, \tt{crys}}_{\Lambda'}(\overline{\rho})$ is representable by
\begin{equation}\label{20170406_yet_another_iso}
R^{\boxempty, \tt{crys}}_{\Lambda'}(\overline{\rho}) \cong \Lambda' \otimes_{\Lambda} R_{\Lambda}^{\boxempty, \tt{crys}}(\overline{\rho}).
\end{equation}
\end{enumerate}
\end{Lem}
We remark that the condition $\tt{crys}$ fulfills the extended requirements as described
in Remark \ref{20170406_remark_extended_conds}, so that (\ref{20170406_yet_another_iso})
holds even if $\infty>[k_{\Lambda'}:k_{\Lambda}] >1$.
%
%
%
\begin{Lem}\label{20150712_crystallinity_lemma}
Under the above hypotheses
\begin{equation*}
R_{\Lambda}^{\boxempty, \tt{crys}}(\overline{\rho}) \cong \Lambda[[x_1, \ldots, x_m]]
\end{equation*}
with $m= n^2 + [K:\mathbb{Q}_{\ell}]\frac{n.(n-1)}{2}$.
\end{Lem}
\begin{proof}
This is a part of the statement of \cite[Corollary 2.4.3]{CHT}.
\end{proof}
We also remark that the condition {\tt{crys}} defines a dual-pre-{\tt{crys}} condition in 
the sense of Definition \ref{20170324_def_dual_pre_cond}, cf. \cite[Lemma 4.15]{MyPhd}.

\subsection{Minimally ramified deformations ($p\neq \ell$)}\label{20170321_sect_min_ram}
For this subsection, recall from \cite[Section 2.4.4]{CHT} the minimal ramification 
condition for a lift $\rho$ of $\overline{\rho}$. Let $P_K$ denote the kernel of 
one (hence, any) surjection $I_K \twoheadrightarrow \mathbb{Z}_{\ell}$.
Moreover, let $\Delta_{\overline{\rho}}$ denote the set of equivalence classes of 
$P_K$-representations over $k$ such that $\operatorname{Hom}_{P_K}(\tau, \overline{\rho})
\neq 0$. Then the following can easily be deduced from the material in \cite[Section 2.4.4]{CHT}, in particular
\cite[Corollary 2.4.21]{CHT}:
\begin{Prop}\label{20170526_newProponmin}
Assume that any $\tau\in \Delta_{\overline{\rho}}$ is absolutely irreducible. Then we have:
\begin{enumerate}
\item The condition of being minimally ramified defines a lifting condition, denoted
$\tt{min}$. The representing universal object fulfills
\begin{equation*}
R^{\boxempty, \tt{min}}_{\Lambda}(\overline{\rho})\cong
\Lambda\llbracket X_1, \ldots, X_{n^2}\rrbracket.
\end{equation*}
\item If $\Lambda$ is the ring of integers
of some finite extension of $\operatorname{Quot}(\Lambda)$ with residue
field $k_{\Lambda'}=k$, we have 
\begin{equation*}
R^{\boxempty, \tt{min}}_{\Lambda'}(\overline{\rho}') \cong
\Lambda'\otimes_{\Lambda} R^{\boxempty, \tt{min}}_{\Lambda}(\overline{\rho}).
\end{equation*}
\end{enumerate}  
\end{Prop}

We will be particularly interested in the case where $\overline{\rho}$ has 
\textit{unipotent ramification}\footnote{This notion is explained by the observation that $\overline{\rho}$ is unipotently 
ramified if and only if $\overline{\rho}(I_K)$ lies in a conjugate of the standard 
unipotent subgroup consisting of upper-triangular matrices in $\operatorname{GL}_n(k)$ with 
diagonal entries all equal to $1$.}, i.e. where $\overline{\rho}(P_K) = \{1\}$. 
In the unipotent case, we have a strong connection between
minimally ramified liftings and liftings of prescribed type as considered in \cite{Shotton_phd}. In order
to make this precise, let $E$ denote the
quotient field of $\Lambda$ and $\overline{E}$ its algebraic closure. 
\begin{Def}[Def. 2.10 of \cite{Shotton_phd}]
Let $\tau: I_K\rightarrow \operatorname{GL}_n(\overline{E})$ be a representation which extends to a continuous representation
of the Weil group $W_K$ of $K$ (considered
with the $\ell$-adic topology). Then the isomorphism class of $\tau$ is called an \textit{inertial type}.
(\textit{Warning}: This differs from the usual definition of an inertial type as e.g. in \cite{GeeKisin}.)
\end{Def}
Let $\rho$ be a lift of $\overline{\rho}$ which has values in $\overline{E}$, then we say that
$\rho$ ``is of type $\tau$'' if $\rho|I_K$ is isomorphic to $\tau$.

For the following we consider a $\tau$ which is defined over $E$. Then we say that a morphism 
$x:\operatorname{Spec}\overline{E}\rightarrow \operatorname{Spec}R^{\boxempty}_{\Lambda}(\overline{\rho})$
is of type $\tau$ if the associated $\overline{E}$-valued representation $\rho_x$ is of type $\tau$. This
notion depends only on the image of $x$ (because $\tau$ is defined over $E$). 
\begin{Def}[Fixed type deformation ring, Def. 2.14 of \cite{Shotton_phd}]
Let $R^{\boxempty,\tau}_{\Lambda}(\overline{\rho})$ be the reduced quotient of 
$R^{\boxempty}_{\Lambda}(\overline{\rho})$ which is characterized by the requirement
that $\operatorname{Spec}R^{\boxempty,\tau}_{\Lambda}(\overline{\rho})$ is the Zariski closure 
of the $\overline{E}$-points of type $\tau$ in $\operatorname{Spec}R^{\boxempty}_{\Lambda}(\overline{\rho})$.
\end{Def}
A general classification of inertial types is given in Section 2.2.1 of \cite{Shotton_phd}. Under the unipotent ramification 
assumption, this becomes particularly simple: The set $\mathcal{I}^{\tt{uni}}$ of the isomorphism classes of inertial
types which are trivial on $P_K$ is in bijection with the set $\mathcal{Y}_n$ of Young diagrams of size $n$. The partition $(l_1, \ldots, l_k)$ (with $l_i\geq l_{i+1}$) corresponds
(using the notation of \cite{Shotton_phd}) to the type given by the $I_K$-restriction of the $W_K$-representation 
\begin{equation*}
\bigoplus_{i=1}^k \operatorname{Sp}(\textbf{1}, l_i),
\end{equation*} 
where $\operatorname{Sp}(\bullet,\bullet)$ is defined as in \cite[Section 3.1]{Shotton_prep}.
We can express this differently: Each member of $\mathcal{I}^{\tt{uni}}$ is uniquely characterized by (the conjugacy class of) its value on the generator 
$\zeta:=\zeta_{\operatorname{triv}}$ of $I_K/P_K$, and a bijection $\nabla:\mathcal{Y}_n\rightarrow \mathcal{I}^{\tt{uni}}$ is given by 
\begin{equation}\label{20150810_Nabla_line}
(l_1, \ldots, l_k) \overset{\nabla}{\longmapsto} \tau(\zeta) = \Bigl[1+\begin{pmatrix}
\mathcal{B}_{l_1}&&&\\
&\mathcal{B}_{l_2}&&\\
&&\ddots&\\
&&&\mathcal{B}_{l_k}
\end{pmatrix}\Bigr] \text{ with } \mathcal{B}_{m}= 
\begin{pmatrix}
0&1&&&\\
&0&1&&\\
&&\ddots&\ddots&\\
&&&0&1\\
&&&&0
\end{pmatrix} \in \mathbb{M}_{m\times m}(E).
\end{equation}
On the other hand, we can associate to a $\tau\in \mathcal{I}^{\tt{uni}}$ a partition
of $n$ by considering the kernel sequences: 
\begin{equation*}
\Theta:  \mathcal{I}^{\tt{uni}} \rightarrow \mathcal{Y}_n \qquad \tau \mapsto (s_1, \ldots, s_r)
\end{equation*}
with
\begin{equation*}
s_i :=  \operatorname{dim}\operatorname{ker}(\tau(\zeta) - \textbf{1})^{i} - 
\operatorname{dim}\operatorname{ker}(\tau(\zeta) - \textbf{1})^{i-1}
\end{equation*}
and
\begin{equation*}
r:= \operatorname{min} \bigl\{i\bigl|\operatorname{dim}\operatorname{ker}(\tau(\zeta) - \textbf{1})^i = \operatorname{dim}\operatorname{ker}(\tau(\zeta) - \textbf{1})^{i+1}\bigr\} =
\operatorname{min} \bigl\{i\bigl|\operatorname{ker}(\tau(\zeta) - \textbf{1})^i = V \bigr\}.
\end{equation*}
(Here, $V$ is the verctor space underlying $\tau$ and we use the convention
that $f^0$ is the identity map for any linear map $f$.) It follows easily from the characterization
of $\mathcal{I}^{\tt{uni}}$ in (\ref{20150810_Nabla_line}) that $s_i \geq s_{i+1}$, i.e. that $\Theta$ has
values in $\mathcal{Y}_n$. 

It is an easy combinatorial calculation to check that $\tau$ is uniquely characterized by its value under $\Theta$ and that
each Young diagram occurs as a kernel sequence (i.e. that $\Theta$ is a bijection). More precisely, we have
\begin{Lem}\label{20150812_young_conjugation_lemma}
The map $\Theta\circ\nabla^{-1}:\mathcal{Y}_n \rightarrow \mathcal{Y}_n$ is given by the conjugation operation
on Young 
diagrams (cf. \cite[\S 4.1]{FultonHarris} or \cite[Section 2.8]{HarrisMichael}). In particular, for a given $\tau\in  \mathcal{I}^{\tt{uni}}$,
the block matrix structure of $\tau(\zeta)$ (up to reordering blocks) as in (\ref{20150810_Nabla_line}) determines its kernel sequence and vice versa.
\end{Lem}
\begin{proof}
Retaining the notation used in (\ref{20150810_Nabla_line}), we first remark that for $i\in\mathbb{N}_0$ we have
\begin{equation*}
\operatorname{dim}\operatorname{ker}\mathcal{B}_m^i = \operatorname{min}(i,m).
\end{equation*}
Thus, setting $\mathcal{B}=\operatorname{diag}(\mathcal{B}_{l_1}, \ldots, \mathcal{B}_{l_k})$, we get
\begin{equation*}
\operatorname{dim}\operatorname{ker}\mathcal{B}^i = \sum_{j=1}^{k} \operatorname{min}(i,l_j).
\end{equation*}
Consequently the kernel sequence $(s_1,\ldots, s_r)$ associated to $(l_1, \ldots, l_k)$ is given by
\begin{equation*}
s_i = \sum_{j=1}^{k} \operatorname{min}(i,l_j)- \operatorname{min}(i-1,l_j)
= \#\{j|l_j\geq i\} = \operatorname{max}\{j|l_j\geq i\}.
\end{equation*}
and
\begin{equation*}
r = \operatorname{max}\{l_j|j=1\ldots k\} = l_1.
\end{equation*}
Hence, the transition $(l_1, \ldots, l_k) \leadsto (s_1,\ldots, s_r)$ is precisely the conjugation operation
of reflecting a Young diagram at the main diagonal (cf. \cite[Section 2.8]{HarrisMichael}), e.g. 
\begin{equation*}
\small\Yvcentermath1 \yng(3,1) \qquad\qquad\leadsto\qquad\qquad \yng(2,1,1)
\end{equation*}\qedhere
\end{proof}

In order to state the desired comparison result, let us recall that we consider a residual representation 
$\overline{\rho}: \operatorname{Gal}_K \rightarrow \operatorname{GL}_n(k)$ with unipotent ramification.
Let $\underline{\lambda}= (l_1, \ldots, l_k) \in \mathcal{Y}_n$ such that $\overline{\rho}(\zeta) \sim 
\textbf{1}+\operatorname{diag}(\mathcal{B}_{l_1}, \ldots, \mathcal{B}_{l_k})$. Let $\tau= \nabla(\underline{\lambda})
\in \mathcal{I}^{\tt{uni}}$. 
\begin{Thm}\label{20150813_min_type_comparison_theorem} Assume $\overline{\rho}$ is unipotently ramified and $\tau$ as above. Then there is an isomorphism of the quotients 
\begin{equation*}
R^{\boxempty,\tau}_{\Lambda}(\overline{\rho}) \cong
R^{\boxempty, \tt{min}}_{\Lambda}(\overline{\rho})\cong
\Lambda\llbracket X_1, \ldots, X_{n^2}\rrbracket 
\end{equation*}
of $R^{\boxempty}_{\Lambda}(\overline{\rho})$, i.e. a lifting of $\overline{\rho}$ is minimally
ramified if and only if it is of type $\tau$.
\end{Thm}
\begin{proof}
The diagram 
\begin{equation*}
\xymatrix{
&&R^{\boxempty, \tt{min}}_{\Lambda}(\overline{\rho})\ar@{->}[drr]\\
R^{\boxempty}_{\Lambda}(\overline{\rho})\ar@{->}[urr]\ar@{->}[drr]&&&&\overline{E}\\
&&R^{\boxempty, \tau}_{\Lambda}(\overline{\rho})\ar@{->}[urr]\\
}
\end{equation*}
allows us to consider the $\overline{E}$-points of $\operatorname{Spec}R^{\boxempty, \tt{min}}_{\Lambda}(\overline{\rho})$ and $\operatorname{Spec}R^{\boxempty, \tau}_{\Lambda}(\overline{\rho})$
as subsets of the $\overline{E}$-points of $\operatorname{Spec}R^{\boxempty}_{\Lambda}(\overline{\rho})$.
We claim that they are equal: Translated into terms of $\overline{E}$-valued representations, we have
to compare the sets
\begin{equation*}
\Xi^{\tt{min}}=\Bigl\{\rho: \operatorname{Gal}_K\rightarrow \operatorname{GL}_n(\overline{E})\;\Bigl|\; 
\substack{{\rho \text{ lifts } \overline{\rho} \text{ and has values in }\mathcal{O}_{\overline{E}},} 
\\ \operatorname{dim}\operatorname{ker}(\rho(\zeta) - \textbf{1})^{i-1} - \operatorname{dim}\operatorname{ker}(\rho(\zeta) - \textbf{1})^{i} = l_i \; \forall i} \Bigr\}
\end{equation*}
and 
\begin{equation*}
\Xi^{\tau}=\Bigl\{\rho: \operatorname{Gal}_K\rightarrow \operatorname{GL}_n(\overline{E})\;\Bigl|\; 
\substack{{\rho \text{ lifts } \overline{\rho} \text{ and has values in }\mathcal{O}_{\overline{E}},} 
\\ {\rho|I_K \cong \tau}} \Bigr\}.
\end{equation*}
Lemma \ref{20150812_young_conjugation_lemma} implies that 
$\Xi^{\tt{min}} = \Xi^{\tau}$.

Now by definition of the ring $R^{\boxempty, \tau}_{\Lambda}(\overline{\rho})$ (as the
schematic closure of the points in $\Xi^{\tau}$) we have
\begin{equation*}
\operatorname{ker}\bigl(R^{\boxempty}_{\Lambda}(\overline{\rho}) \rightarrow
R^{\boxempty, \tau}_{\Lambda}(\overline{\rho})\bigr)=\bigcap_{\rho\in \Xi^{\tau}}\operatorname{ker}(\rho).
\end{equation*}

Moreover, we clearly have 
\begin{equation*}
\operatorname{ker}\bigl(R^{\boxempty}_{\Lambda}(\overline{\rho}) \rightarrow
R^{\boxempty, \tt{min}}_{\Lambda}(\overline{\rho})\bigr)\subseteq\bigcap_{\rho\in \Xi^{\tt{min}}}\operatorname{ker}(\rho).
\end{equation*}
Hence, by $\Xi^{\tau}= \Xi^{\tt{min}}$ we get a factorization
\begin{equation*}
R^{\boxempty}_{\Lambda}(\overline{\rho}) \twoheadrightarrow
R^{\boxempty, \tt{min}}_{\Lambda}(\overline{\rho})
\overset{\varphi}{\twoheadrightarrow}
R^{\boxempty, \tau}_{\Lambda}(\overline{\rho})
\end{equation*}
where the middle and the right ring have the same spectrum as topological spaces. 
Now we know by Proposition \ref{20170526_newProponmin} that $R^{\boxempty, \tt{min}}_{\Lambda}(\overline{\rho})$ 
is formally smooth over $\Lambda$ of relative dimension $n^2$ and that 
$\operatorname{dim}R^{\boxempty, \tau}_{\Lambda}(\overline{\rho}) = n^2 + 1$ (combine Theorem 2.4
and Proposition 2.15 of \cite{Shotton_phd}). 
Thus, $\varphi$ is an isomorphism by Proposition \ref{20170324_some_prop} and the claim follows.
\end{proof}

\subsection{Taylors deformation condition $(1,\ldots, 1)$ ($\ell \neq p$)}
We continue to consider a unipotently ramified residual representation
$\overline{\rho}: \operatorname{Gal}_K \rightarrow \operatorname{GL}_n(k)$. If $A\in 
\mathcal{C}_{\mathcal{O}}$ is a coefficient ring, we say that an $A$-valued lift $\rho$ 
of $\overline{\rho}$ fulfills the condition $(1, \ldots, 1)$ if $\operatorname{charPoly}(\rho(\xi)) = (T-1)^n$ for all $\zeta\in I_K$. By our assumption that $\overline{\rho}$ is
unipotently ramified, it is sufficient to check the case where $\xi$ is a topological
generator of the tame inertia. This defines a deformation condition (and, in comparison to
\cite{Taylor2}, we don't assume that $\overline{\rho}$ is trivial, cf. 
\cite[Remark before Proposition 3.17]{Thorne_small}).
%
%
\begin{Prop}
If a lift $\rho$ is minimally ramified, it fulfills the Taylor condition. In particular,
there is a canonical surjection 
\begin{equation*}
R^{\boxempty,(1,\ldots, 1)}(\overline{\rho}) \twoheadrightarrow R^{\boxempty,\tt{min}}(\overline{\rho}),
\end{equation*}
and a morphism $R^{\boxempty,(1,\ldots, 1)}(\overline{\rho}) \rightarrow A$ factors through
this surjection
if and only if the associated $A$-valued lift of $\overline{\rho}$ is minimally 
ramified.
\end{Prop}
\begin{proof}
By the unipotency assumption, we can assume that $\overline{\rho}|P_{K}$ is 
trivial and $\overline{\rho}(\zeta)$ is upper-triangular with each diagonal entry equal 
to $1$ (where $\zeta$ is a topological generator of $I_{K}/P_{K}$). If a lift $\rho$ 
is minimal, it follows that $\rho|P_{K}$ is trivial and that
$\rho(\zeta)$ is unipotent, cf. [CHT, Lemma 2.4.15, Assertion $3. \Rightarrow 1.$].
It follows that $\rho(\sigma)$ is unipotent for any $\sigma\in I_{K}$. This 
proves the claim.
\end{proof}
The easy proof of the following Proposition is left to the reader (cf. \cite[Proof of Proposition 3.8]{KW_on_serre}).
\begin{Prop}\label{93hc0k320}
Let $L$ be a finite extension 
of $K$. Let 
\begin{equation*}
\rho^{\boxempty, (1,\ldots, 1)}: G_K \rightarrow 
\operatorname{GL}_n(R^{\boxempty, (1,\ldots, 1)}(\overline{\rho}))
\end{equation*}
be the universal lifting of $\overline{\rho}$ with respect to the condition $(1,\ldots, 1)$ and let 
\begin{equation*}
\rho^{\boxempty, (1,\ldots, 1)}_L: G_L \rightarrow 
\operatorname{GL}_n(R^{\boxempty, (1,\ldots, 1)}(\overline{\rho}|G_L))
\end{equation*}
be the universal lifting of $\overline{\rho}|G_L$ with respect to the condition $(1,\ldots, 1)$.
Then there exists a unique morphism of $\mathcal{C}_W$-algebras $\varphi: 
R^{\boxempty, (1,\ldots, 1)}(\overline{\rho}|G_L)/(\ell) \rightarrow
R^{\boxempty, (1,\ldots, 1)}(\overline{\rho})/(\ell)$ such that 
\begin{equation*}
\overline{\rho^{\boxempty, (1,\ldots, 1)}}|G_L = 
\varphi \circ \overline{\rho^{\boxempty, (1,\ldots, 1)}_L}.
\end{equation*}
\end{Prop}

\begin{Lem} \label{20170213_Lemma_min_obsolet_if_l_big}
Let $\tilde{\rho}$ be an $A$-valued lift, where we assume that $A$ is reduced. 
Write $X=\tilde{\rho}(\zeta)$.
Then $\chi_X:=\operatorname{charPoly}(X)$ equals $(T-1)^n$ if $\ell \geq q^{n!}$.
\end{Lem}   
\begin{proof}
Assume first that $A$ is an integral domain. By the condition $\varphi X \varphi^{-1} = 
X^q$ we see that raising to the $q$-th power permutes the eigenvalues of $X$ (understood
as a list of $n$ elements).
Thus, any eigenvalue of $X$ must be a $(q^{\#S_n}-1) = (q^{n!}-1)$-th root of unity.
Thus, if $Q(\mu)$ denotes the decomposition field of
the polynomial $f(T) = T^{q^{n!}-1} -1$ over the quotient field
of $A$ and $A(\mu)$ denotes the integral closure of $A$ in $Q(\mu)$, then $\chi_X$ 
decomposes completely in $A(\mu)[T]$. On the other hand, each eigenvalue of $X$ is sent to $1$ by
the canonical reduction map
\begin{equation*}
\pi': A(\mu)= A(\mu)\otimes_A A \rightarrow A(\mu)\otimes_A k.
\end{equation*}
As the kernel of $\pi'$ is a pro-$\ell$-subgroup and as $(\ell^m, q^{n!}-1) = 1$ for any
$m\in \mathbb{N}$, it follows that any eigenvalue of $X$ is $1$, i.e. that
$\chi_X = (T-1)^n$. The result for a general (reduced) $A$
follows easily from using the embedding
\begin{equation*}
A \hookrightarrow \prod_{\mathfrak{q}}  A/\mathfrak{q},
\end{equation*}
where $\mathfrak{q}$ runs through the minimal primes of $A$.
\end{proof}
\begin{Cor}\label{20170210_cor_on_taylor_cond}
If $\ell \geq q^{n!}$, then 
$R^{\boxempty, (1, \ldots, 1)}(\overline{\rho}) = R^{\boxempty}(\overline{\rho})$. In particular, $R^{\boxempty, (1, \ldots, 1)}(\overline{\rho})$ is reduced (cf. Theorem
\ref{20170213_Helms_theorem}). 
\end{Cor}
\begin{proof}
By Lemma \ref{20170213_Lemma_min_obsolet_if_l_big} (together with Theorem
\ref{20170213_Helms_theorem}), we see that the identity map on 
$R^{\boxempty}(\overline{\rho})$ factorizes through
$R^{\boxempty, (1, \ldots, 1)}(\overline{\rho})$. On the other hand, 
$R^{\boxempty, (1, \ldots, 1)}(\overline{\rho})$ is by definition a quotient
of $R^{\boxempty}(\overline{\rho})$. Thus, we have found a surjective endomorphism
of $R^{\boxempty, (1, \ldots, 1)}(\overline{\rho})$ (which must then be an isomorphism, as the rings in question are noetherian) which factorizes via 
$R(\overline{\rho})^{\boxempty}$. This proves the claim.
\end{proof}
%
%
\section{On automorphic forms on unitary groups}
\subsection{The group $\mathcal{G}_n$}\label{20170407_section_grp_gn}
For $n\in \mathbb{N}$ recall from \cite[Section 2.1]{CHT} the definition of the group scheme 
$\mathcal{G}_n$ over $\mathbb{Z}$ and the multiplier character $\texttt{m}:\mathcal{G}_n
\rightarrow \operatorname{GL}_1$. We write $\mathcal{G}_n^0$ for the connected component 
of the identity and $\mathfrak{g}_n$ for the Lie algebra of $\mathcal{G}_n$ (where we differ
in notation from \cite{CHT}). We have $\mathcal{G}_n^{\tt{der}}\cong \operatorname{GL}_n$ 
and $\mathcal{G}_n^{\tt{ab}}\cong \operatorname{GL}_1\times \mathbb{Z}/2\mathbb{Z}$.
If $F$ is a CM-field with totally real subfield $F^+$, recall in particular the connection
between $\operatorname{GL}_n$-valued conjugate self-dual representations of $\operatorname{Gal}_F$ and 
$\mathcal{G}_n$-valued representations of $\operatorname{Gal}_{F^+}$, cf. \cite[Lemma 1.1.4]{CHT} or 
\cite[Lemma 5.1.1]{GeePrescribed}. 

We will be particularly interested in deformations of $\mathcal{G}_n$-valued residual 
representations. In the local split case, there is a substantial simplification 
possible: Let $k$ be a finite field and let $\overline{\rho}$ be a 
$\operatorname{GL}_n$-valued representation of $\operatorname{Gal}_F$, let $\overline{\chi}$
a character such that $\overline{\chi}\overline{\rho}^{\vee} \cong \overline{\rho}^c$ and let
$\overline{r}$ be the associated $\mathcal{G}_n(k)$-valued representation of 
$\operatorname{Gal}_{F^+}$. Moreover, let $\Lambda$ be the ring of integers of a 
finite extension of the quotient field of $W(k)$. The following proposition now follows 
easily from the definitions:
\begin{Prop}\label{20170214_prop_local_split_defos}
Let $\nu$ be a place of $F^+$ which splits as $\tilde{\nu}\tilde{\nu}^c$ in $F$. Denote $\overline{r}_{\nu}:=\overline{r}|\operatorname{Gal}_{F^+_{\nu}}$ and 
$\overline{\rho}_{\tilde{\nu}}:=\overline{\rho}|\operatorname{Gal}_{F_{\tilde{\nu}}}$.
Fix a lift $\chi_{\nu}: \operatorname{Gal}_{F^+_{\nu}} \rightarrow \Lambda^{\times}$ of
$\mathtt{m}\circ \overline{r}_{\nu}$. Then
\begin{equation*}
R^{(\boxempty), \chi_{\nu}}_{\Lambda}(\overline{r}_{\nu}) \cong
R^{(\boxempty)}_{\Lambda}(\overline{\rho}_{\tilde{\nu}}) \text{ and }
H^i(F^+_{\nu}, \mathfrak{g}_n^{\tt{der}}) \cong  
H^i(F_{\tilde{\nu}}, \mathfrak{gl}_n),
Z^1(F^+_{\nu}, \mathfrak{g}_n^{\tt{der}}) \cong  
Z^1(F_{\tilde{\nu}}, \mathfrak{gl}_n).
\end{equation*}
\end{Prop}
This observation allows us to define local conditions for deformations of $\overline{r}$ at split places by $\operatorname{GL}_n$-valued local conditions. In order to make this 
precise, let $\Sigma\subset \operatorname{Pl}_{F^+}^{\textnormal{fin}}$ be a finite set of places and 
assume that any place in $\Sigma$ splits as $\nu = \tilde{\nu}\tilde{\nu}^c$ 
in the extension $F|F^+$ (so, in 
particular, we fix a place $\tilde{\nu}$ above $\nu$). Moreover, assume that $\overline{r}$ is unramified outside $\Sigma$, i.e. factorizes through $\operatorname{Gal}_{F^+, \Sigma}$. We set $\widetilde{\Sigma} := \{ \tilde{\nu}|\nu \in \Sigma\}$.
Fix a character $\chi: \operatorname{Gal}_{F^+, \Sigma} \rightarrow 
\Lambda^{\times}$ lifting $\texttt{m}\circ \overline{r}$.
Moreover, for 
each $\tilde{\nu}\in \widetilde{\Sigma}$ fix a deformation condition $D_{\nu}$ of 
the $\operatorname{GL}_n$-valued representation $\overline{\rho}_{\tilde{\nu}}$.
\begin{Def}[Deformation problem, following \cite{CHT}]
The collection $$\mathscr{S}= (F|F^+, \Sigma, \widetilde{\Sigma}, \Lambda, \overline{r}, 
\chi, \{D_{\nu}\}_{\nu\in \Sigma}),$$ parametrizing deformations $r$ of $\overline{r}$ 
to $\mathcal{C}_{\Lambda}$ which fulfill $\texttt{m}\circ r = \chi$, which are unramified outside $\Sigma$ and fulfill $D_{\nu}$ (via Proposition \ref{20170214_prop_local_split_defos}) at $\nu\in \Sigma$, defines a global deformation condition.
\end{Def}

We end this section by a remark on the conventions for multiple framings, in which we
differ from \cite{CHT}. For this, let $T\subset \Sigma$ be a non-empty subset and recall 
our Definition \ref{20170328_multiply_framed_defo_functor} for the multiply framed deformation functor $D^{\boxempty_T, \mathcal{S}}_{\Lambda}(\overline{r})$ and its representing object $R^{\boxempty_T, \mathcal{S}}_{\Lambda}(\overline{r})$.
Comparing this with the functor and representing object considered in 
\cite[Definition 2.2.7]{CHT}, which we denote by $D^{\boxdot_T, \mathcal{S}}_{\Lambda}(\overline{r})$ and $R^{\boxdot_T, \mathcal{S}}_{\Lambda}(\overline{r})$,
we easily get the following observation:
%
%
\begin{Prop}
$D^{\boxempty_T, \mathcal{S}}_{\Lambda}(\overline{r})$ is representable if and only if 
$D^{\boxdot_T, \mathcal{S}}_{\Lambda}(\overline{r})$ is representable, and in this case
we have
\begin{equation*}
R^{\boxempty_T, \mathcal{S}}_{\Lambda}(\overline{r}) \cong
R^{\boxdot_T, \mathcal{S}}_{\Lambda}(\overline{r})\llbracket X_1, \ldots, X_{\#T}\rrbracket.
\end{equation*}
\end{Prop}
\subsection{Automorphic forms and Hecke algebras}\label{20170301_sectn_hecke_algs}
For this subsection, let us assume
that the extension $F|F^+$ is unramified at all finite
places and, in case $n$ is even, that $\frac{n}{2}[F^+:\mathbb{Q}]$ is even. This allows 
us to fix a definite unitary group $H$ over $\mathcal{O}_{F^+}$, as considered in
\cite[Section 2.11]{Guerberoff} or \cite[Section 1.1]{Geraghty_thesis}, whose key properties we recall here: 
\begin{itemize}
\item The extension of scalars of $H$ to $F^+$ is an outer form of $\operatorname{GL}_n/F^+$, 
which becomes isomorphic to $\operatorname{GL}_n/F$ after extending scalars to $F$;
\item $H$ is quasi-split at every finite place of $F^+$;
\item $H$ is totally definite, i.e. $H(F_{\infty}^+)$ is compact and $H(F_{\nu}^+)\cong
U_n(\mathbb{R})$ for all infinite places $\nu$ of $F^+$;
\item For any finite place $\nu$ of $F^+$ which splits as $\tilde{\nu}\tilde{\nu}^c$ in $F$, 
we can choose an isomorphism $\iota_{\tilde{\nu}}: H(F^+_{\nu}) \rightarrow 
\operatorname{GL}_n(F_{\tilde{\nu}})$ whose restriction to $H(\mathcal{O}_{F^+_{\nu}})$ 
provides an isomorphism $H(\mathcal{O}_{F^+_{\nu}}) \cong \operatorname{GL}_n(\mathcal{O}_{F_{\tilde{\nu}}})$.
\end{itemize}
\paragraph{Level subgroups} Let us fix a finite subset $\mathcal{T} \subset\operatorname{Pl}_{F^+}^{\text{fin}}$ such that each $\nu\in \mathcal{T}$ 
splits as $\tilde{\nu}\tilde{\nu}^c$ in $F$.
%
For the remainder of this section, the letter $U$ will denote an open compact subgroup
of $H(\mathbb{A}^{\infty}_{F^+})$. For later applications, we will be particularly interested
in the choice $U_{\mathcal{T}} := \prod_{\nu\in \operatorname{Pl}^{\text{fin}}_{F^+}} U_{\nu}$ with:
\begin{itemize}
\item If $\nu$ is not split in $F|F^+$, then $U_{\nu}$ is a hyperspecial maximal compact
subgroup of $H(F^+_{\nu})$;
\item If $\nu\notin \mathcal{T}$ splits, then $U_{\nu} = H(\mathcal{O}_{F^+_{\nu}})$;
\item If $\nu\in \mathcal{T}$, then $U_{\nu} = \iota_{\tilde{\nu}}^{-1}(\operatorname{Iw})$, where $\operatorname{Iw}\subset \operatorname{GL}_n(\mathcal{O}_{F_{\tilde{\nu}}})$ denotes the Iwahori subgroup.
\end{itemize}
We remark that in many articles (e.g. \cite{CHT})
the set $\mathcal{T}$ is 
enlarged by a choice of auxiliary places at which a suitable
level condition is imposed. Our arguments don't require
such auxiliry places.
\paragraph{Weights} Recall the parametrization of complex and $\ell$-adic representations of unitary and general linear groups, e.g. from \cite{Guerberoff}:
\begin{itemize}
\item To a tuple $\boldsymbol\omega=(\underline{\omega}_{\tau}) \in 
(\mathbb{Z}^{n,+})^{\operatorname{Hom}(F^+, \mathbb{R})}$ we associate the representation
\begin{equation*}
\xi^u_{\boldsymbol\omega}: H(F^+_{\infty}) = \prod_{\tau\in \operatorname{Hom}(F^+, \mathbb{R})} H(F^+_{\tau}) \cong \prod_{\tau\in \operatorname{Hom}(F^+, \mathbb{R})} U_n(\mathbb{R}) \overset{\varphi}{\rightarrow} \prod_{\tau\in \operatorname{Hom}(F^+, \mathbb{R})} \operatorname{GL}_n(W^u_{\underline{\omega}_{\tau}}) \subset
\operatorname{GL}_n(W^u_{\boldsymbol\omega}),
\end{equation*}
where $W^u_{\boldsymbol\omega} = \otimes_{\tau} W^u_{\underline{\omega}_{\tau}}$ and where 
$\varphi$ is the product of the highest weight representations $W^u_{\underline{\omega}_{\tau}}$ attached to the weight $\underline{\omega}_{\tau}$ (see e.g. \cite{BellaicheChenevier, Guerberoff, Geraghty_thesis}).
\item Let $\ell$ be a rational prime such that every place $\nu$ of $F^+$ above $\ell$ 
splits in $F|F^+$ and fix for each such $\nu$ a place $\tilde{\nu}$ of $F$ above $\nu$. Let 
$\mathcal{K}$ be a finite extension of $\mathbb{Q}_{\ell}$ which is $F$-big enough and
let $\boldsymbol\omega=(\underline{\omega}_{\tau}) \in 
(\mathbb{Z}^{n,+})^{\operatorname{Hom}(F, \mathcal{K})}$. To each $\tau\in \operatorname{Hom}(F, \mathcal{K})$ we can associate a place $\nu$ of $F^+$ above $\ell$ for which we have just fixed a place $\tilde{\nu}$. Denote this assignment $\operatorname{Hom}(F, \mathcal{K}) \rightarrow \Omega_{\ell}^F$ by $\tau\mapsto w_{\tau}$. Let 
\begin{equation*}
\xi_{\boldsymbol\omega}^{\mathcal{K}} : \prod_{\nu \in \Omega_{\ell}^F} H(F^+_{\nu}) 
\cong  \prod_{\nu \in \Omega_{\ell}^F} \operatorname{GL}_n(F_{\tilde{\nu}}) \overset{\prod d_{\nu}}{\longrightarrow} \prod_{\nu \in \Omega_{\ell}^F} \prod_{\substack{\tau \in \operatorname{Hom}(F, \mathcal{K})\\ w_{\tau} = \tilde{\nu}}} \operatorname{GL}_n(F_{\tilde{\nu}}) = 
\prod_{\tau \in \operatorname{Hom}(F, \mathcal{K})} \operatorname{GL}_n(F_{\tilde{\nu}})
\end{equation*}
\begin{equation*}
\qquad\qquad\qquad\qquad\qquad\qquad\qquad \overset{\psi}{\longrightarrow} 
\prod_{\tau \in \operatorname{Hom}(F, \mathcal{K})} \operatorname{GL}_n(W^{\mathcal{K}}_{\underline{\omega}_{\tau}}) \subset \operatorname{GL}_n(W^{\mathcal{K}}_{\boldsymbol\omega})
\end{equation*}
be the representation where each $d_{\nu}$ is the diagonal embedding, where $W^{\mathcal{K}}_{\boldsymbol\omega} = \otimes_{\tau} W^{\mathcal{K}}_{\underline{\omega}_{\tau}}$ and where 
$\psi$ is the product of the highest weight representations $W^{\mathcal{K}}_{\underline{\omega}_{\tau}}$ attached to the weight $\underline{\omega}_{\tau}$.
The representation $\xi_{\boldsymbol\omega}^{\mathcal{K}}$ admits an integral model over $\mathcal{O}_{\mathcal{K}}$, whose underlying finite free $\mathcal{O}_{\mathcal{K}}$-module
we denote by $M_{\boldsymbol\omega}^{\mathcal{O}_{\mathcal{K}}}$.
\end{itemize}

\paragraph{Automorphic forms} We denote by 
\begin{equation*}
\mathcal{A}(H) = \bigoplus_{\pi} \pi^{m(\pi)}
\end{equation*}
the space of (complex) automorphic forms on $H$, which decomposes into isomorphism
classes of irreducible representations of $H(\mathbb{A}_{F^+})$, each occurring with 
finite multiplicity $m(\pi)$ (see e.g. \cite{Guerberoff}).
\begin{Def}[Vector-valued automorphic form]
Let $\boldsymbol\omega\in (\mathbb{Z}^{n,+})^{\operatorname{Hom}(F^+, \mathbb{R})}$ be a 
weight, then we denote by $\mathcal{S}_{\boldsymbol\omega}$ the space of locally constant
functions $f:H(\mathbb{A}_{F^+}^{\infty}) \rightarrow W^{u,\vee}_{\boldsymbol\omega}$ which 
fulfill
\begin{equation*}
f(\gamma.h) = \gamma_{\infty} f(h) \qquad \forall h\in H(\mathbb{A}_{F^+}^{\infty}), 
\gamma \in H(F^+).
\end{equation*}
(We denote by $\gamma_{\infty}$ the image of $\gamma$ under the canonical embedding 
$H(F^+) \rightarrow H(F^+_{\infty})$.). $H(\mathbb{A}^{\infty}_{F^+})$ acts on 
$\mathcal{S}_{\boldsymbol\omega}$ by right translation, and for a level subgroup
$U$ we denote by $\mathcal{S}_{\boldsymbol\omega}(U)$ the space of $U$-fixed vectors.
\end{Def}

There exists an $H(\mathbb{A}_{F^+}) = H(\mathbb{A}_{F^+,\infty})\times 
H(\mathbb{A}_{F^+}^{\infty})$-equivariant decomposition
\begin{equation*}
\mathcal{A}(H) = \bigoplus_{\boldsymbol\omega} W^u_{\boldsymbol\omega} \otimes 
\mathcal{S}_{\boldsymbol\omega}.
\end{equation*}
Thus we can associate to an $f\in \mathcal{S}_{\boldsymbol\omega}$ the (irreducible) 
automorphic representation $\langle f \rangle$ which is uniquely characterized by the
condition that it contains all vectors of $W^u_{\boldsymbol\omega}\otimes f$. The main feature
of the group $H$ is the existence of avatars:
\begin{Thm}
Let $\Pi$ be a RACSDC automorphic representation of $\operatorname{GL}_n(\mathbb{A}_F)$ of 
weight $\boldsymbol\omega \in (\mathbb{Z}^{n,+})^{\operatorname{Hom}(F,\mathbb{C})}$ in the
sense of \cite[Section 4]{CHT}. Then there exists an automorphic representation $\pi_0$ 
of $H(\mathbb{A}_{F^+})$ such that $\Pi$ is a base change of $\pi_0$, i.e.
\begin{itemize}
\item for each archimedean place $\nu$ of $F^+$ and each place $\tilde{\nu}$ of $F$ above
$\nu$, we have $\pi_{0, \nu} \cong \xi^u_{\underline{\omega}_{\tilde{\nu}}}$;
\item for each finite place $\nu$ of $F^+$ which splits as $\tilde{\nu}\tilde{\nu}^c$ in 
$F$, $\Pi_{\tilde{\nu}}$ is the local base change of $\pi_{0,\nu}$;
\item if $\nu$ is a finite place of $F^+$ which stays inert in $F$ and for which 
$\Pi_{\nu}$ is unramified, then $\pi_{\nu}$ has a fixed vector for a maximal hyperspecial
compact subgroup of $H(F_{\nu}^+)$.
\end{itemize}
\end{Thm}
\begin{proof}
See \cite[Theorem 2.2]{Guerberoff} and \cite[Lemma 2.2.7]{Geraghty_thesis}.
\end{proof}
\paragraph{Hecke algebras} We continue to consider a fixed set of places $\mathcal{T}$ as 
above (with corresponding level subgroup $U = U_{\mathcal{T}}$) and a weight $\boldsymbol\omega$. For $j\in \{1, \ldots, n\}$ and 
for $w$ a place of $F$ which is split over $F^+$ and does not divide an element 
of $\mathcal{T}$, we consider the following Hecke operator (acting on 
$\mathcal{S}_{\boldsymbol\omega}(U)$):
\begin{equation*}
T^{(j)}_{F_{w}} = \left[ U.\iota^{-1}_w\begin{pmatrix}
\varpi_{F_w}\textbf{1}_j & 0 \\ 0& \textbf{1}_{n-j}
\end{pmatrix}.U\right]
\end{equation*}
For a finite set $\mathcal{T}'\subset \operatorname{Pl}_{F^+}^{\text{fin}}$ containing 
$\mathcal{T}$ and a subring $\mathscr{R}$ of $\mathbb{C}$ we define the Hecke algebra
\begin{equation*}
^{\mathscr{R}}\textbf{T}^{\mathcal{T}'}_{\boldsymbol\omega}(U) := 
\operatorname{im}\left(
\mathscr{R}[T_{F_{w}}^{(j)} | j\in \{1, \ldots, n\}, w\in \operatorname{Pl}_F^{\tt{split}, \mathcal{T}'}] \rightarrow \operatorname{End}_{\mathbb{C}}(\mathcal{S}_{\boldsymbol\omega}(U))
\right),
\end{equation*}
where $\operatorname{Pl}_F^{\tt{split}, \mathcal{T}'}$ denotes the set of places of $F$ 
which are split over $F^+$ and which do not divide an element of $\mathcal{T}'$. Besides the case $\mathscr{R}= \mathbb{Z}$ we will be interested in $\mathscr{R} = \mathcal{E}_f$ (the coefficient field of an eigenform $f$ with respect to $^{\mathbb{Z}}\textbf{T}^{\mathcal{T}}_{\boldsymbol\omega}(U)$)) and in $\mathscr{R} = \mathcal{E}(U) = \prod_f \mathcal{E}_f$, 
where the product (i.e. the field compositum operation) runs through all 
eigenforms of $\mathcal{S}_{\boldsymbol\omega}(U)$. We note the following well-known facts:
There are only finitely many (one-dimensional) eigenspaces $\mathbb{C}.f_1, \ldots, 
\mathbb{C}.f_r$ contained in $\mathcal{S}_{\boldsymbol\omega}(U)$, so $\mathcal{E}(U)$ is a number field. Moreover, $\mathcal{S}_{\boldsymbol\omega}(U)$ admits a basis of eigenforms, i.e. 
we can choose the $f_i$ such that
\begin{equation*}
\mathcal{S}_{\boldsymbol\omega}(U)\cong
\mathbb{C}.f_1 \oplus \ldots \oplus \mathbb{C}.f_r
\end{equation*}
as a $\textbf{T}^{\mathcal{T}}_{\boldsymbol\omega}(U)$-module
(see decomposition (3.1.1) of \cite{Guerberoff}). By mapping a 
Hecke operator to its $f$-eigenvalue, any eigenform $f\in \mathcal{S}_{\boldsymbol\omega}(U)$ gives rise to a $\mathbb{Z}$-algebra-homomorphism
\begin{equation*}
\varphi_f: \;^{\mathbb{Z}}\mathbf{T}^{\mathcal{T}}_{\boldsymbol\omega}(U) \longrightarrow
\mathcal{E}(U) \qquad T^{(j)}_{F_{w}} \longmapsto a_f(T^{(j)}_{F_w})
\end{equation*}
and it can be shown that $\operatorname{im}(\varphi_f)\subset \mathcal{O}_{\mathcal{E}(U)}$.
The form $f$ is uniquely characterized by $\varphi_f$ (up to $\mathbb{C}$-multiples).

\paragraph{$\ell$-adic models of automorphic forms} The following is based on Section
2.3 of [Gue11]. For this paragraph, we fix a rational prime $\ell$ which does not 
lie below $\mathcal{T}$ and such that all places of $F^+$ above $\ell$ are split
in the extension $F|F^+$ and consider the
following setup: Let $\mathcal{K}$ be a finite extension of $\mathbb{Q}_{\ell}$ 
which is $F$-big enough and fix an isomorphism $\iota: \overline{\mathcal{K}} \cong 
\mathbb{C}$. Moreover, we fix an $\ell$-adic weight $\boldsymbol\omega$, i.e. an element 
of 
\begin{equation*}
(\mathbb{Z}^{n,+})_c^{\operatorname{Hom}(F, \mathcal{K})} =
\left\{
\boldsymbol\omega \in (\mathbb{Z}^{n,+})^{\operatorname{Hom}(F, \mathcal{K})}\;\Bigl|\;
\underline{\omega}_{\tau^c, i} = -\underline{\omega}_{\tau, n-i+1} \; \forall
\tau\in \operatorname{Hom}(F, \mathcal{K}), i\in \{1, \ldots, n\}
\right\}.
\end{equation*}
\begin{Def}
For $U\subset H(\mathbb{A}^{\infty}_{F^+})$ a compact subgroup and an $\mathcal{O}_{\mathcal{K}}$-algebra $A$, suppose that either the projection of $U$ to $H(F^+_{\ell})$ is contained in $H(\mathcal{O}_{F^+_{\ell}})$ or that $A$ is a $\mathcal{K}$-algebra. Then we define 
\begin{equation*}
S_{\boldsymbol\omega}(U,A) = \left\{
f: H(F^+)\backslash H(\mathbb{A}_{F^+}^{\infty}) \rightarrow A\otimes_{\mathcal{O}_K} M^{\mathcal{O}_{\mathcal{K}}}_{\boldsymbol\omega}
\;\bigl|\; u_{\ell}.f(hu) = f(h)\; \forall u\in U, h\in H(\mathbb{A}^{\infty}_{F^+})
\right\},
\end{equation*}
where $u_{\ell}$ denotes the image of $u$ under the projection map $H(\mathbb{A}_{F^+}^{\infty}) \rightarrow H(F^{+}_{\ell})$.
\end{Def}
We are primarily interested in the case 
that $A$ is $\mathcal{O}_{\mathcal{K}}$-flat, so that
%
we have $S_{\boldsymbol\omega}(U,A) \cong A \otimes_{\mathcal{O}_{\mathcal{K}}}
S_{\boldsymbol\omega}(U,\mathcal{O}_{\mathcal{K}})$. 

The main connection with complex automorphic forms is as follows (cf. also \cite[Section 2.3]{Guerberoff}): The isomorphism $\iota$ 
gives rise to a bijection $\iota^+_{\ast}: (\mathbb{Z}^{n,+})_c^{\operatorname{Hom}(F, \mathcal{K})} \cong (\mathbb{Z}^{n,+})^{\operatorname{Hom}(F, \mathbb{R})}$, and the assignment 
$f \mapsto (h\mapsto \theta_{\boldsymbol\omega}(h_{\ell}.f(h)))$ provides isomorphisms 
of $\mathbb{C}H(\mathbb{A}^{\infty}_{F^+})$-modules
\begin{equation}\label{20170224_some_iso}
\bigcup_U S_{\boldsymbol\omega}(U, \mathbb{C}) \cong \mathcal{S}_{\iota^+_{\ast}(\boldsymbol\omega)^{\vee}} \qquad \text{and} \qquad S_{\boldsymbol\omega}(U, \mathbb{C}) \cong \mathcal{S}_{\iota^+_{\ast}(\boldsymbol\omega)^{\vee}}(U).
\end{equation}
(Here, $\mathbb{C}$ is understood as a $\mathcal{O}_{\mathcal{K}}$-algebra via $\iota$ and $\iota^+_{\ast}(\boldsymbol\omega)^{\vee}$ is defined by $\iota^+_{\ast}(\boldsymbol\omega)^{\vee}_{\tau, i} = - \iota^+_{\ast}(\boldsymbol\omega)^{\vee}_{\tau, n+1-i}$.) 

For a place $w$ not dividing $\ell$, the operators $T^{(j)}_{F_w}$ also act on 
$S_{\boldsymbol\omega}(U, \mathcal{O}_{\mathcal{K}})\subset S_{\boldsymbol\omega}(U, \mathbb{C})$, and this action commutes with the isomorphism (\ref{20170224_some_iso}). 
This motivates the following definition: Let $\mathcal{T}'$ be a finite set of places of $F^+$ containing $\mathcal{T}\cup \Omega_{\ell}^{F^+}$
and let $\mathscr{R}$ be a subring of 
$\mathcal{O}_{\mathcal{K}}$, 
then we define the Hecke algebra
\begin{equation*}
^{\mathscr{R}}\mathbb{T}^{\mathcal{T}'}_{\boldsymbol\omega}(U) = 
\operatorname{im}\left(
q: \mathscr{R}[T^{(j)}_{F_w}| j\in \{1, \ldots, n\}, w\in \operatorname{Pl}_F^{\tt{split}, \mathcal{T}'}] \longrightarrow \operatorname{End}_{\mathcal{O}_{\mathcal{K}}}(S_{\boldsymbol\omega}(U, \mathcal{O}_{\mathcal{K}})
\right),
\end{equation*}
where we will often abbreviate $\mathbb{T}^{\mathcal{T}'}_{\boldsymbol\omega}(U) = 
\,^{\mathcal{O}_{\mathcal{K}}}\mathbb{T}^{\mathcal{T}'}_{\boldsymbol\omega}(U)$.
If $f\in S_{\boldsymbol\omega}(U, \mathcal{O}_{\mathcal{K}})$ is an eigenform for this algebra, then we see, using the compatibility with the isomorphism (\ref{20170224_some_iso}), that 
the eigenvalue for a Hecke operator $T$ is given by $\iota^{-1}(a_{\tilde{f}})$, where 
$\tilde{f} \in \mathcal{S}_{\iota^+_{\ast}(\boldsymbol\omega)^{\vee}}(U)$ is the corresponding complex automorphic form. In other words, we can interpret the map $\varphi_{\tilde{f}}$ from above as 
\begin{equation*}
\varphi_f^{\ell}: \,^{\mathbb{Z}}\mathbb{T}_{\boldsymbol\omega}^{\mathcal{T}_{\ell}}(U) 
\longrightarrow \iota(\mathcal{E}(U)) \cong \mathcal{E}(U).
\end{equation*}
Note that we use the bold symbol $\mathbf{T}$ for complex Hecke algebras and the blackboard bold symbol $\mathbb{T}$ for $\ell$-adic Hecke algebras.

\paragraph{Fixed type Hecke algebras} Fix a finite set $\widetilde{\Sigma}\subset 
(\mathcal{T}' - \Omega^F_{\ell})$ of places of $F$ together with a tuple $\underline{\sigma}
= (\sigma_{\nu})_{\nu\in \widetilde{\Sigma}}$, where each $\sigma_{\nu}$ is a finite-dimensional complex representation of $\operatorname{GL}_n(\mathcal{O}_{F_{\nu}})$. Let 
$_{\underline{\sigma}}S_{\boldsymbol\omega}(U, \mathcal{O}_{\mathcal{K}})
\subset S_{\boldsymbol\omega}(U, \mathcal{O}_{\mathcal{K}})$ be the subspace of those 
forms $f$ whose complex correspondents $\tilde{f}$ fulfill the following condition for all 
places $\nu\in \widetilde{\Sigma}$: If $\pi_{\nu}$ denotes the local component of the automorphic representation $\pi = \langle \tilde{f}\rangle$ at $\nu$, then 
$\pi_{\nu}|\operatorname{GL}_n(\mathcal{O}_{F_{\nu}})$ contains $\sigma_{\nu}$ as a subrepresentation. Note that the $T^{(j)}_{F_w}$ (for $w$ in $\operatorname{Pl}_F^{\tt{split},\mathcal{T}'}$) stabilize the subspace $_{\underline{\sigma}}S_{\boldsymbol\omega}(U, \mathcal{O}_{\mathcal{K}})$, so we can define
\begin{equation*}
_{\underline{\sigma}}^{\mathscr{R}}\mathbb{T}^{\mathcal{T}'}_{\boldsymbol\omega}(U) = 
\operatorname{im}\left(
\,_{\underline{\sigma}}q: \mathscr{R}[T^{(j)}_{F_w}| j\in \{1, \ldots, n\}, w\in \operatorname{Pl}_F^{\tt{split}, \mathcal{T}'}] \longrightarrow \operatorname{End}_{\mathcal{O}_{\mathcal{K}}}(\,_{\underline{\sigma}}S_{\boldsymbol\omega}(U, \mathcal{O}_{\mathcal{K}})
\right).
\end{equation*}
We easily see that the assignment $q(T^{(j)}_{F_w}) \mapsto \,_{\underline{\sigma}}q(T^{(j)}_{F_w})$ defines an $\mathscr{R}$-algebra surjection $_{\underline{\sigma}}\theta$ from $^{\mathscr{R}}\mathbb{T}^{\mathcal{T}'}_{\boldsymbol\omega}(U)$ to 
$_{\underline{\sigma}}^{\mathscr{R}}\mathbb{T}^{\mathcal{T}'}_{\boldsymbol\omega}(U)$.
We note the following (for $\mathscr{R} = \mathcal{O}_{\mathcal{K}}$):
\begin{itemize}
\item In the same way as for $^{\mathcal{O}_{\mathcal{K}}}\mathbf{T}_{\boldsymbol\omega}^{\mathcal{T}'}(U)$, we can check that $_{\underline{\sigma}}^{\mathcal{O}_{\mathcal{K}}}\mathbb{T}^{\mathcal{T}'}_{\boldsymbol\omega}(U)$ is free and finitely generated over $\mathcal{O}_{\mathcal{K}}$; 
\item assume that $^{\mathcal{O}_{\mathcal{K}}}\mathbb{T}^{\mathcal{T}'}_{\boldsymbol\omega}(U)_{\mathfrak{m}}\cong \mathcal{O}_{\mathcal{K}}$ holds for any maximal ideal 
$\mathfrak{m}$, then $_{\underline{\sigma}}^{\mathcal{O}_{\mathcal{K}}}\mathbb{T}^{\mathcal{T}'}_{\boldsymbol\omega}(U)_{\mathfrak{n}}$ is a quotient of $\mathcal{O}_{\mathcal{K}}$ for any maximal ideal 
$\mathfrak{n}$. By the above bullet point, it thus follows that $_{\underline{\sigma}}^{\mathcal{O}_{\mathcal{K}}}\mathbb{T}^{\mathcal{T}'}_{\boldsymbol\omega}(U)_{\mathfrak{n}}$ is isomorphic to 
$\mathcal{O}_{\mathcal{K}}$ for any maximal ideal $\mathfrak{n}$.
\end{itemize}
\subsection{Attaching Galois representations to automorphic forms}\label{20170328_attach}
Retain all notation from above and let $\mathfrak{m}\subset \,^{\mathcal{O}_{\mathcal{K}}}\mathbb{T}^{\mathcal{T}_{\ell}}_{\boldsymbol\omega}(U)$ be a maximal ideal. 
\begin{Prop}[{\cite[Proposition 3.1 and 3.2]{Guerberoff}}]\label{20170307_Prop_attach_reps}
There exists a representation
\begin{equation*}
\rho_{\mathfrak{m}}: \operatorname{Gal}_F \rightarrow \operatorname{GL}_n\left(
^{\mathcal{O}_{\mathcal{K}}}\mathbb{T}^{\mathcal{T}_{\ell}}_{\boldsymbol\omega}(U)_{\mathfrak{m}}
\right)
\end{equation*}
with the following properties, where the first two already characterize $\rho_{\mathfrak{m}}$
uniquely: 
\begin{enumerate}
\item $\rho_{\mathfrak{m}}$ is unramified at all but finitely many places; If a place 
$\nu$ of $F^+$ is inert and unramified in $F$ and if $U_{\nu}$ is a hyperspecial 
maximal compact subgroup of $H(F_{\nu}^+)$, then $\rho_{\mathfrak{m}}$ is unramified 
above $\nu$;
\item If $\nu\in \operatorname{Pl}_{F^+}^{\textnormal{fin}} \backslash \mathcal{T}_{\ell}$ 
splits as $\tilde{\nu}\tilde{\nu}^c$ in $F$, then $\rho_{\mathfrak{m}}$ is unramified at 
$\tilde{\nu}$ and 
\begin{equation*}
\operatorname{charPoly}(\rho_{\mathfrak{m}}(\operatorname{Frob}_{\tilde{\nu}})) = 
X^n - T_{\tilde{\nu}}^{(1)}X^{n-1} + \ldots + (-1)^j(\textbf{N}\tilde{\nu})^{\frac{j(j-1)}{2}}
T_{\tilde{\nu}}^{(j)}X^{n-j} + \ldots + (-1)^n (\textbf{N}\tilde{\nu})^{\frac{n(n-1)}{2}}
T_{\tilde{\nu}}^{(n)}.
\end{equation*}
\item $\overline{\rho}_{\mathfrak{m}}^c \cong \overline{\rho}_{\mathfrak{m}} \otimes \overline{\epsilon}_{\ell}^{1-n}$;
\item Fix a set $\widetilde{\Omega}_{\ell}^{F^+}$ of places of $F$ such that 
$\widetilde{\Omega}_{\ell}^{F^+} \sqcup \widetilde{\Omega}_{\ell}^{F^+, c} = 
\Omega_{\ell}^F$ and denote by $\tilde{I}_{\ell}$ the set of embeddings $F\hookrightarrow
\mathcal{K}$ which give rise to an element of $\widetilde{\Omega}_{\ell}^{F^+}$.
Suppose that $w\in \widetilde{\Omega}_{\ell}^{F^+}$ is unramified over $\ell$, that 
$U_{\overline{w}} = H(\mathcal{O}_{F^+, \overline{w}})$ (for $\overline{w}\in \operatorname{Pl}_{F^+}$ the place below $w$) and that for each $\tau\in \tilde{I}_{\ell}$ above $w$ we have
\begin{equation*}
\ell -1 -n \geq \underline{\omega}_{\tau, 1} \geq \ldots \geq \underline{\omega}_{\tau, n}
\geq 0.
\end{equation*}
Then, for each open ideal $I\subset \,^{\mathcal{O}_{\mathcal{K}}}\mathbb{T}^{\mathcal{T}_{\ell}}_{\boldsymbol\omega}(U)$ there is an object $M_{\mathfrak{m}, I, w}$ of 
$\underline{\operatorname{MF}}_{\mathcal{O}_{F_w}, \mathcal{O}_{\mathcal{K}}}$ such that
\begin{equation*}
(\rho_{\mathfrak{m}} \otimes_{\,^{\mathcal{O}_{\mathcal{K}}}\mathbb{T}^{\mathcal{T}_{\ell}}_{\boldsymbol\omega}(U)} \,^{\mathcal{O}_{\mathcal{K}}}\mathbb{T}^{\mathcal{T}_{\ell}}_{\boldsymbol\omega}(U)/I)|\operatorname{Gal}_{F_w} \cong
{\tt{G}}_{F_{w}}(M_{\mathfrak{m}, I, w}).
\end{equation*}
\end{enumerate}
If $\mathfrak{m}$ is non-Eisenstein in the sense of \cite[Definition 3.4.3]{CHT}, then $\rho_{\mathfrak{m}}$ and its reduction extend to
\begin{equation*}
r_{\mathfrak{m}}: \operatorname{Gal}_{F^+} \rightarrow 
\mathcal{G}_n\left( 
{\,^{\mathcal{O}_{\mathcal{K}}}\mathbb{T}^{\mathcal{T}_{\ell}}_{\boldsymbol\omega}(U)_{\mathfrak{m}}}
\right)
\qquad\text{ and }\qquad
\overline{r}_{\mathfrak{m}}: \operatorname{Gal}_{F^+} \rightarrow 
\mathcal{G}_n\left( 
{\,^{\mathcal{O}_{\mathcal{K}}}\mathbb{T}^{\mathcal{T}_{\ell}}_{\boldsymbol\omega}(U)}/{\mathfrak{m}}
\right).
\end{equation*}
Moreover, $\texttt{m}\circ r_{\mathfrak{m}} = \epsilon_{\ell}^{1-n}\delta_{F|F^+}^{\mu_{\mathfrak{m}}}$ for a suitable $\mu_{\mathfrak{m}}\in \mathbb{Z}/2\mathbb{Z}$, where $\delta_{F|F^+}$ is the non-trivial character of 
$\operatorname{Gal}(F|F^+)$.
\end{Prop}
In this way we can associate to a RACSDC automorphic representation
$\pi$ of $\operatorname{GL}_n(\mathbb{A}_F)$ and a finite place $\lambda$ of $\mathcal{E}(U)$ a residual representation $\overline{r}_{\pi, \lambda}: \operatorname{Gal}_{F^+} \rightarrow
\mathcal{G}_n(\overline{\mathbb{F}}_{\ell(\lambda)})$. Let us assume that 
$\overline{r}_{\pi, \lambda}$ is absolutely irreducible for all $\lambda$ in a subset 
of $\operatorname{Pl}^{\text{fin}}_{\mathcal{E}(U)}$ of Dirichlet density $1$. 
Then the set 
\begin{equation*}
\Lambda^1_{\mathcal{E}(U)} = \{\lambda | \overline{\rho}_{\pi, \lambda'} \text{ is absolutely
irreducible for all } \lambda' \text{ dividing } \ell(\lambda)\}
\end{equation*}
has also Dirichlet density $1$. In this way, we get an association from $\pi$ to the 
compatible systems of residual Galois representations 
$\mathcal{R}_{\pi} = ( \overline{r}_{\pi, \lambda})_{\lambda\in \Lambda^1_{\mathcal{E}(U)}}$ and 
$\mathcal{R}_{\pi}' = ( \overline{\rho}_{\pi, \lambda})_{\lambda\in \Lambda^1_{\mathcal{E}(U)}}$.

\section{Consequences from modularity lifting theorems}\label{20160407_sect_cons}
Let us start with
the following
adaption of 
\cite[Lemma 3.6]{KW_on_serre}:
\begin{Lem}\label{20170209_KW_Lemma}
Let $k$ be a finite field of characteristic $\ell$, $G$ a profinite group satisfying the
$\ell$-finiteness condition and $\eta: G\rightarrow \mathcal{G}_n(k)$ be an absolutely
irreducible continuous representation. Let $\mathcal{F}_n(G)$ be a subcategory of 
deformations of $\eta$ in $k$-algebras which defines a deformation condition. 
Let $\eta_{\mathcal{F}}: G\rightarrow \operatorname{GL}_n(R_{\mathcal{F}})$ be the 
universal deformation of $\eta$ in $\mathcal{F}_n(G)$. Then $R_{\mathcal{F}}$ is finite
if and only if $\eta_{\mathcal{F}}(G)$ is finite.
\end{Lem}
\begin{proof}[Proof of the Lemma.]
The proof of Lemma 3.6 of Khare-Wintenberger goes through verbatim, except that we 
have to refer to \cite[Lemma 2.1.12]{CHT} instead of Carayol's Lemma.
\end{proof}
\subsection{A minimal $R=T$-theorem}
Our starting point is a RACSDC automorphic representation 
$\pi = \langle f \rangle \subset\mathcal{S}_{\boldsymbol\omega}(U)$ (where $U=U(S)$ for a finite set 
of places $S$ of $F$) and a place $\lambda\in \Lambda^1_{\mathcal{E}(U)}$.
Fix a finite $F$-big enough extension $\mathcal{K}$ of $\mathcal{E}(U)_{\lambda}$. We abbreviate
$r,\overline{r}, \rho, \overline{\rho}$ for the associated Galois representations 
via Proposition \ref{20170307_Prop_attach_reps} for the unique maximal ideal 
$\mathfrak{m}$ containing $\mathcal{O}_{\mathcal{K}} \otimes_{\mathbb{Z}} 
\operatorname{ker}\varphi^{\ell}_{f}$.
%
%
We assume furthermore the following:
\begin{itemize}
\item All places of $S_{\ell}$ split in the extension $F|F^+$;
\item all ramification of ${\rho}$ is unipotent (this can always be achieved 
by a finite solvable base change);
\item $\rho$ is a minimally ramified (at all places in $S$) and 
FL-crystalline (at all places dividing $\ell$) lift of $\overline{\rho}$;
\item $\overline{\rho}$ is absolutely irreducible;
\item $\overline{\rho}(\operatorname{Gal}_{F(\zeta_{\ell})})$ is adequate.
\end{itemize}
Let us abbreviate 
$R^{\tt{(min),[crys]}}  := 
R^{\chi,\tt{(min),[crys]}}_{{\mathcal{O}_{\mathcal{K}}},S_{\ell}}(\overline{r})$ for the 
ring parametrizing fixed-determinant deformations of $\overline{r}$ which are unramified
outside $S_{\ell}$, (minimally ramified in $S$) and [FL-crystalline at places dividing $\ell$]. Moreover, let $\mathbb{T}$ (resp. $\mathbb{T}^{\tt{min}}$) denote 
the Hecke algebra $^{\mathcal{O}_{\mathcal{K}}}\mathbb{T}^{S_{\ell}}_{\boldsymbol\omega}(U)_{\mathfrak{m}}$
(resp. 
$_{\underline{\sigma}}^{{\mathcal{O}_{\mathcal{K}}}}\mathbb{T}^{S_{\ell}}_{\boldsymbol\omega}(U)_{{_{\underline{\sigma}}\theta}(\mathfrak{m})}$), where $\mathfrak{m}$
is the maximal ideal such that $\overline{r}_{\mathfrak{m}} \cong \overline{r}$ and where
$\underline{\sigma}
= (\sigma_{\nu})_{\nu\in \widetilde{S}}$
is defined as follows:

For each $\nu\in \tilde{S}$ we can associate an inertial type $\tau_{\nu}$ in the same way
as we did just before Theorem \ref{20150813_min_type_comparison_theorem}.
To each $\tau_{\nu}$ one can associate a representation $\sigma_{\nu} = \sigma(\tau_{\nu})$ 
of $K=\operatorname{GL}_n(\mathcal{O}_{F_{\nu}})$ (which is then the $K$-type of the
$\operatorname{GL}_n(F_{\nu})$-representation associated to an extension 
of $\tau_{\nu}$ to $\operatorname{Gal}_{F_{\nu}}$.) For more details on the construction of 
$\sigma(\tau)$, see \cite[Section 4.6]{Shotton_phd}, \cite[Section 6.5.2]{BellaicheChenevier}, 
\cite{SchneiderZink} and the following remark:
\begin{Rem}
Recall the following:
\begin{itemize}
\item Consider the finite group $\mathfrak{G}= \operatorname{GL}_n(\ell(\nu))$ and its standard
Borel subgroup $\mathfrak{B}\subset \mathfrak{G}$. Then the irreducible constituents of
the (complex) representation $\operatorname{ind}_{\mathfrak{B}}^{\mathfrak{G}}(1)$ are 
called the unipotent representations of $\mathfrak{G}$. These representations can (canonically)
be parametrized by the irreducible representations of the Weyl group $\mathcal{W}(\mathfrak{G})
\cong S_n$, see e.g. \cite[Corollary 4.4]{Prasad}. The irreducible representations of $S_n$ 
in turn can be parametrized by partitions of $n$ in terms of Specht modules, cf. 
\cite{JamesKerber}. In other words, we get a canonical bijection $h: \mathcal{Y}_n \cong
\operatorname{Rep}(\mathfrak{G})^{\tt{uni}}$, where 
$\operatorname{Rep}(\mathfrak{G})^{\tt{uni}}$ denotes the set of all unipotent representations
of $\mathfrak{G}$ up to isomorphism. The map $h$ can be explicitly described in terms of 
induction from Levi subgroups (cf. \cite[Definition 4.34]{Shotton_phd}) and sends $(1,\ldots,1)$ to the trivial representation and $(n)$ to the Steinberg representation.
\item Under the unipotent ramification assumption, the set of inertial types $\mathcal{I}^{\tt{uni}}$ is in bijection with the set $\mathcal{Y}_n$ of partitions of $n$ via $\nabla$
from Section \ref{20170321_sect_min_ram}.
\end{itemize}
Now, we have a decomposition 
\begin{equation*}
\operatorname{ind}_I^K(1) \cong \operatorname{infl}_{\mathfrak{G}}^K 
\operatorname{ind}_{\mathfrak{B}}^{\mathfrak{G}}(1) \cong
\bigoplus_{\Pi\in \operatorname{Rep}(\mathfrak{G})^{\tt{uni}}}
m_{\Pi} \operatorname{infl}_{\mathfrak{G}}^K(\Pi),
\end{equation*}
where $I\subset K$ denotes the Iwahori subgroup, $\operatorname{infl}_{\mathfrak{G}}^K$ denotes
the inflation along the pro-$\ell(\nu)$-radical of $K$ and where the $m_{\Pi}\geq 1$ are suitable
multiplicities. Analogous to \cite[Remark 6.5.2 iii)]{BellaicheChenevier}, one can 
thus check that the assignment $\tau\mapsto \sigma(\tau)$ is described in terms of 
partitions as $\tau \mapsto \operatorname{infl}_{\mathfrak{G}}^K(h\circ \nabla(\tau))$.
Observe that the special case $n=2$ is precisely \cite[Remark 6.5.2 iii)]{BellaicheChenevier}
and \cite[Example 2.17]{Shotton_phd}.
\end{Rem}
We stress that the notions $\mathbb{T}$ and $\mathbb{T}^{\tt{min}}$ depend on the choice of the
place $\lambda$.

\begin{Prop}\label{20170330_factorization_prop}
A map $h: \mathbb{T} \rightarrow \overline{\mathbb{Q}}_{\ell}$ factorizes through 
$\mathbb{T}^{\tt{min}}$ if and only if the concatenation ${h':R^{\tt{crys}}\rightarrow
\mathbb{T} \rightarrow \overline{\mathbb{Q}}_{\ell}}$ factors through 
$R^{\tt{min,crys}}$.
\end{Prop}
\begin{proof}
The map $h$ corresponds to an automorphic form $g\in S_{\boldsymbol\omega}(U,{\mathcal{O}_{\mathcal{K}}})$ such that
$\overline{r}_{\langle g\rangle} \cong \overline{r}$. By the above and Theorem
\ref{20150813_min_type_comparison_theorem}, ${\rho}_{\langle g\rangle,\nu}$ (for 
$\nu\in S$) is a minimally ramified lift of $\overline{\rho}_{\nu}$ if and only 
if $\langle g\rangle_{\nu}$ is of type $\sigma_{\nu}$. Thus, $h$ factorizes through 
$\mathbb{T}^{\tt{min}}$ if and only if $g\in \,
_{\underline{\sigma}}S_{\boldsymbol\omega}(U, {\mathcal{O}_{\mathcal{K}}})$, if and only if $r_{\langle g\rangle}$ 
is (as a lift of $\overline{r}$) minimally ramified in $S$, if and only if the associated
map $h': R^{\tt{crys}} \rightarrow \overline{\mathbb{Q}}_{\ell}$ factorizes through 
$R^{\tt{min,crys}}$.
\end{proof}

\begin{Thm}\label{20170210_main_rttheorem}
$R^{\tt{min,crys}}$ is finite flat over $\mathcal{O}_{\mathcal{K}}$, so in particular there exists
a characteristic-$0$ point of $\operatorname{Spec}R^{\tt{min,crys}}$. Moreover, we have isomorphisms
\begin{equation*}
R^{\tt{min,crys}} \cong R^{\tt{min,crys}}_{\tt{red}} \cong \mathbb{T}^{\tt{min}}.
\end{equation*}
\end{Thm}
\begin{proof}
We first remark that $R^{\tt{min,crys}}/(\ell)$ is of finite cardinality, or, equivalently (by Nakayama's Lemma), that $R^{\tt{min,crys}}$ is finitely generated as a ${\mathcal{O}_{\mathcal{K}}}$-module. 
This follows directly from \cite[Theorem 2.3.2]{BLGGT}, 
as we know that the local deformation rings $R^{\boxempty, \chi_{\nu},\tt{crys}}(\overline{\rho}_{\nu})$ and 
$R^{\boxempty, \chi_{\nu},\tt{min}}(\overline{\rho}_{\nu})$
are smooth, hence correspond to irreducible components of 
$\operatorname{Spec}R^{\boxempty, \chi_{\nu}}(\overline{\rho}_{\nu})$ on which the local lifts $\rho_{\nu}$ live.

Next, we remark that by Corollary \ref{2017_cor_on_repn} (together with the
 smoothness-results Lemma \ref{20150712_crystallinity_lemma}, Proposition \ref{20170526_newProponmin}, Proposition \ref{archProp}, the
identity $\operatorname{dim}(\mathfrak{gl}_n^{c_{\nu} = -1})
= \frac{n(n+1)}{2}$ and Remark \ref{20170406_remark_on_mu_m} below)
there exists a presentation
\begin{equation*}
R^{\tt{min, crys}} \cong {\mathcal{O}_{\mathcal{K}}}\llbracket X_1, \ldots, X_m\rrbracket/(f_1, \ldots, f_m).
\end{equation*}
for some $m\in \mathbb{N}_0$.

Using this presentation and the finiteness of $R^{\tt{min, crys}}/(\ell)$, 
it follows as in the proof of Theorem 3.7 of 
\cite{KW_on_serre} or of Lemma 2 in B\"ockle's Appendix to 
\cite{Khare_isos} that
$R^{\tt{min, crys}}$ is finite flat over ${\mathcal{O}_{\mathcal{K}}}$, hence free and finitely 
generated over ${\mathcal{O}_{\mathcal{K}}}$. This proves the first claim.

As a second step, we remark that any morphism $f:R^{\tt{min, crys}} \rightarrow \overline{\mathbb{Q}}_{\ell}$ 
factorizes over $\mathbb{T}^{\tt{min}}$:
By \cite[Thm 2.3.1]{BLGGT}, such an $f$ factorizes over
the non-minimal Hecke algebra $\mathbb{T}$.
Therefore Proposition \ref{20170330_factorization_prop}
applies.

Now, 
$
R^{\tt{min, crys}}[\frac{1}{\ell}]
= R^{\tt{min, crys}}\otimes_{\mathcal{O}_{\mathcal{K}}} \mathcal{K}
$
is a finite $\mathcal{K}$-algebra, hence $R^{\tt{min, crys}}[\frac{1}{\ell}]$ is Artinian, see e.g. \cite[Exercise 8.3]{AtiyahMacdo}. 
Therefore, $R^{\tt{min, crys}}[\frac{1}{\ell}]$
can be decomposed into a product of finitely many local
Artinian rings
\begin{equation*}
R^{\tt{min, crys}}\left[\frac{1}{\ell}\right] \cong 
\oplus_{\mathfrak{p}}\, R^{\tt{min, crys}}\left[\frac{1}{\ell}\right]_{\mathfrak{p}}
\end{equation*}
and by \cite[Theorem 3.1.3]{Allen_polarized} the tangent space ${\mathfrak{p}}/{\mathfrak{p}}^2$ of
each $R^{\tt{min, crys}}[\frac{1}{\ell}]_{\mathfrak{p}}$ vanishes. 
Hence, $\mathfrak{p}= {\mathfrak{p}}^2$, and it follows from Nakayama's lemma that
$\mathfrak{p} = 0$, i.e. that $R^{\tt{min, crys}}[\frac{1}{\ell}]_{\mathfrak{p}}$ is a field. Thus, $R^{\tt{min, crys}}[\frac{1}{\ell}]$ is a finite product of fields.
The same is true for $\mathbb{T}^{\tt{min}}[\frac{1}{\ell}]$: As $\mathbb{T}^{\tt{min}}[\frac{1}{\ell}]$ is finitely-generated, its Jacobson radical equals its nilradical, which
vanishes because $\mathbb{T}^{\tt{min}}$ is reduced. Hence $\mathbb{T}^{\tt{min}}[\frac{1}{\ell}]$ is semisimple, i.e. a product of finitely many fields. 

Consider the exact sequence
\begin{equation}\label{20170526_yet_another_ex_seq}
0 \rightarrow \operatorname{ker}(\varphi) \rightarrow R^{\tt{min, crys}}
\overset{\varphi}{\rightarrow} \mathbb{T}^{\tt{min}} \rightarrow 0,
\end{equation}
where $\varphi$ denotes the canonical projection. It follows from the above observation 
about  $R^{\tt{min, crys}}[\frac{1}{\ell}]$ and $\mathbb{T}^{\tt{min}}[\frac{1}{\ell}]$
together with Proposition \ref{20170330_factorization_prop} that $\varphi[\frac{1}{\ell}]$ is
an isomorphism. Moreover, as both $R^{\tt{min, crys}}$ and $\mathbb{T}^{\tt{min}}$ are finite
free over ${\mathcal{O}_{\mathcal{K}}}$, 
(\ref{20170526_yet_another_ex_seq}) is a split exact sequence
of free $\mathcal{O}_{\mathcal{K}}$-modules.
Hence, $\operatorname{ker}(\varphi)=0$ 
since $\varphi[\frac{1}{\ell}]$ is an isomorphism.
%
%
This completes the proof of the theorem.\qedhere
\end{proof}

\begin{Cor}\label{20170210_pushout_cor}
\begin{equation*}
\xymatrix{
R^{\tt{crys}} \ar[r]\ar[d]& R^{\tt{min, crys}}\ar[d]\\
\mathbb{T}\ar[r]& \mathbb{T}^{\tt{min}}
}
\end{equation*}
is a pushout diagram.
\end{Cor}
\begin{Rem}\label{20170406_remark_on_mu_m}
We remark that for each $\nu\in \Omega_{\infty}$ the local deformation
ring $^{L^+_{\nu}}R^{\boxempty, \chi_{\nu}}_{W(k_{\lambda})}(\overline{r}_{\lambda, \nu})$ is 
formally smooth of relative dimension $d_{\nu}^{\boxempty} = 
\operatorname{dim}(\mathfrak{b}^{\tt{der}}_n)$: We get from Proposition \ref{archProp} that 
$^{L^+_{\nu}}R^{\boxempty, \chi_{\nu}}_{W(k_{\lambda})}(\overline{r}_{\lambda, \nu})$ is
formally smooth of relative dimension $\operatorname{dim}\bigl( (\mathfrak{g}_n)^{c_{\nu} = 
-1}\bigr) = \operatorname{dim}\bigl (\mathfrak{gl}_n)^{c_{\nu} = -1})$, where 
$c_{\nu}$ is the non-trivial element of the decomposition group at $\nu$. By construction
(see Lemma 2.1.4 and Proposition 3.4.4 of \cite{CHT}), the image
of $\overline{r}_{\lambda}(c_{\nu})$ is not contained in $\operatorname{GL}_n \times
\operatorname{GL}_1$. Moreover, 
\begin{equation*}
\mathtt{m}\circ \overline{r}_{\lambda}(c_{\nu}) = \overline{\epsilon}_{\ell}^{1-n}(c_{\nu})
\delta_{F|F^+}^{\mu_{\mathfrak{m}}}(c_{\nu}) = (-1)^{\mu_{\mathfrak{m}+ p}},
\end{equation*}
where $p= n+1 (\operatorname{mod} 2) \in \mathbb{Z}/2\mathbb{Z}$, where $\epsilon_{\ell}$
denotes the cyclotomic character (sending $c_{\nu}$ to $-1$), where $\delta_{F|F^+}$
denotes the non-trivial character of $\operatorname{Gal}(F|F^+)$ and where $\mu_{\mathfrak{m}}$
is a suitable element of $\mathbb{Z}/2\mathbb{Z}$. As in \cite[Corollary 6.9]{Thorne_small}, we 
get 
$\mu_{\mathfrak{m}} \equiv n (\operatorname{mod} 2)$, so we have 
$\texttt{m}\circ \overline{r}_{\lambda}(c_{\nu}) = -1$, independent of the parity of $n$.
Using \cite[Lemma 2.1.3]{CHT}, this implies $\operatorname{dim}(\mathfrak{gl}_n^{c_{\nu} = -1})
= \frac{n(n+1)}{2}= \operatorname{dim}(\mathfrak{b}_n^{\tt{der}})$.
\end{Rem}

\subsection{A $T=O$-theorem}
Now, let $\mathcal{E} \supset \mathcal{E}(U)$ be a number field with ring of integers $\mathcal{O}_{\mathcal{E}}$.
For each $\lambda\in \operatorname{Pl}^{\textnormal{fin}}_\mathcal{E}$ such that $\ell := \ell(\lambda)\gg 0$,
let us fix an $F$-big enough extension $\mathcal{K}_{\lambda}$ of $\mathcal{E}_{\lambda}$ and let us 
abbreviate 
\begin{equation*}
\mathbf{T} := \, ^{\mathcal{O}_{\mathcal{E}}}\mathbf{T}^{\mathcal{T}}_{\boldsymbol\omega}(U) \text { and }
\mathbf{T}_{\lambda} := 
\mathcal{O}_{\mathcal{K}_{\lambda}} \otimes_{\mathcal{O}_{\mathcal{E}}}
\, \mathbf{T} \cong
\, ^{\mathcal{O}_{\mathcal{K}_\lambda}}\mathbb{T}^{\mathcal{T}_{\ell}}_{\boldsymbol\omega_{\ell}}(U).
\end{equation*}
Observe the following about the isomorphism on the right hand side: Using that 
$\mathcal{S}_{\boldsymbol\omega}(U)$ admits a basis of eigenforms, we can embed
$\mathbf{T}$ into a product of finitely many $\mathcal{O}_{\mathcal{E}(U)}$. Hence, 
$\mathbf{T}$ is finitely generated as a $\mathbb{Z}$-module, hence as a $\mathbb{Z}$-algebra.
It follows that there exists a Sturm-like bound $C\in \mathbb{N}$ such that 
$\mathbf{T}$ is already generated by those $T^{(j)}_{F_w}$ with $\ell(w) \leq C$. Hence, using
the compatibility from (\ref{20170224_some_iso}), we get
\begin{equation*}
\mathcal{O}_{\mathcal{K}_{\lambda}} \otimes_{\mathcal{O}_{\mathcal{E}}}
\, \mathbf{T} \cong
\mathcal{O}_{\mathcal{K}_{\lambda}} \otimes_{\mathcal{O}_{\mathcal{E}}}
\,^{\mathcal{O}_{\mathcal{E}}}\mathbf{T}^{\mathcal{T}_{\ell}}_{\boldsymbol\omega}(U) \cong
\, ^{\mathcal{O}_{{\mathcal{K}}_{\lambda}}}\mathbb{T}^{\mathcal{T}_{\ell}}_{\boldsymbol\omega_{\ell}}(U)
\end{equation*}
as long as $\ell>C$.
Then we have:
\begin{Lem}
For almost all $\lambda$ (the failure set depending only on $\mathbf{T}$), 
$\mathbf{T}_{\lambda}$ decomposes as a product of finitely many complete discrete 
valuation rings, finite over $\mathbb{Z}_{\ell}$.
\end{Lem}
\begin{proof}
First, we see that $\mathbf{T}$ is an order in $\mathbf{T} \otimes_{\mathcal{O}_\mathcal{E}}
\mathcal{E} \cong k_1\times \ldots \times k_m$, where the $k_i$ denote suitable number fields (containing $E$) and the 
decomposition follows because 
$\mathbf{T} \otimes_{\mathcal{O}_\mathcal{E}} \mathcal{E}$ is reduced (as already remarked). Hence, 
$\mathbf{T}$ is contained in the maximal order $\oplus_{i=1}^m \mathcal{O}_{k_i}$.
It follows that there exists a suitable $N\in \mathbb{N}$ such that
\begin{itemize}
\item $\mathbf{T}[\frac{1}{N}] \cong \oplus_{i=1}^m \mathcal{O}_{k_i}[\frac{1}{N}]$;
\item for any $\lambda$ with $\ell(\lambda)\not|\, N$, we have $\mathbf{T}_{\lambda} \cong
\mathbf{T}[\frac{1}{N}]_{\lambda} := \mathcal{O}_{\mathcal{K}_{\lambda}} \otimes_{\mathcal{O}_{\mathcal{E}}}\mathbf{T}[\frac{1}{N}]$.
\end{itemize}
Thus, for those $\lambda$ we get an isomorphism 
$\mathbf{T}_{\lambda} \cong \oplus_{i=1}^m \mathcal{O}_{\mathcal{K}_{\lambda}} \otimes_{\mathcal{O}_{\mathcal{E}}} \mathcal{O}_{k_i}$. As each factor itself is a product of 
complete discrete valuation rings (cf. e.g. \cite[Ch.2, \S 3, Theorem 1(ii)]{Serre_Local_fields}), the lemma follows.
\end{proof}
Because we assumed that $\mathcal{E}$ contains all Hecke eigenvalues, in fact all the fields $k_i$ in 
the above proof are equal to $\mathcal{E}$.
Hence, for almost all 
$\ell$, the above lemma implies that $\mathbf{T}_{\lambda}$ is isomorphic to a product of 
finitely many copies of $\mathcal{O}_{\mathcal{K}_\lambda}$. Thus, we get
\begin{Cor}\label{20170307_TO_corollary}
For almost all $\lambda$ and all maximal ideals $\mathfrak{m}\subset \mathbf{T}_{\lambda}$, 
we have an isomorphism $\mathbf{T}_{\lambda, \mathfrak{m}} \cong \mathcal{O}_{\mathcal{K}_\lambda}$.
\end{Cor}
\subsection{An $R=R^{\tt{min}}$-theorem}
We retain all notation from the above and start with a preparatory corollary (to Corollary \ref{20170210_cor_on_taylor_cond}):
\begin{Cor}\label{20170404_an_prep_cor}
For almost all $\lambda$ for which $\overline{\rho}_{\lambda}$ is absolutely irreducible,
we have 
$R^{\tt{crys},(1, \ldots, 1)} = R^{\tt{crys}}$.
\end{Cor}
\begin{proof}
Let $m:= \operatorname{max}\{p \in \mathbb{N}\;|\; p \text{ prime}, \nu | p \text{ for some }
\nu \in S\}$. Then, for all $\lambda$ with $\ell(\lambda) > m^{n!}$, the claim follows directly 
from Corollary \ref{20170210_cor_on_taylor_cond}.
\end{proof}
Moreover, we need a congruence argument:
First, recall that the Hecke algebra $^{\mathcal{O}_{\mathcal{K}}}\mathbb{T}^{S_{\ell}}_{\boldsymbol\omega}(U)$
acts semisimply on $\mathcal{S}_{\boldsymbol\omega}(U)$, so the space 
$\mathcal{S}_{\boldsymbol\omega}(U)$ decomposes into finitely many eigenspaces.
For the following, let us consider \textit{congruences}, by what we mean triples $(H_1,H_2, \ell)$, 
where $H_1\neq H_2$ are two Hecke eigenspaces and
where $\ell$ is a rational prime such that there exists an isomorphism 
$\overline{\rho}_{f_1,\lambda_1} \otimes \overline{\mathbb{F}}_{\ell}
\cong 
\overline{\rho}_{f_2,\lambda_2} \otimes \overline{\mathbb{F}}_{\ell}$ for some choice of
forms $f_i\in H_i$ and of places $\lambda_i$ of the corresponding coefficient fields fulfilling
$\ell(\lambda_i) = \ell$. 
\begin{Prop}\label{20170404_prop_cong_arg}
There exist only finitely many such congruences in $\mathcal{S}_{\boldsymbol\omega}(U)$.
\end{Prop}
\begin{proof}
We easily see that a congruence $(H_1,H_2,\ell)$ corresponds
to two distinct minimal prime ideals $\mathfrak{p}_{f_1}, \mathfrak{p}_{f_2}$ of $\mathbf{T}$ for which there exists 
a maximal ideal $\mathfrak{m}\subset \mathbf{T}$
which contains $\ell$, $\mathfrak{p}_{f_1}$ and $\mathfrak{p}_{f_2}$.
It follows from the finite flatness of $\mathbf{T}$ over 
$\mathbb{Z}$ that for given eigenforms $f_1,f_2$ there exist only finitely
many maximal ideals containing 
$\mathfrak{p}_{f_1}$ and $\mathfrak{p}_{f_2}$.
Thus, the claim follows immediately from the finite-dimensionality of the space of automorphic
representations of given level and weight.
\end{proof}

\begin{Thm}\label{20170213_global_rrmin_thm}
For almost all $\lambda$ for which $\overline{\rho}_{\lambda}$ is absolutely irreducible,
we have  
\begin{equation*}
R^{\tt{crys}} \cong R^{\tt{min,crys}}.
\end{equation*}
\end{Thm}
\begin{proof}
We apply the proof of Theorem \ref{20170210_main_rttheorem},
where we replace $R^{\tt{min,crys}}$ by $R^{\tt{crys}}$ and $\mathbb{T}^{\tt{min}}$ by
$R^{\tt{min,crys}}$. In the following, we only point out those steps which need additional
explanation: 

The part that $R^{\tt{crys}}/(\ell)$ is finite follows as above, but we 
have to assure that Guerberoff's Theorem 4.1 remains applicable, at least for almost 
all $\ell$. This is a direct consequence
of Corollary \ref{20170404_an_prep_cor}.

The existence of a presentation
\begin{equation*}
R^{\tt{crys}} \cong {\mathcal{O}_{\mathcal{K}}}\llbracket X_1, \ldots, X_m\rrbracket/(f_1, \ldots, f_m);
\end{equation*}
follows by 
Corollary \ref{2017_cor_on_repn}
(as in the proof of Theorem \ref{20170210_main_rttheorem}, but
referring to Theorem \ref{20170213_Helms_theorem} instead 
of Proposition \ref{20170526_newProponmin}).

Moreover, we claim that for almost all $\lambda$, any morphism $R^{\tt{crys}}\rightarrow \overline{\mathbb{Q}}_{\ell}$ factorizes over $R^{\tt{min,crys}}$. Using automorphy lifting, this claim can
be restated as follows: For almost all $\lambda$, the following holds: For any automorphic form $g$ whose associated
Galois representation $\rho_{g,\lambda}$ reduces to $\overline{\rho}_{\lambda}$, 
$\rho_{g,\lambda}$ is a minimally ramified lift of $\overline{\rho}_{\lambda}$. 
Now, let $\lambda$ be a place such that this statement fails. 
Then, as there always exists a minimally ramified lift 
of $\overline{\rho}_{g,\lambda}$
with a corresponding automorphic form $f$ (cf. Theorem \ref{20170210_main_rttheorem}), we get
a congruence $({^{\mathcal{O}_{\mathcal{K}}}\mathbb{T}^{S_{\ell}}_{\boldsymbol\omega}(U)}\cdot f, {^{\mathcal{O}_{\mathcal{K}}}\mathbb{T}^{S_{\ell}}_{\boldsymbol\omega}(U)}\cdot g, \ell(\lambda))$. Thus, the claim follows
from Proposition \ref{20170404_prop_cong_arg}.
%

%

The rest of the proof goes through verbatim. 
\end{proof}
We need a technical lemma:
\begin{Lem}\label{20150526_a_tech_lemma}
Let $R_1, R_2\in \mathcal{C}_{\Lambda}$ and assume that
there exist suitable $\alpha, \beta, m,h \in \mathbb{N}_0$ such
that
\begin{enumerate}
\item $R_1 \cong \Lambda\llbracket y_1, \ldots, y_{\beta+h}\rrbracket / (g_1, \ldots, g_{\beta})$;
\item $ R_2 \cong \Lambda\llbracket z_1, \ldots, z_{m + h}\rrbracket$;
\item $
R_1\llbracket x_1, \ldots, x_{\alpha + m}\rrbracket/(f_1, \ldots, f_{\alpha}) \cong R_2$.
\end{enumerate}
Then, $R_1$ is has only one irreducible component.
\end{Lem}
\begin{proof}
Let us rewrite the last condition as
\begin{equation*}
\Lambda\llbracket x_1, \ldots, x_{\alpha+m}, y_1, \ldots, y_{\beta+h}\rrbracket/(f_1, \ldots, f_{\alpha}, 
g_1, \ldots, g_{\beta}) \cong 
\Lambda\llbracket z_1, \ldots, z_{m+h}\rrbracket.
\end{equation*}
It follows immediately by \cite[Proposition 22, Part c) $\Rightarrow$ a))]{Serre_Local_fields}, that $f_1,\ldots, f_\alpha, g_1, \ldots, g_{\beta}$ is a
subset of regular system of parameters of $\Lambda\llbracket x_1, \ldots, x_{\alpha+m}, y_1, \ldots, y_{\beta + h}\rrbracket$. It follows easily that then $g_1, \ldots, g_h$ is a subset of 
a regular system of parameters of $\Lambda\llbracket y_1, \ldots, y_{\beta + h}\rrbracket$, and so
(by \cite[Proposition 22, Part a) $\Rightarrow$ c)]{Serre_Local_fields}), it follows that
$R_1$ is
regular.
As any regular local ring is an integral domain, the claim 
follows.
\end{proof}

Let us close with the following corollary (to Theorem \ref{20170213_global_rrmin_thm}), giving a local $R=R^{\tt{min}}$ result:
\begin{Cor}\label{20170314_cor_on_min_is_all}
For almost all $\lambda$, 
$R^{\boxempty, \chi_{\nu}, \tt{min}}(\overline{\rho}_{\lambda,\nu}) \cong 
R^{\boxempty, \chi_{\nu}}(\overline{\rho}_{\lambda,\nu})$ holds for any $\nu\in S$.
\end{Cor}
\begin{proof}
Assume that $\lambda$ is such that Theorem \ref{20170213_global_rrmin_thm} holds. Moreover,
fix a place $\nu_0\in S$ and write $R^{\boxempty_{S_{\ell}},(S-\nu_0)-\tt{min, crys}}$ for the ring parametrizing the same  (global) lifts as 
$R^{\boxempty_{S_{\ell}},\tt{min, crys}}$, except that we don't impose
any condition at $\nu_0$ at all. Then we have isomorphisms
\begin{equation*}
R^{\tt{loc},\chi, \textnormal{$(S-\nu_0)$}-\tt{min, crys}}\llbracket x_1,\ldots, x_a\rrbracket/(f_1, \ldots, f_b) \cong 
R^{\boxempty_{S_{\ell}},(S-\nu_0)-\tt{min, crys}} 
\cong R^{\boxempty_{S_{\ell}},\tt{min, crys}}
\end{equation*}
where the first isomorphism (as well as the identity $a-b = \# S_{\ell} -1$) follows from 
Proposition \ref{20170329_first_rep_result}.
We know that all local deformation rings occurring as tensor factors in the definition $R^{\tt{loc},\chi, \textnormal{$(S-\nu_0)$}-\tt{min, crys}}$ (except for possibly at $\nu$) are formally smooth, and we also know their dimension:
\begin{equation*}
\operatorname{dim}_{W}(R^{\boxempty,\chi_{w}, \mathcal{D}_{w}}_{\nu}) = 
\begin{cases}
\frac{n(n+1)}{2} & \text{ if $w|\infty$, $\mathcal{D}_{w} = \emptyset$};\\
n^2 + [F:\mathbb{Q}]\frac{n(n-1)}{2} & \text{ if $w|\ell$, $\mathcal{D}_{w} = \tt{crys}$};\\
n^2 & \text{ if $w\in (S-\nu)$, $\mathcal{D}_{w} = \tt{min}$};
\end{cases}
\end{equation*}

By our $R^{\tt{min}}\cong T^{\tt{min}}\cong \mathcal{O}$-results from the previous two sections,
which hold again 
after ruling out finitely many $\lambda$, also 
$R^{\boxempty_{S_{\ell}},\tt{min},\tt{crys}}$ is formally smooth of
dimension $(n^2+1)\#S_{\ell} - 1$.
Thus, we get that 
\begin{equation*}
R^{\boxempty,\chi_{\nu_0}}(\overline{\rho}_{\nu_0})\llbracket x_1, \ldots, x_{\alpha+m}\rrbracket/(f_1,\ldots, f_\alpha) = R^{\boxempty,\tt{min},\tt{crys}} = {\mathcal{O}_{\mathcal{K}}}\llbracket x_1, \ldots, x_{(n^2+1)\#S_{\ell} - 1}\rrbracket
\end{equation*}
with $m = (n^2+1)\#S_{\ell} - 1 - n^2$. Now, by Theorem \ref{20170213_Helms_theorem}, 
$R^{\boxempty,\chi_{\nu_0}}(\overline{\rho}_{\nu_0})$ is a complete intersection ring, i.e. we can write
\begin{equation*}
R^{\boxempty,\chi_{\nu_0}}(\overline{\rho}_{\nu_0}) = {\mathcal{O}_{\mathcal{K}}}\llbracket y_1, \ldots, y_{\beta+h}\rrbracket/(g_1, \ldots, g_{\beta}) 
\end{equation*}
with $h = \operatorname{dim} R^{\boxempty,\chi_{\nu_0}}(\overline{\rho}_{\nu_0}) = n^2$.
By Lemma \ref{20150526_a_tech_lemma}
applied to $\Lambda = \mathcal{O}_{\mathcal{K}}$, 
$R_1 = R^{\boxempty,\chi_{\nu_0}}(\overline{\rho}_{\nu_0})$ and 
$R_2 = R^{\boxempty,\tt{min},\tt{crys}}$, we see that 
%
%
%
$R^{\boxempty,\chi_{\nu_0}}(\overline{\rho}_{\nu_0})$ has only one irreducible component and hence we must have
$R^{\boxempty, \chi_{\nu_0}, \tt{min}}(\overline{\rho}_{\nu_0}) \cong 
R^{\boxempty, \chi_{\nu_0}}(\overline{\rho}_{\nu_0})$.
\end{proof}

\section{Unobstructedness for RACSDC automorphic representations}\label{20170407_main_sect}
We are now in a position to state and prove our main result. For this, let $\pi$ be a 
RACSDC autormorphic representation of $\operatorname{GL}_n(\mathbb{A}_F)$ with ramification set $S$.
By passing to a unitary avatar $\Pi\subset \mathcal{S}_{\boldsymbol\omega}(U)$ (for a suitable
weigth $\boldsymbol\omega$ and level $U= U_S$),
we can attach the compatible system $\mathcal{R}_{\pi} = (\overline{r}_{\lambda})_{\lambda\in \Lambda_{\mathcal{E}_{\Pi}}^1}$. Here, $\mathcal{E}_{\Pi}$ denotes the 
number field generated by all Hecke eigenvalues of $\Pi$, $\Lambda_{\mathcal{E}_{\Pi}}^1 \subset \operatorname{Pl}_{\mathcal{E}_{\Pi}}$ denotes the set of places for which $\overline{\rho}_{\lambda}$ is absolutely irreducible and  
 we assume the following:
\begin{Ass}\label{20170228_nec_assumptions}
\begin{enumerate}
\item\textbf{(Irreducibility):} The set $\Lambda_{\mathcal{E}_{\Pi}}^1\subset \operatorname{Pl}_{\mathcal{E}_{\Pi}}$  has Dirichlet density $1$;
\item\textbf{(No consecutive weights):} The multisets of Hodge-Tate weights $HT_{\tau}$ of the system 
$\mathcal{R}_{\pi}$ fulfill (for all embeddings $\tau$) the condition from Theorem \ref{20170323_FL_main_thm}: If two numbers $a,b$ occur in $HT_{\tau}$, then $|a-b|\neq 1$.
\end{enumerate}
\end{Ass}
We stress that we understand the first part as a general conjecture on Galois representations
attached to RACSDC representations (so, in particular, we assume that this is correct independently of the choice of $F$ or $\pi$), while the second part puts a constraint on our
choice of $\pi$. We also have the following:
\begin{Rem}\label{20150407_irreducibility_remark}
The first part of 
Assumption \ref{20170228_nec_assumptions}
is known to hold e.g. if $\pi$ is extremely regular \cite{BLGGT} or if 
$n\leq 5$ \cite{CalegariGee}. Results in this direction are also contained in 
\cite{PatrikisTaylor}, but they are not directly applicable to our situation.
We also remark that all entries
in the $\ell$-adic system $(\rho_{\pi, \lambda})_{\lambda\in \operatorname{Pl}_{\mathcal{E}(U)}}$ 
are expected (by cuspidality of $\pi$) to be absolutely irreducible and that this, using suitable 
modularity lifting theorems, is expected to imply absolute irreducibility of
$\overline{\rho}_{\pi,\lambda}$ for almost all $\lambda$. 
An established result in this direction is that absolute
irreducibility of the $\ell$-adic system implies absolute 
irreducibility of $\overline{\rho}_{\pi,\lambda}$ except
for a failure set of Dirichlet density $0$, see \cite{PSW}.
%
%
\end{Rem}

Our main result is now as follows: 
\begin{Thm}\label{20170228_main_thm}
Presuming Assumption \ref{20170228_nec_assumptions}, there exists a subset 
$\Lambda^0_{\mathcal{E}_{\Pi}} \subset \Lambda^1_{\mathcal{E}_{\Pi}}$ of Dirichlet density $1$ such that the functor $D^{\boxempty_{S_{\ell}}, \chi}_{S_{\ell}, W(k_{\lambda})}(\overline{r}_{\lambda}) $ is globally unobstructed whenever $\lambda\in 
\Lambda^0_{\mathcal{E}_{\Pi}}$.
\end{Thm}
As a first step towards the proof, let us consider the following simplifying assumption:
\begin{Ass}\label{20170228_nonec_assumptions}
\begin{enumerate}
\item $F|F^+$ is unramified at all finite places and, in case $n$ is even, then also
$\frac{n}{2}[F^+:\mathbb{Q}]$ is even;
\item each place $\nu$ of $F^+$ which lies below $S$ splits in $F|F^+$ as, say, $\tilde{\nu}\tilde{\nu}^c$; (For archimedean places, this condition is automatically fulfilled, so we can replace $S$ by $S\sqcup \Omega_{\infty}$ without loss of generality.)
\item for each place $\nu$ of $F^+$ which lies below $S$, the Weil-Deligne representation 
$(r_{\nu}, W_{\nu})$ attached to $\Pi$ has unramified
underlying Weil-representation $r_{\nu}$.
%
\end{enumerate}
\end{Ass}
Remark that the third part can be characterized as follows: The $\ell$-adic representation 
$r_{\Pi,\lambda}$ is at $\nu$ a minimally ramified deformation of $\overline{r}_{\Pi,\lambda}$. (As the system associated to $\Pi$ is \textit{compatible}, this does not depend on the choice of $\lambda\in \Lambda^1_{\mathcal{E}_{\Pi}}$.) 
Now, consider the following (seemingly weaker) variation of Theorem \ref{20170228_main_thm}:
\begin{Thm}\label{20170301_weaker_main_thrm}
Presuming Assumptions \ref{20170228_nec_assumptions} and \ref{20170228_nonec_assumptions},
there exists a subset 
$\Lambda^0_{\mathcal{E}_{\Pi}} \subset \Lambda^1_{\mathcal{E}_{\Pi}}$ of Dirichlet density $1$ such that the functor $D^{\boxempty_{S_{\ell}}, \chi}_{S_{\ell}, W(k_{\lambda})}(\overline{r}_{\lambda}) $ is globally unobstructed whenever $\lambda\in 
\Lambda^0_{\mathcal{E}_{\Pi}}$.
\end{Thm}
\begin{proof}[Proof (that Theorem \ref{20170301_weaker_main_thrm} implies 
Theorem \ref{20170228_main_thm})]
It is a standard argument (cf. e.g. the proof of \cite[Theorem 4.4.2]{CHT}) that there exists a finite solvable extension $F_1^+|F^+$ of totally real fields such that the extension $F_1 = F^+_1.F|F_1^+$ and the compatible family associated to the base change $\Pi_{F_1}$ of $\Pi$
to $F_1$ fulfill Assumption \ref{20170228_nonec_assumptions}. Thus, referring to 
Lemma \ref{20170404_pot_unobstr}
and eliminating the finitely many places $\lambda$ which divide the index $[F^+_1:F^+]$, 
we see that Assumption \ref{20170228_nonec_assumptions} can be included in the 
statement of Theorem \ref{20170228_main_thm} without causing loss of generality.
%
\end{proof}
Consequently, the remainder of this section is devoted to the proof of Theorem \ref{20170301_weaker_main_thrm}.
For better comprehension, let us give an overview of the
strategy of the proof:
We want to arrange for a situation where the framework 
of Section \ref{20170405_section_on_framework} is applicable, i.e. we want to consider 
suitable field extensions $L^+_{(\lambda)}$ for as many 
$\lambda$ as possible such that 
Theorem \ref{20160908_MainThm} implies the vanishing of the dual Selmer groups
of the base canged functors 
$^{L^+_{(\lambda)}}D^{\boxempty_{S_{\ell}\langle L^+_{(\lambda)}\rangle}, \chi}_{S_{\ell}\langle L^+_{(\lambda)}\rangle, W(k_{\lambda})}(\overline{r}_{\lambda})$. 
This application of Theorem \ref{20160908_MainThm}
happens in Theorem \ref{20170314_key_thm_for_proof_of_6_2} 
below.
By a careful choice of the extensions $L^+_{(\lambda)}$, we 
ensure that the potential unobstructedness arguments of 
Section \ref{20170405_section_on_framework} apply and yield the vanishing of the dual Selmer
groups of the non-base changed functor. The local parts 
of the unobstructedness-condition then follow directly from
the material in Section \ref{20170324_section_unres_defos}, allowing us to conclude the
statement of Theorem \ref{20170301_weaker_main_thrm}.
The crucial property we have to impose on the extension
$L^+_{(\lambda)}$ is \textit{procurability} (Definition
\ref{20170528_def_procurability}), i.e. that
the deformation ring 
$R^{\chi|{\operatorname{Gal}_{L^+}}, \tt{crys}}_{S_{\ell}\langle L^+\rangle,\mathcal{O}_{\mathcal{K}_{\lambda}}}(\overline{r}_{\lambda}|{\operatorname{Gal}_{L^+}})$
is isomorphic to $\mathcal{O}_{\mathcal{K}_{\lambda}}$ 
(corresponding to condition ($R=T$) in Section \ref{20170405_section_on_framework}).
It is the content of Theorem 
\ref{20170308_firstmaintechthm} that for a set of places
of Dirichlet density $1$ we can find such suitable procurable
extensions.
This, in turn, is established by studying the seemingly weaker
condition of $\star$-procurability (see the list
$\star_1 - \star_5$ below), which is proved to imply
procurability almost everywhere (see Claim 1 below). 
By an argument based on Chebotarev's density theorem (and
postponed to Appendix A), we can conclude that for 
a density-$1$ set we can find such $\star$-procurable extensions of $2$-power degree.

\subsection{Proof of Theorem \ref{20170301_weaker_main_thrm}}
Let us begin with some preparatory definitions:
\begin{Def}
A totally real, finite extension $L^+$ of $F^+$ is called \textit{pre-admissible} if 
the extension $L^+|F^+$ is Galois and solvable and if $L:= F.L^+$ is unramified over
$L^+$ at every finite place.
\end{Def}
We remark that these conditions are designed to capture the following: If $L^+$ is pre-admissible, then there exists a unitary group $H$ over $L^+$ (as considered in Section 
\ref{20170301_sectn_hecke_algs}) and a unitary avatar $\Pi_L$ on $H(\mathbb{A}_{L^+})$ 
of the base change $\pi_L$ of $\pi$ to $L$.

For the following, let $\mathcal{E}$ be a number field containing $\mathcal{E}(U)$ and 
let $L^+$ be pre-admissible.
\begin{Def}\label{20170528_def_procurability}
A place $\lambda\in \Lambda_{\mathcal{E}}^1$ is $L^+$-\textit{procurable} if the following 
two conditions are fulfilled:
\begin{enumerate}
\item[P.1)] The restriction of $\overline{\rho}_{\lambda}$ to $\operatorname{Gal}_L$ remains
absolutely irreducible;
\item[P.2)] There exists an $L$-big enough extension $\mathcal{K}_{\lambda}$ of $\mathcal{E}_{\lambda}$ such that there is an isomorphism 
\begin{equation}\label{20170307_iso_to_prove}
^{L^+}R^{\chi, \tt{crys}}_{S_{\ell},\mathcal{O}_{\mathcal{K}_{\lambda}}}(\overline{r}_{\lambda}) := 
R^{\chi|{\operatorname{Gal}_{L^+}}, \tt{crys}}_{S_{\ell}\langle L^+\rangle,\mathcal{O}_{\mathcal{K}_{\lambda}}}(\overline{r}_{\lambda}|{\operatorname{Gal}_{L^+}})
\cong \mathcal{O}_{\mathcal{K}_{\lambda}}.
\end{equation}
\end{enumerate}
\end{Def}
We remark that the first condition is rather harmless and affects only a failure set of Dirichlet density $0$, cf. Assumption \ref{20170228_nec_assumptions}.
We also remark that in the second condition, we consider $\overline{r}_{\lambda}$ as a representation with values in the residue field $k_{\mathcal{O}_{\mathcal{K}_{\lambda}}}$ 
of ${\mathcal{O}_{\mathcal{K}_{\lambda}}}$ instead of $k_{\lambda}$.

With respect to a pre-admissible extension $L^+$, define
$\operatorname{Proc}(L^+)\subset \Lambda^1_{\mathcal{E}}$ as the subset of those
$\lambda$ which are $L^+$-procurable. Then we have:
\begin{Thm}\label{20170308_firstmaintechthm}
There exists a nested sequece $F^+ = L^+_0 \subset L^+_1 \subset \ldots$ of pre-admissible
extensions of $F^+$ such that 
\begin{equation}\label{20170307_limit_sequence}
\operatorname{lim}_{i\rightarrow \infty} \; \delta\left( \bigcup^i_{j=1} 
\operatorname{Proc}(L^+_j)\right) = 1,
\end{equation}
where $\delta(\Delta)$ denotes the density of those rational primes $q$ for which each $\lambda\in \operatorname{Pl}_{\mathcal{E}}$ above $q$ fulfills $\lambda\in \Delta$.
\end{Thm}
\begin{proof}
Let us first introduce another notation: With respect to a pre-admissible extension $L^+$, we say that $\lambda\in \Lambda^1_{\mathcal{E}}$ is $L^+$-$\star$-\textit{procurable}, if the following list is met (where, as usual, we abbreviate $\ell = \ell(\lambda)$):
\begin{enumerate}
\item[$\star_1$)] $\ell$ is not divisible by any element of $S$;
\item[$\star_2$)] $\ell$ is unramified in the extension $L|\mathbb{Q}$;
\item[$\star_3$)] all places of $L$ above $S_{\ell}$ are split in the extension $L|L^+$;
\item[$\star_4$)] the base change $\pi_L$ of $\pi$ to $L$ remains cuspidal;
\item[$\star_5$)] if $\nu\in \operatorname{Pl}_A$ lies above $S$, then $\pi_L$ admits a non-trivial Iwahori fix-vector. 
\end{enumerate}
As above, this defines a subset $\operatorname{Proc}^{\star}(L^+)\subset \Lambda^1_{\mathcal{E}}$. (Observe that condition $\star_4$ does not depend on $\lambda$, but 
we intentionally include it in the list. So, if $\Pi_L$ fails to be cuspidal, we have
$\operatorname{Proc}^{\star}(L^+)= \emptyset$.)\\
\textbf{Claim 1:} $\operatorname{Proc}^{\star}(L^+)-\operatorname{Proc}(L^+)$ is 
finite.\\
\textit{Proof of Claim 1.} We can suppose that $\operatorname{Proc}^{\star}(L^+)$ is not empty (otherwise the claim is trivially true), so in particular that $\pi_L$ is a RACSDC representation and there exists a unitary group and an avatar $\Pi_L$ over $L$. Now, for each
$\lambda\in \operatorname{Proc}^{\star}(L^+)$ we pick an $L$-big enough field extension 
$\mathcal{K}_{\lambda}$ of $\mathcal{E}_{\lambda}$. We consider the complex Hecke algebra
$^{\mathcal{O}_{\mathcal{E}}}\mathbf{T}^S_{\boldsymbol\omega}(U)$ and
the $\ell$-adic model $\mathbb{T}:= \,
^{\mathcal{O}_{\mathcal{E}}}\mathbb{T}^{S_{\ell}}_{\boldsymbol\omega}(U)$.

Write $\Pi_{L} = \langle f \rangle$ for the unitary avatar of the base
change of $\pi$ to $L$ and for a suitable choice $f\in \mathcal{S}_{\boldsymbol\omega}(U)$. We see that $\overline{\rho}_{\lambda}|\operatorname{Gal}_L$ equals the reduction of the representation attached to the maximal ideal $\mathfrak{m} = 
\operatorname{ker}(\varphi_{f^{(\lambda)}})\subset 
\mathbb{T}$ by Proposition \ref{20170307_Prop_attach_reps}, where
$f^{(\lambda)}$ is the $\ell$-adic model of $f$.

Recalling that we presume Assumption \ref{20170228_nec_assumptions},
we see easily that the preconditions of Theorem \ref{20170210_main_rttheorem}
hold for almost all of these choice of $L|L^+, \ell = \ell(\lambda), U, \boldsymbol\omega, \mathcal{E}(U), \mathcal{K}_{\lambda}$ and $\mathfrak{m}$: 
The main issue is the adequateness of $\overline{\rho}(\operatorname{Gal}_{F(\zeta_{\ell})})$, which follows from 
\cite[Proposition 2.1.2]{BLGGT} as long as $\ell > 2(n+1)$.
Thus, the desired isomorphism (\ref{20170307_iso_to_prove})
follows for allmost all $\lambda$ in $\operatorname{Proc}^{\star}(L^+)$
by Corollary \ref{20170307_TO_corollary}. This completes the proof 
of the claim.
\hfill$\clubsuit$\\

Consequently, it suffices to show that there exists a nested sequence 
$F^+ = L^+_0 \subset L^+_1 \subset \ldots$ of pre-admissible extensions
of $F^+$ such that equation (\ref{20170307_limit_sequence}) holds 
with $\operatorname{Proc}^{\star}$ instead of $\operatorname{Proc}$.
For the construction of these extensions, we define the set 
\begin{equation*}
\Theta_F := \{ d\in \mathbb{N} \,|\, \sqrt{d}\notin F,
\text{ the base change } \pi \leadsto \pi_{F(\sqrt{d})} 
\text{ remains cuspidal }
\}.
\end{equation*}
This set is not empty (cf. \cite[Thm 4.2]{ArthurClozel}), so choose a $d_1 \in 
\Theta_F$ and take $L^+_1 = F^+(\sqrt{d_1})$.\\
\textbf{Claim 2:} $L^+_1$ is pre-admissible.\\
\textit{Proof of Claim 2.} 
The extension $L^+_1|F^+$ is automatically Galois and solvable 
because $[L^+_1:F^+]=2$. Thus we are left to check that 
$L_1|L^+_1$ is unramified everywhere. For this, we use the
identity of the discriminants 
$\Delta_{L_1|F^+} = \Delta_{L^+_1|F^+}\Delta_{F|F^+}
= \Delta_{L^+_1|F^+}$ (cf. \cite[Exercise 3 on p. 51]{Janusz}). Consider the
diagram of field extensions 
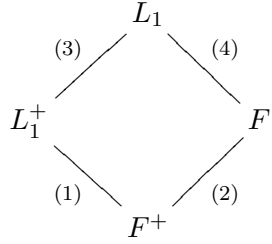
\begin{equation*}
\xymatrix{
& L_1\ar@{-}_{(3)}[dl]\ar@{-}^{(4)}[dr] &\\
L^+_1 && F\\
&F^+\ar@{-}^{(1)}[ul]\ar@{-}_{(2)}[ur]
}
\end{equation*}
If now $w$ is a prime of $L^+_1$ that ramifies in $(3)$, then the prime
$v$ of $F^+$ which lies below $w$ must ramify in the extension 
$L_1|F^+$. But then $v$ divides 
$\Delta_{L^+_1|F^+}= \Delta_{L_1|F^+}$, i.e. $v$ ramifies in $(1)$.
This implies that $v$ has ramification index $4$ in the extension
$L_1|F^+$. But in $(2)$, $v$ is unramified by the prerequisites, so it 
can at most ramify in $(4)$, yielding a ramification index of $2$ 
in $L_1|F^+$. This contradicts the assumption that $w$ ramifies
in $(3)$.
\hfill$\clubsuit$\\

\textbf{Claim 3:} $\delta(\operatorname{Proc}^{\star}(F^+_1)) \geq
\frac{1}{2}$.\\
\textit{Proof of Claim 3.} 
We check which $\lambda$ fail the list $\star_1$ - $\star_5$:
\begin{itemize}
\item Concerning $\star_1$ and $\star_2$, we have to exclude the
finitely many places $\lambda$ for which $\ell(\lambda)$ is not
coprime to $S$ or ramifies in $L^+_1|\mathbb{Q}$;
\item By an estimation based on Chebotarev's density thenorem (postponed as Lemma \ref{20170308_Lemma_on_prime_densities} to
the appendix), the density
of those $\ell$ which fulfill the condition that all primes of $L_1^+$
above $\ell$ are split in the extension $L_1|L_1^+$ is at least 
$\frac{1}{2}$;
\item Condition $\star_4$ is universally fulfilled by our 
choice of $L^+_1$;
\item Concerning condition $\star_5$, we remark that by local-global
compatibility (cf. \cite[Theorem 1.4]{CH13} and the references therein)
$\pi_L$ admits an Iwahori-fixed vector if $\rho|\operatorname{Gal}_L$
has unipotent ramification at $\nu$ \cite[(4.3.6) Proposition]{Wed08}. Thus,
condition $\star_5$ follows immediately from Assumption 
\ref{20170228_nonec_assumptions}.
\end{itemize}
\begin{flushright}
$\clubsuit$
\end{flushright}
For the next tower step we take $F_2^+ = F_1^+(\sqrt{d_2})$ for 
some $d_2\in \Theta_{F_2}$. It is again easy to check that 
$\Theta_{F_2}\neq \emptyset$  and that $F_2^+$ is pre-admissible. 
As in the proof of Claim 3, the statement of Lemma \ref{20170308_Lemma_on_prime_densities} implies 
$\delta(\operatorname{Proc}^{\star}(F^+_2)) \geq
\frac{3}{4}$. 
Iterating this construction of quadratic extensions we end up with
a nested sequence of pre-admissible fields $F^+_j$ such that
\begin{equation*}
\delta\left( \bigcup^i_{j=1} 
\operatorname{Proc}^{\star}(L^+_j)\right) \geq 
\delta(  
\operatorname{Proc}^{\star}(L^+_i)) \geq 1-\frac{1}{2^i} 
\;\underset{i\rightarrow \infty}{\longrightarrow} \;1.
\end{equation*}
Together with Claim 1, this concludes the proof of 
Theorem \ref{20170308_firstmaintechthm}.
\end{proof}
We now give a slight variant of the above:
\begin{Def}
With regard to a pre-admissible extension $L^+$ of $F^+$, we 
say that $\lambda\in \Lambda^1_{\mathcal{E}}$ is $L^+$-$\sharp$-\textit{procurable} if the restriction of $\overline{\rho}_{\lambda}$
to $\operatorname{Gal}_L$ (with $L=F.L^{+}$) remains absolutely 
irreducible and if there is an isomorphism 
\begin{equation}\label{20170309_another_aux_iso}
^{L^+}R^{\boxempty, \chi, \tt{crys}}_{\lambda} \cong W(k_{\lambda})\llbracket x_1, 
\ldots, x_u\rrbracket,
\end{equation}
where $^{L^+}R^{\boxempty,\chi,  \tt{crys}}_{\lambda} = \,
^{L^+}R^{\boxempty, \chi, \tt{crys}}_{S_{\ell}, W(k_{\lambda})}(\overline{r}_{\lambda})$ and $u= 
\operatorname{dim}(\mathfrak{g}_n^{\tt{der}}) = n^2$. 
The set of all $\lambda$ which are 
$L^+$-$\sharp$-procurable
is denoted by $\operatorname{Proc}^{\sharp}(L^+)$.
\end{Def}
\begin{Cor}
There exists a nested sequece $F^+ = L^+_0 \subset L^+_1 \subset \ldots$ of pre-admissible
extensions of $F^+$ such that 
\begin{equation*}
\operatorname{lim}_{i\rightarrow \infty} \; \delta\left( \bigcup^i_{j=1} 
\operatorname{Proc}^{\sharp}(L^+_j)\right) = 1.
\end{equation*}
\end{Cor}
\begin{proof}
For $i\in \mathbb{N}$, denote $\Delta_i = \cup_{j\leq i} 
\operatorname{Proc}(L_j^+)$. Also fix for each $\lambda\in \Delta_i$ 
some $j\leq i$ such that $\lambda\in \operatorname{Proc}(L^+_j)$.
Denote the corresponding field extension from the proof of 
Theorem \ref{20170308_firstmaintechthm} by 
$L_{(\lambda)} = L^+_{(\lambda)}.F$. By Theorem \ref{20170308_firstmaintechthm}, for such a $\lambda\in \Delta_i$ we have the identity (\ref{20170307_iso_to_prove}) for a suitable
extension $\mathcal{O}_{\mathcal{K}_{\lambda}}$ of 
$W(k_{\lambda})$. The third part of Proposition \ref{20170525_rep_res_proposition} then yields
\begin{equation*}
^{L^+_{(\lambda)}}R^{\boxempty, \chi, \tt{crys}}_{S_{\ell}, \mathcal{O}_{\mathcal{K}_{\lambda}}}(\overline{r}_{\lambda}) \cong \mathcal{O}_{\mathcal{K}_{\lambda}}\llbracket x_1, 
\ldots, x_u\rrbracket.
\end{equation*}
Thus, we can use Lemma \ref{20170405_lemma_res_of_coeff_ring} (and, if 
necessary, Remark \ref{20170406_remark_extended_conds}) to deduce the desired isomorphism (\ref{20170309_another_aux_iso}).
\end{proof}
\begin{Cor}\label{20170311_some_corollary}
There exists a subset $\Lambda^2_{\mathcal{E}} \subset
\Lambda^1_{\mathcal{E}}$ of Dirichlet density $1$ such that 
for each $\lambda\in \Lambda^2_{\mathcal{E}}$ there exists 
a finite, totally real extension $L^+_{(\lambda)}$ of $F$ and an
isomorphism 
\begin{equation*}
^{L^+_{(\lambda)}}R^{\boxempty_{S_{\ell}\langle L^+_{(\lambda)}\rangle}, \chi, \tt{crys}}_{S_{\ell}\langle L^+_{(\lambda)}\rangle, W(k_{\lambda})}(\overline{r}_{\lambda}) \cong 
W(k_{\lambda})\llbracket x_1, 
\ldots, x_{w(\lambda)}\rrbracket
\end{equation*}
with $w(\lambda) = (n^2+1).\#S_{\ell}\langle L^+_{(\lambda)}\rangle - 1$.
\end{Cor}
\begin{proof}
This follows directly from Proposition \ref{20170525_rep_res_proposition}.
\end{proof}
Next, we will apply the framework of Section \ref{20170405_section_on_framework} to the 
attained $\lambda$:
\begin{Thm}\label{20170314_key_thm_for_proof_of_6_2}
There exists a cofinite subset $\Lambda^3_{\mathcal{E}} \subset
\Lambda^2_{\mathcal{E}}$ such that the following holds: 
Let $\lambda\in \Lambda^3_{\mathcal{E}}$ and 
$L^+_{(\lambda)}$ the corresponding extension from 
Corollary \ref{20170311_some_corollary}. Then the functors 
\begin{equation*}
^{L^+_{(\lambda)}}D^{\boxempty_{S_{\ell}\langle L^+_{(\lambda)}\rangle}, \chi, \tt{min}}_{S_{\ell}\langle L^+_{(\lambda)}\rangle, W(k_{\lambda})}(\overline{r}_{\lambda}) 
\quad \text{ and }\quad
^{L^+_{(\lambda)}}D^{\boxempty_{S_{\ell}\langle L^+_{(\lambda)}\rangle}, \chi}_{S_{\ell}\langle L^+_{(\lambda)}\rangle, W(k_{\lambda})}(\overline{r}_{\lambda}) 
\end{equation*}
have vanishing dual Selmer group.
\end{Thm}
\begin{proof}
We start with the $\tt{min}$-case:
When applying the framework, we take for $\texttt{sm}$ the condition
parametrizing arbitrary deformations, for $\texttt{crys}$ the condition
parametrizing FL-crystalline deformations (cf. Section \ref{20170321_sect_crys_defos})
and  for $\texttt{min}$ the condition
parametrizing minimally ramified deformations (cf. Section \ref{20170321_sect_min_ram}). Moreover, we take $\chi = \epsilon_{\ell}^{1-n}
\delta_{F|F^+}^{n (\text{mod} 2)}$. Let us now
check the following list of conditions (and we abbreviate 
$L^+ = L^+_{(\lambda)}$ as we check this for a fixed $\lambda\in 
\Lambda^2_{\mathcal{E}}$):
\begin{enumerate}
\item \textbf{(sm/$k$)}: As we took for $\texttt{sm}$ the unrestricted deformation 
condition, we have to check that for each $\nu\in \Omega_{\ell}$ the functor
$^{L^+_{\nu}}D^{\boxempty, \chi_{\nu}}_{W(k_{\lambda})}(\overline{r}_{\lambda, \nu})$ is representable and that
the representing object is formally smooth of relative dimension 
\begin{equation*}
d^{\boxempty, \tt{sm}}_{\nu} = \operatorname{dim}(\mathfrak{g}^{\tt{der}}_{n})([L_{\tilde{\nu}}:\mathbb{Q}_{\ell}] + 1) = n^2([L_{\tilde{\nu}}:\mathbb{Q}_{\ell}] + 1)
= n^2([L^+_{{\nu}}:\mathbb{Q}_{\ell}] + 1).
\end{equation*}
(This also amounts to the vanishing of the error terms $\delta_{\nu}$ in Theorem 
\ref{20160908_MainThm}.)\\
\underline{Check}: Representability was already remarked in Section \ref{20170405_section_liftings_and_defos}.
For the remaining claim, we first refer to 
Proposition \ref{20170214_prop_local_split_defos} in order to get an isomorphism
\begin{equation*}
^{L^+_{\nu}}R^{\boxempty, \chi_{\nu}}_{W(k_{\lambda})}(\overline{r}_{\lambda, \nu})
\cong \,
^{L_{\tilde{\nu}}}D^{\boxempty, \chi_{\nu}}_{W(k_{\lambda})}(\overline{\rho}_{\lambda, \nu}).
\end{equation*}
Now the claim follows from Theorem \ref{20170313_Theorem_smoothness_sm_over_ell}.
\item \textbf{(crys)}: For each $\nu\in\Omega_{\ell}$, the subfunctor 
\begin{equation*}
^{L^+_{\nu}}D^{\boxempty, \chi_{\nu},\tt{crys}}_{W(k_{\lambda})}(\overline{r}_{\lambda, \nu})
\hookrightarrow\,
^{L^+_{\nu}}D^{\boxempty, \chi_{\nu}}_{W(k_{\lambda})}(\overline{r}_{\lambda, \nu})
\end{equation*}
is relatively representable and the representing object is formally smooth
of relative dimension $d^{\boxempty, \tt{crys}}_{\nu} = \operatorname{dim}(\mathfrak{g}^{\tt{der}}_{n}) + 
(\operatorname{dim}(\mathfrak{g}^{\tt{der}}_{n}) - \operatorname{dim}(\mathfrak{b}^{\tt{der}}_{n}))[L^+_{\nu}:\mathbb{Q}_{\ell}]$, where $\mathfrak{b}_n$ denotes the Lie algebra
of a Borel subgroup of $\mathcal{G}_n$.\\
\underline{Check}: By definition, we have 
\begin{equation*}
^{L^+_{\nu}}R^{\boxempty, \chi_{\nu},\tt{crys}}_{W(k_{\lambda})}(\overline{r}_{\lambda, \nu})
\cong \,
^{L_{\tilde{\nu}}}D^{\boxempty, \chi_{\nu},\tt{crys}}_{W(k_{\lambda})}(\overline{\rho}_{\lambda, \nu}).
\end{equation*}
Thus, the claim follows from Lemma \ref{20150712_crystallinity_lemma}.
\item \textbf{(min)}: For each $\nu\in S$, the subfunctor 
\begin{equation*}
^{L^+_{\nu}}D^{\boxempty, \chi_{\nu},\tt{min}}_{W(k_{\lambda})}(\overline{r}_{\lambda, \nu})
\hookrightarrow\,
^{L^+_{\nu}}D^{\boxempty, \chi_{\nu}}_{W(k_{\lambda})}(\overline{r}_{\lambda, \nu})
\end{equation*}
is relatively representable and the representing object is formally smooth
of relative dimension $d^{\boxempty, \tt{min}}_{\nu} = \operatorname{dim}(\mathfrak{g}^{\tt{der}}_{n})$.\\
\underline{Check}: Again, by definition, we have 
\begin{equation*}
^{L^+_{\nu}}R^{\boxempty, \chi_{\nu},\tt{min}}_{W(k_{\lambda})}(\overline{r}_{\lambda, \nu})
\cong \,
^{L_{\tilde{\nu}}}D^{\boxempty, \chi_{\nu},\tt{min}}_{W(k_{\lambda})}(\overline{\rho}_{\lambda, \nu}).
\end{equation*}
Thus, the claim follows from Proposition \ref{20170526_newProponmin}.
\item \textbf{($\infty$)}: For each $\nu\in \Omega_{\infty}$, the local deformation
ring $^{L^+_{\nu}}R^{\boxempty, \chi_{\nu}}_{W(k_{\lambda})}(\overline{r}_{\lambda, \nu})$ is 
formally smooth of relative dimension $d_{\nu}^{\boxempty} = 
\operatorname{dim}(\mathfrak{b}^{\tt{der}}_n)$.\\
\underline{Check}: This was already used, see Remark \ref{20170406_remark_on_mu_m}.
\item \textbf{(Presentability)}: Consider the ring
\begin{equation*}
^{L^+}R^{\tt{loc, min}} := \widehat{\bigotimes}_{\nu\in S_{\ell}\langle L^+\rangle} \, 
^{L^+}\widetilde{R}_{\nu}
\end{equation*}
with  
$\,^{L^+}\widetilde{R}_{\nu} = \,^{L^+_{\nu}}D^{\boxempty, \chi_{\nu}, \tt{min}}_{W(k_{\lambda})} (\overline{r}_{\lambda, \nu})$ if $\nu\in S$ and 
$\,^{L^+_{\nu}}D^{\boxempty, \chi_{\nu}}_{W(k_{\lambda})} (\overline{r}_{\lambda, \nu})$ otherwise.
Then, there exists a presentation 
\begin{equation*}
^{L^+}R^{\boxempty_{S_{\ell}}, \chi}_{S_{\ell}, W(k_{\ell})}(\overline{r}_{\lambda})
\cong\,
^{L^+}R^{\tt{loc, min}}\llbracket X_1, \ldots, X_a\rrbracket /(f_1, \ldots, f_b)
\end{equation*}
with $a-b = (\#S_{\ell}\langle L^+\rangle - 1).\operatorname{dim}(\mathfrak{g}^{\tt{ab}}_n)$.\\
\underline{Check}: This is the content of Proposition \ref{20170329_first_rep_result}, but we have to
check Assumption \ref{20150904_Ass_H0_vanishing_for_presentations}. As $\mathfrak{g}_n^{\tt{der}} = \mathfrak{gl}_n$, 
this condition holds by Corollary \ref{20150904_newH1dualvanishing} for almost all $\lambda$.
\item \textbf{($R=T$)}: The ring $^{L^+}R^{\boxempty_{S_{\ell}}, \chi, \tt{min, crys}}_{S_{\ell}, W(k_{\ell})}(\overline{r}_{\lambda})$ is formally smooth of relative dimension 
$r_0 = \operatorname{dim}(\mathfrak{g}).\#S_{\ell}\langle L^+\rangle - 
\operatorname{dim}(\mathfrak{g}^{\tt{ab}})$.\\
\underline{Check}: This follows from Corollary \ref{20170311_some_corollary}.
\end{enumerate}
We see that the general requirements of Theorem \ref{20160908_MainThm} are met, so let us 
check the additional requirements of part 2. of Theorem \ref{20160908_MainThm}:
\begin{itemize}
\item[a)] The condition $\ell \gg 0$ can be achieved by leaving out finitely many $\lambda$;
\item[b)] The vanishing of $H^0(\operatorname{Gal}_{L^+}, \mathfrak{g}_n^{\tt{der}, \vee})$ can
be checked by observing
\begin{equation*}
H^0(\operatorname{Gal}_{L^+}, \mathfrak{g}_n^{\tt{der}, \vee}) \subset
H^0(\operatorname{Gal}_{L}, \mathfrak{g}_n^{\tt{der}, \vee}) \cong 
H^0(\operatorname{Gal}_{L}, \mathfrak{gl}_n^{\vee}),
\end{equation*}
as the adjoint representation of $\operatorname{Gal}_L$ on $\mathfrak{g}_n^{\tt{der}}$
(via $\overline{r}_{\lambda}$) corresponds to the adjoint representation of 
$\operatorname{Gal}_L$ on $\mathfrak{gl}_n$ (via $\overline{\rho}_{\lambda}$), cf. 
\cite[Section 2.1]{CHT}. Thus, the desired vanishing follows for almost all $\lambda$ by 
Corollary \ref{20150904_newH1dualvanishing}.
\item[c)] For $\nu\in S\langle L^+\rangle$, $\operatorname{dim}(L_{\lambda,\nu}) = 
h^0(\operatorname{Gal}_{L^+_{\nu}}, \mathfrak{g}^{\tt{der}}_n)$: As $\nu$ is split, 
Proposition \ref{20170214_prop_local_split_defos} yields 
$h^0(\operatorname{Gal}_{L^+_{\nu}}, \mathfrak{g}^{\tt{der}}_n) =
h^0(\operatorname{Gal}_{L_{\tilde{\nu}}}, \mathfrak{gl}^{\tt{der}}_n)$, where the
action on $\mathfrak{gl}_n$ is via $\overline{\rho}_{\lambda, \tilde{\nu}}$. The claim thus
follows from \cite[Corollary 2.4.21]{CHT}.
\end{itemize}
The finitely many exclusions which occurred in the above items are now the places we 
must exclude from $\Lambda^2_{\mathcal{E}}$ to get 
$\Lambda^3_{\mathcal{E}}$. This finishes the first part,
i.e. that $^{L^+_{(\lambda)}}D^{\boxempty_{S_{\ell}\langle L^+_{(\lambda)}\rangle}, \chi, \tt{min}}_{S_{\ell}\langle L^+_{(\lambda)}\rangle, W(k_{\lambda})}(\overline{r}_{\lambda}) $ has vanishing dual Selmer group.\\
Concerning the second statement (i.e. the claimed vanishing of the non-minimal dual Selmer group) we first note that on each level $L^+_{(\lambda)}$ we can apply the $R=R^{\tt{min}}$-result of Corollary \ref{20170314_cor_on_min_is_all}, yielding the desired vanishing 
except for a finite failure set. In other words: Fix a place $\lambda'$, then we have
\begin{equation*}
^{L^+_{(\lambda)}}D^{\boxempty_{S_{\ell}\langle L^+_{(\lambda)}\rangle}, \chi, \tt{min}}_{S_{\ell}\langle L^+_{(\lambda)}\rangle, W(k_{\lambda})}(\overline{r}_{\lambda}) 
= \,
^{L^+_{(\lambda)}}D^{\boxempty_{S_{\ell}\langle L^+_{(\lambda)}\rangle}, \chi}_{S_{\ell}\langle L^+_{(\lambda)}\rangle, W(k_{\lambda})}(\overline{r}_{\lambda}) 
\end{equation*}
for all $\lambda$ with $L^+_{(\lambda)} = L^+_{(\lambda')}$, except for a finite
failure set 
$\mathfrak{F}_{\lambda'}$. We should check that the occurrence of these failure sets at
each step in the tower of field extensions 
does not disturb the desired result. For this, recall that the $L^+_{(\lambda)}$ show up
in the tower $F^+ = L^+_0 \subset L^+_1 \subset \ldots$ and that, by the first statement, we have
\begin{equation*}
\operatorname{lim}_{i\rightarrow \infty} \; \delta\Bigl\{ \lambda \Bigl| \;
^{L^+_{(\lambda)}}D^{\boxempty_{S_{\ell}\langle L^+_{(\lambda)}\rangle}, \chi, \tt{min}}_{S_{\ell}\langle L^+_{(\lambda)}\rangle, W(k_{\lambda})}(\overline{r}_{\lambda}) \text{ has
vanishing dual Selmer group, } L^+_{(\lambda)} \subset L^+_i \Bigr\} = 1.
\end{equation*}
But this clearly implies
\begin{equation*}
\operatorname{lim}_{i\rightarrow \infty} \; \delta\Bigl\{ \lambda \Bigl| \;
^{L^+_{(\lambda)}}D^{\boxempty_{S_{\ell}\langle L^+_{(\lambda)}\rangle}, \chi}_{S_{\ell}\langle L^+_{(\lambda)}\rangle, W(k_{\lambda})}(\overline{r}_{\lambda}) \text{ has
vanishing dual Selmer group, } L^+_{(\lambda)} \subset L^+_i, \lambda\notin\mathfrak{F}_{\lambda}  \Bigr\} = 1,
\end{equation*}
completing the proof.
\end{proof}
\begin{proof}[Proof of Theorem \ref{20170301_weaker_main_thrm}.]
The \glqq has vanishing dual Selmer group\grqq-part of
Theorem \ref{20170301_weaker_main_thrm} follows immediately from Theorem \ref{20170314_key_thm_for_proof_of_6_2} and the potential unobstructedness
result of Lemma \ref{20170404_pot_unobstr},
as each $[L^+_{(\lambda)}:F^+]$ is a power of $2$ (and $k_{\lambda}$ has odd characteristic
for $\lambda\in \Lambda^2_{\mathcal{E}}$).\\
It remains to show that for all $\nu\in \Omega^{F^+}_{\ell}$ the local lifting ring
$R^{\boxempty, \chi_{\nu}}_{W(k_{\lambda})}(\overline{r}_{\lambda, \nu})$ is relatively
smooth.
By Theorem \ref{20170313_Theorem_smoothness_sm_over_ell} we know that
\begin{equation*}
\operatorname{lim}_{i\rightarrow \infty} \; \delta\Bigl\{ \lambda \Bigl| \;
R^{\boxempty, \chi_{\nu}}_{W(k_{\lambda})}(\overline{r}_{\lambda}|
\operatorname{Gal}_{L^+_{(\lambda)}, \nu'}) 
\text{ is unobstructed for all }\nu'\in \Omega_{\ell(\lambda)}^{L^+_{(\lambda)}},
 L^+_{(\lambda)} \subset L^+_i \Bigr\} 
\end{equation*}
\begin{equation*}
= \operatorname{lim}_{i\rightarrow \infty} \; \delta\Bigl\{ \lambda \Bigl| \;
\text{ any }\nu'\in \Omega_{\ell(\lambda)}^{L^+_{(\lambda)}}\text{ is split in the extension }
L_{(\lambda)}|L^+_{(\lambda)},
 L^+_{(\lambda)} \subset L^+_i \Bigr\} = 1.
\end{equation*}
Using Proposition \ref{20170405_prop_easy_pot_unobstr} and Corollary
\ref{20170314_cor_317_phd}, the claim follows.
\end{proof}

\section*{Appendix A: A lemma on prime densities in 
non-Galois extensions}
Let us consider a CM-field $F$ with totally real subfield
$F^+$ and we denote by $L^+= F^+(\sqrt{d_1}, \ldots, \sqrt{d_k})$
the totally real extension of $F^+$ of degree $2^k$, obtained by 
adjoining the square roots of some choice of elements
$d_1, \ldots, d_k\in \mathbb{N}$ such that each $d_i$ is a non-square
in the Galois closure $\widetilde{F}^+$ of $F^+$. 
Set $L=L^+.F$.
Then we have:
\begin{thmx}
\label{20170308_Lemma_on_prime_densities}
Let $\Xi_{\mathbb{Q}}$ be the set of all those rational primes $\ell$
with the following property: For any place $\wp$ of $L^+$, 
\begin{equation*}
\left[\wp \text{ divides } \ell \right] \; \Longrightarrow \;
\left[ \wp \text{ splits in } L|L^+ \right].
\end{equation*}
Then the density $\delta(\Xi_{\mathbb{Q}})$ of $\Xi_{\mathbb{Q}}$ in the
set of all rational primes is at least $1-\frac{1}{2^k}$.
\end{thmx}
\begin{proof}
Consider the following diagram of fields
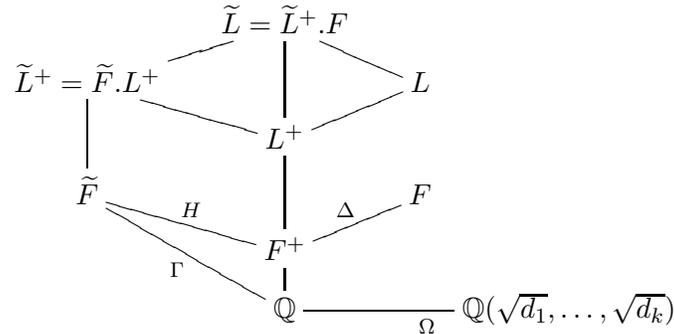
\begin{equation*}
\xymatrix@C=0.5em@R=0.5em{
&&\widetilde{L} = \widetilde{L}^+.F\ar@{-}[dll]\ar@{-}[drr]\\
\widetilde{L}^+ = \widetilde{F}.L^+ &&&&L\\
&&L^+\ar@{-}[ull]\ar@{-}[urr]\ar@{-}[uu]\\
\widetilde{F}\ar@{-}[uu] &&&& F\\
&&F^+\ar@{-}[ull]_{H}\ar@{-}[urr]^{\Delta}\ar@{-}[uu]\\
&&\mathbb{Q}\ar@{-}[u]\ar@{-}[rrr]_{\Omega}\ar@{-}[uull]^{\Gamma}&&&\mathbb{Q}(\sqrt{d_1}, \ldots, \sqrt{d_k})
}
\end{equation*}
with corresponding Galois groups $\Delta = \mathbb{Z}/2\mathbb{Z}$, 
$\Omega = (\mathbb{Z}/2\mathbb{Z})^k$ and $\Gamma, H$ (for which we
don't make an assumption). By our initial assumption that the 
$d_i$ are not squares we have
\begin{equation*}
\operatorname{Gal}(\widetilde{L}^+|\mathbb{Q}) \cong
\Gamma \times \Omega \quad \text{ and, hence, }\quad
\operatorname{Gal}(\widetilde{L}|\mathbb{Q}) \cong
\Gamma \times \Omega \times \Delta.
\end{equation*}
Now, let $\mathfrak{P}$ be a place of $\widetilde{L}$ with 
corresponding Frobenius element $(\gamma, \omega, \delta) \in 
\operatorname{Gal}(\widetilde{L}|\mathbb{Q})$. As $\Omega$ and $\Delta$
are abelian, the conjugacy class of $\mathfrak{P}$ can be written 
as $\{(u\gamma u^{-1}, \omega, \delta) | u\in \Gamma\}$ and 
consists precisely of the Frobenii of the places of $L$ lying 
over the same rational prime $p$ as $\mathfrak{P}$. Let $\wp$ be
the place of $L$ below $\mathfrak{P}$. Its Frobenius element is 
given by 
\begin{equation*}
(\gamma, \omega, \delta)^{e_{\gamma, \omega, \delta}} \in 
H\times \{1\} \times \Delta = \operatorname{Gal}(\widetilde{L}|
L^+)
\end{equation*}
for $e_{\gamma, \omega, \delta}$ minimal such that 
$(\gamma, \omega, \delta)^{e_{\gamma, \omega, \delta}} \in 
H\times \{1\} \times \Delta$. The condition that $\wp$ splits in
$L|L^+$ then amounts precisely to
$(\gamma, \omega, \delta)^{e_{\gamma, \omega, \delta}} \in 
H\times \{1\} \times \{1\}$, or, written in a more sophisticated way, 
that $q((\gamma, \omega, \delta)^{e_{\gamma, \omega, \delta}}) = 1$, 
where
\begin{equation*}
q: \operatorname{Gal}(\widetilde{L}|L^+) \rightarrow 
\operatorname{Gal}(\widetilde{L}|L^+)/\operatorname{Gal}(\widetilde{L}|
\widetilde{L}^+)
\end{equation*}
is the quotient map.
If $\omega\neq 1$, we clearly must have $2|e_{\gamma, \omega, \delta}$, 
which imples that $\wp$ splits in $L|L^+$. It is also important to note
that the condition $\omega\neq 1$ is not destroyed by conjugation inside
$\operatorname{Gal}(\widetilde{L}|\mathbb{Q})$. Now, set
\begin{equation*}
\Xi^{\ast} = \{(\gamma, \omega, \delta) \in 
\operatorname{Gal}(\widetilde{L}|\mathbb{Q})\;|\; 
q((\gamma, \omega, \delta)^{e_{\gamma, \omega, \delta}}) = 1 \}
\end{equation*}
and consider the subset $\Xi\subset \Xi^{\ast}$ which consists of those
$g\in \Xi^{\ast}$ for which the complete conjugacy class is contained
in $\Xi^{\ast}$, i.e. $\Xi = \{ g\in \Xi^{\ast} \,|\, \langle g \rangle
\subset \Xi^{\ast}\}$.\\
We can give another characterization of this set: $\Xi$ is the union
of all conjugacy classes $\langle g \rangle \subset 
\operatorname{Gal}(\widetilde{L}|\mathbb{Q})$ with the following 
property: If $\mathbf{P}_g$ denotes the set of all places $\mathfrak{P}$
of $\widetilde{L}$ such that $\operatorname{Frob}_{\mathfrak{P}} \in
\langle g \rangle$, then for any place $\wp$ of $L^+$ the following holds: 
\begin{equation*}
\left[\,\exists\, \mathfrak{P}\in \textbf{P}_g \text{ such that 
$\mathfrak{P}$ divides }\wp \,\right] \; \Longrightarrow \;
\left[ \wp \text{ splits in } L|L^+ \right].
\end{equation*}
Then we have 
\begin{equation*}
\# \Xi \geq \# \{ (\gamma, \omega, \delta) \in 
\operatorname{Gal}(\widetilde{L}|\mathbb{Q}) \,|\, \omega\neq 1\} = 
(2^k -1).2.\#\Gamma.
\end{equation*}
As $\Xi_{\mathbb{Q}} = \{\ell \in \operatorname{Pl}_{\mathbb{Q}} \,|\,
\exists \,g\in \Xi \text{ such that $\mathfrak{P}|\ell$ for all }
\mathfrak{P}\in \textbf{P}_g\,\}$, it follows from 
Chebotarev's density theorem that 
\begin{equation*}
\delta(\Xi_{\mathbb{Q}}) \geq \frac{(2^k -1).2.\#\Gamma}{\operatorname{Gal}(\widetilde{L}|\mathbb{Q})} = 
1 - \frac{1}{2^k}.\qedhere
\end{equation*}
\end{proof}

\let\Section\section 
\def\section*#1{\Section{#1}} 


\begin{thebibliography}{10}

\bibitem[All16]{Allen_polarized}
{\sc P. Allen}, \textit{On automorphic points in polarized deformation rings}, preprint, 2016.

\bibitem[AC89]{ArthurClozel}
{\sc J. Arthur and L. Clozel},
\textit{{Simple algebras, base change, and the advanced theory of the
trace formula}}, Annals of Mathematics Studies \textbf{120}, Princeton University Press, Princeton, NJ, 1989.

\bibitem[AM69]{AtiyahMacdo}
{\sc M. F. Atiyah and I. G. Macdonald}, \textit{Introduction to commutative algebra},
Addison-Wesley Publishing Co., Reading, Mass.-London-Don Mills, Ont., 1969.

\bibitem[BLGGT14]{BLGGT}
{\sc T. Barnet-Lamb, T. Gee, D. Geragthy and R. Taylor},
\textit{Potential automorphy and change of weight},
Ann. of Math. (2) \textbf{179} (2014),
no. 2, pp. 501-609.

\bibitem[BMR05]{LiebeckTesterman}
{\sc M. Bate, B. Martin and G. R\"ohrle}, \textit{A geometric approach to complete reducibility},
Invent. Math. \textbf{161} (2005), no. 1, 177-218.

\bibitem[BMRT10]{BMRT}
{\sc M. Bate, B. Martin, G. R\"ohrle and R. Tange}, \textit{Complete reducibility and separability},
Trans. Amer. Math. Soc. \textbf{362} (2010), no. 8, 4283-4311.


\bibitem[Bal12]{Balaji}
{\sc S. Balaji},
\textit{G-valued potentially semi-stable deformation rings}, Ph.D. thesis, 
The University of Chicago, Chicago, IL, 2012.

\bibitem[Bel09]{Bellaiche}
{\sc J. Bella\"iche}, \textit{Introduction into the conjecture of Bloch-Kato}, unpublished 
notes for the Clay Summer School on Galois representations, Honolulu, 2009.

\bibitem[BC09]{BellaicheChenevier}
{\sc J. Bella\"iche and G. Chenevier}, \textit{Families of Galois Representations and Selmer Groups},
Ast\'{e}risque \textbf{324}, Publications de la SMF, Paris, 2009.

\bibitem[BC03]{Bleher_Chinburg}
{\sc F. Bleher and T. Chinburg},
\textit{Deformations with respect to an algebraic group}, 
Illinois J. Math. \textbf{47} (2003), no. 3, pp. 899-919. 

\bibitem[B\"oc07]{Boeckle_presentations}
{\sc G. B\"ockle}, \textit{Presentations of universal deformation rings}, in 
 L-functions and Galois representations, London Math. Soc. Lecture Note Ser. \textbf{320}, Cambridge Univ. Press, Cambridge, 2007, pp. 24-58.

\bibitem[B\"oc13a]{Boeckle_CRM}
{\sc G. B\"ockle}, \textit{Deformations of Galois representations}, 
in Elliptic curves, Hilbert modular forms and Galois deformations, Adv. Courses Math. CRM Barcelona,
Basel, 2013, pp. 21-115.

\bibitem[CG13]{CalegariGee}
{\sc F. Calegari and T. Gee}, \textit{The conjectural connections between automorphic
representations and Galois representations}, to appear in the Proceedings of the 
LMS Durham Symposium 2011.

\bibitem[CH13]{CH13}
{\sc G. Chenevier and M. Harris}, \textit{Construction of automorphic {G}alois representations, {II}}, Camb. J. Math. \textbf{1}, 2013, pp. 53--73.

\bibitem[CHT08]{CHT}
{\sc L. Clozel, M. Harris and R. Taylor}, \textit{Automorphy for some l-adic lifts of automorphic mod l Galois representations}, Publ. Math. Inst. Hautes \'Etudes Sci. No. \textbf{108} (2008), 1-181.

\bibitem[CDT99]{CDT}
{\sc B. Conrad, F. Diamond and R. Taylor}, \textit{Modularity of certain potentially 
Barsotti-Tate Galois Representations}, J. Amer. Math. Soc. \textbf{12} (1999), no. 2, pp. 521-567.

\bibitem[Del71]{Deligne}
{\sc P. Deligne}, \textit{Formes modulaires et r\'epresentations de $\operatorname{GL}(2)$},
Lecture Notes in Math. \textbf{179} (1971), pp. 139--172.

\bibitem[DS74]{Deligne_Serre}
{\sc P. Deligne and J.-P. Serre}, \textit{Formes modulaires de poids 1}, Ann. Sci. ENS 
\textbf{7} (1974), pp. 507--530.

\bibitem[FL82]{FL}
{\sc J.-M. Fontaine and G. Laffaille}, \textit{Construction de repr\'esentations $p$-adiques},
Ann. Sci. \'Ecole Norm. Sup. (4) \textbf{15} (1982), no. 4, 547-608.

\bibitem[Fre64]{Freyd}
{\sc P. Freyd}, \textit{Abelian categories. An introduction to the theory of functors}, 
Harper's Series in Modern Mathematics, Harper \& Row, Publishers, New York, 1964.

\bibitem[FH91]{FultonHarris}
{\sc W. Fulton and J. Harris}, \textit{Representation Theory. A first course}, 
 Graduate Texts in Mathematics \textbf{129}, Readings in Mathematics, Springer-Verlag, New York, 1991.
 
\bibitem[Gam13]{Gamzon}
{\sc A. Gamzon}, \textit{Unobstructed Hilbert modular deformation problems}, 2013, to appear in 
Journal de Th\'eorie des Nombres des Bordeaux. 
 
\bibitem[Gee11]{GeePrescribed}
{\sc T. Gee}, \textit{Automorphic lifts of prescribed types},
Math. Ann. \textbf{350} (2011), no. 1, 107-144. 

\bibitem[GK14]{GeeKisin}
{\sc T. Gee and M. Kisin}, \textit{The Breuil-M\'ezard conjecture for potentially Barsotti-Tate representations}, Forum Math. Pi \textbf{2} (2014), e1, 56 pp.

\bibitem[Ger10]{Geraghty_thesis}
{\sc D. Geraghty}, \textit{Modularity lifting theorems for ordinary Galois representations}, Ph.D. thesis,
Harvard University, Cambridge, MA, 2010.

\bibitem[Gou01]{Gouvea}
{\sc F. Q. Gouv\^{e}a}, \textit{Deformations of Galois representations}, 
IAS/Park City Math. Ser. \textbf{9}, Arithmetic algebraic geometry (Park City, UT, 1999), Amer. Math. Soc., Providence, RI, 2001, pp. 233-406.

\bibitem[Gue11]{Guerberoff}
{\sc L. Guerberoff}, \textit{Modularity lifting theorems for Galois representations of unitary type},
Compos. Math. \textbf{147} (2011), no. 4, 1022-1058.

\bibitem[Gui16]{MyPhd}
{\sc D.-A. Guiraud}, \textit{A framework for unobstructedness of Galois deformation rings}, 
PhD thesis, Heidelberg Unversity, 2016. Available at {http://www.ub.uni-heidelberg.de/archiv/20248}.

\bibitem[Gro64]{Grothendieck_ega}
{\sc A. Grothendieck}, \textit{El\'ements de g\'eom\'etrie alg\'ebrique. IV. \'Etude locale des sch\'emas et des morphismes de sch\'emas. I.},
Inst. Hautes \'Etudes Sci. Publ. Math. No. 20, 1964.

\bibitem[HHM08]{HarrisMichael}
{\sc J. Harris, J. Hist and M. Mossinghoff}, \textit{ Combinatorics and graph theory. Second edition},
Undergraduate Texts in Mathematics. Springer, New York, 2008.

\bibitem[Hat15]{Hatley}
{\sc J. Hatley}, \textit{Obstruction criteria for modular deformation problems}, PhD thesis,
University of Massachusetts Amherst, Amherst, MA, 2015.



\bibitem[JK81]{JamesKerber}
{\sc G. James and A. Kerber},
\textit{The representation theory of the symmetric group}, Encyclopedia of Mathematics
and its applications \textbf{16}, Addison-Wesley Publishing Co., Reading, Mass., 1981.

\bibitem[Jan89]{Jannsen}
{\sc U. Jannsen}, \textit{On the $l$-adic cohomology of varieties over number fields and its 
Galois cohomology}, in: Galois groups over $\mathbf{Q}$ (Berkeley, CA, 1987), Math. Sci. Res.
Inst. Publ., \textbf{16}, Springer, New York, 1989, pp. 315--360.

\bibitem[Jan96]{Janusz}
{\sc G. Janusz}, \textit{Algebraic number fields. 
Second edition},  Graduate Studies in Mathematics \textit{7}, American Mathematical Society, Providence, RI, 1996.

\bibitem[Kha03]{Khare_isos}
{\sc C. Khare}, \textit{On isomorphisms between deformation rings and {H}ecke rings}, 
with an appendix by G. B\"ockle, Invent. Math. \textbf{154}, no. 1, 2003, pp. 199--222.

\bibitem[KW09a]{KW_on_serre}
{\sc C. Khare and J.-P. Wintenberger},
\textit{On {S}erre's conjecture for 2-dimensional mod {$p$}
representations of {${\rm Gal}(\overline{\Bbb Q}/\Bbb Q)$}}, 
Ann. of Math. (2), \textbf{(169)}, 2009, pp. 229--253.

\bibitem[KW09b]{KW2}
{\sc C. Khare and J.-P. Wintenberger},
\textit{Serre's modularity conjecture (II)},
Invent. Math. \textbf{178}, 2009, pp. 505-586.

\bibitem[Kis07]{Kisin2}
{\sc M. Kisin}, 
\textit{Modularity of 2-dimensional Galois representations},
Current Developments in Mathematics \textbf{2005}, 2007, 
pp. 191-230.

\bibitem[Lev13]{Levin}
{\sc B. Levin}, \textit{$G$-valued flat deformations and local models}, Ph.D. thesis, Stanford University, 
Stanford, CA, 2013.

\bibitem[Mau00]{MaugerThesis}
{\sc D. Mauger}, \textit{Alg\`ebre de Hecke quasi-ordinaire universelle d'un groupe r\'eductif}, Ph.D. thesis, Universit\'e Paris-Nord
(Paris 13), Paris, 2000.

\bibitem[Maz87]{Mazur_deforming}
{\sc B. Mazur}, \textit{Deforming Galois Representations}, in Galois groups over $\textbf{Q}$, Y. Ihara, K. Ribet, J.-P. Serre (editors),
MSRI Publ. \textit{16} (1987), Springer-Verlag, New York-Berlin, 1989, 385-437.

\bibitem[Maz97]{Mazur1}
{\sc B. Mazur}, \textit{An introduction to the deformation theory of Galois representations}, in Modular Forms and
Fermat's Last Theorem, G. Cornell, J. Silverman, and G. Stevens (editors), Springer-Verlag, New York-Berlin, 1997, pp. 243-311.



\bibitem[NSW08]{NSW}
{\sc J. Neukirch, A. Schmidt, K. Wingberg}, \textit{Cohomology of Number Fields, Second edition},
Grundlehren der Mathematischen Wissenschaften \textit{323}, Springer-Verlag, Berlin, 2008.

\bibitem[PT15]{PatrikisTaylor}
{\sc S. Patrikis and R. Taylor}, \textit{Automorphy and irreducibility of some 
$l$-adic representations}, Composition Math. \textbf{151}, 2015, p. 207.

\bibitem[PSW16]{PSW}
{\sc S. Patrikis, A. Snowden and A. Wiles}, \textit{Residual irreducibility of compatible
systems}, preprint, 2016.


\bibitem[Pra14]{Prasad}
{\sc D. Prasad}, \textit{Notes on representations of finite groups of Lie type}, unpublished
lecture notes, 2014.

\bibitem[Ram93]{Ramakrishna}
{\sc R. Ramakrishna}, \textit{On a variation of Mazur's deformation functor},
Compositio Math. \textbf{87} (1993), 269-286.

\bibitem[SZ99]{SchneiderZink}
{\sc P. Schneider and E.-W. Zink},
\textit{$K$-types for the tempered components of a $p$-adic general linear group}, 
J. Reine Angew. Math. \textbf{517} (1999), pp. 161-208.

\bibitem[Ser79]{Serre_Local_fields}
{\sc J.-P. Serre}, \textit{Local Fields}, Graduate Texts in in Mathematics \textbf{67}, Springer-Verlag,
New York-Berlin, 1979.

\bibitem[Ser98]{SerreMorsound}
{\sc J.-P. Serre}, \textit{The notion of complete reducibility in group theory}, Part II of the Moursund lectures, available at \textit{http://math.uoregon.edu/resources/serre}, 1998.

\bibitem[Ser00]{SerreLocalAlg}
{\sc J.-P. Serre}, \textit{Local Algebra}, Springer Monographs in Mathematics, Springer-Verlag, Berlin, 2000.

\bibitem[Shi71]{Shimura}
{\sc G. Shimura}, \textit{Introduction to the arithmetic theory of automorphic functions}, 
Iwanami Shoten Publishers, Tokyo and Princeton University Press, Princeton, NJ, 1971.

\bibitem[Sho15]{Shotton_phd}
{\sc J. Shotton}, \textit{The Breuil-M\'ezard conjecture when $l$ is not equal to $p$}, 
Ph.D. thesis, Imperial College, London, 2015.

\bibitem[Sho16]{Shotton_prep}
{\sc J. Shotton}, \textit{The Breuil-M\'{e}zard conjecture when $l \neq p$}, preprint, 2016.

\bibitem[Spr01]{Sprang}
{\sc J. Sprang}, \textit{A universal deformation ring with unexpected Krull dimension}, Math. Z. \textbf{275} (2013), no. 1-2, pp. 647-652.

\bibitem[Tay08]{Taylor2}
{\sc R. Taylor}, \textit{Automorphy for some {$l$}-adic lifts of automorphic mod {$l$} {G}alois representations. {II}}, Publ. Math. Inst. Hautes \'Etudes Sci. \textbf{(108)}, 2008, pp. 183--239.

\bibitem[Tho12]{Thorne_small}
{\sc J. Thorne}, \textit{On the automorphy of {$l$}-adic {G}alois representations with small residual image}, J. Inst. Math. Jussieu \textbf{(12)}, no. 4, 2012, pp. 855--920.

\bibitem[Til96]{Tilouine}
{\sc J. Tilouine}, \textit{Deformations of Galois representations and Hecke algebras}, 
Publ. Mehta Res. Inst., Narosa Publ., New Delhi, 1996.

\bibitem[Wed08]{Wed08}
{\sc T. Wedhorn}, \textit{The local Langlands correspondence for $\operatorname{GL}(n)$ over
$p$-adic fields}, In: School on Automorphic Forms on $\operatorname{GL}(n)$, (ed. by
L. G\"ottsche, G. Harder, M.S. Raghunathan), ICTP Lecture Notes \textbf{21} (2008), 
pp. 237-320.

\bibitem[Wes04]{Weston}
{\sc T. Weston}, \textit{Unobstructed modular deformation problems}, American J. Math. 
\textbf{126} (2004), pp. 1237-1252.

\bibitem[Yam04]{Yamagami}
{\sc A. Yamagami}, \textit{On the Unobstructedness of the Deformation Problems of
Residual Modular Representations}, Tokyo J. of Math. \textbf{27}, no. 2 (2004), pp. 443--445.

\end{thebibliography}
\end{document}